%% file: BB-CartierCrystals.tex
\documentclass[a4paper,11pt,abstracton=true,headsepline,bibtotoc]{scrartcl}

\RequirePackage{amsmath}                        
\RequirePackage{mathtools}
\RequirePackage{amssymb}                        
\RequirePackage{latexsym}                       
\RequirePackage{stmaryrd}                       

\RequirePackage[all]{xy}                        
\UseTips                                        
\newdir^{c}{{}*!/-3pt/\dir^{(}}
\newdir_{c}{{}*!/-3pt/\dir_{(}}
\newdir^{ (}{{}*!/-3pt/\dir^{(}}
\newdir_{ (}{{}*!/-3pt/\dir_{(}}

\RequirePackage{xspace}                         
\RequirePackage{calc}                           
\RequirePackage{ifthen}                          

\RequirePackage{mathrsfs}
\renewcommand{\mathcal}[1]{\mathscr{#1}}

\RequirePackage[dvipsnames,usenames]{color}  
\RequirePackage{hyperref}
\hypersetup{
bookmarks,
bookmarksdepth=3,
bookmarksopen,
bookmarksnumbered,
pdfstartview=FitH,
colorlinks,backref,hyperindex,
linkcolor=RawSienna,
anchorcolor=BurntOrange,
citecolor=OliveGreen,
filecolor=BlueViolet,
menucolor=Yellow,
urlcolor=OliveGreen
}

\usepackage{colonequals}
\usepackage[german,english]{babel}      
\usepackage[ansinew]{inputenc}
\usepackage[figuresleft]{rotating}
\usepackage[final]{mabliviews}
\usepackage{mabliautoref}

\theoremstyle{plain}   
\numberwithin{theorem}{subsection}

\maketheorem{prop}{Proposition}{theorem}
\maketheorem{prop-def}{Proposition-Definition}{theorem}
\maketheorem{cor}{Corollary}{theorem}
\maketheorem{lem}{Lemma}{theorem}
\maketheorem{conj}{Conjecture}{theorem}

\theoremstyle{definition}    

\maketheorem{defn}{Definition}{theorem}
\maketheorem{claim}{Claim}{theorem}

\theoremstyle{remark}    

\maketheorem{rem}{Remark}{theorem}
\maketheorem{ex}{Example}{theorem}
\maketheorem{ques}{Question}{theorem}
\maketheorem{assumption}{Assumption}{theorem}
\maketheorem{con}{Convention}{theorem}

\long\def\forget#1{}

\numberwithin{equation}{section}

\input{BB-Macros.tex}

\begin{document}

%
%

\title{Cartier Crystals}
\author{Manuel Blickle\thanks{Johannes Gutenberg-Universität Mainz (JGU), Institut für Mathematik, Staudingerweg 9, D-55128~Mainz Germany, manuel.blickle@gmail.com} \and Gebhard Böckle\thanks{Universität Heidelberg, IWR, Im Neuenheimer Feld 368, D-69120~Heidelberg, Germany,  boeckle@uni-hd.de}}

\date{\today}

\maketitle
\tableofcontents
%
%

\input{BB-CartierCrystals_text.tex}

%
%
\bibliographystyle{hep}
\bibliography{BB}
\end{document}

%% file: BB-Macros.tex
\newcommand{\ie}{\textit{i.e.}\ }           
\newcommand{\eg}{\textit{e.g.}\ }           
\newcommand{\cf}{\textit{cf.}\ }

\newcommand{\op}[1]{\operatorname{#1}}
\renewcommand{\cal}[1]{\mathcal{#1}}
\newcommand{\usc}[1][m]{\underline{\phantom{#1}}}

\newcommand{\BA}{{\mathbb{A}}}

\newcommand{\BF}{{\mathbb{F}}}

\newcommand{\BN}{{\mathbb{N}}}

\newcommand{\BZ}{{\mathbb{Z}}}
\newcommand{\Fa}{{\mathfrak{a}}}

\newcommand{\Fm}{{\mathfrak{m}}}

\newcommand{\Fp}{{\mathfrak{p}}}

\newcommand{\Fr}{{\mathfrak{r}}}

\newcommand{\FC}{{\mathfrak{C}}}

\newcommand{\FP}{{\mathfrak{P}}}

\newcommand{\FU}{{\mathfrak{U}}}
\newcommand{\FV}{{\mathfrak{V}}}

\newcommand{\CA}{\cal{A}}

\newcommand{\CC}{\cal{C}}

\newcommand{\CE}{\cal{E}}
\newcommand{\CF}{\cal{F}}
\newcommand{\CG}{\cal{G}}
\newcommand{\CH}{\cal{H}}
\newcommand{\CI}{\cal{I}}
\newcommand{\CJ}{\cal{J}}

\newcommand{\CM}{\cal{M}}
\newcommand{\CN}{\cal{N}}
\newcommand{\CO}{\cal{O}}

\newcommand{\CS}{\cal{S}}
\newcommand{\CT}{\cal{T}}
\newcommand{\CU}{\cal{U}}
\newcommand{\CV}{\cal{V}}
\newcommand{\CW}{\cal{W}}

\renewcommand{\phi}{\varphi}

\renewcommand{\theta}{\vartheta}

\renewcommand{\epsilon}{\varepsilon}

\newcommand{\Ann}{\op{Ann}}         

\newcommand{\coker}{\op{coker}}

\newcommand{\dirlim}{\varinjlim}

\newcommand{\Hom}{\op{Hom}}
\newcommand{\id}{\op{id}}

\newcommand{\image}{\op{im}}

\newcommand{\pr}{\op{pr}}

\newcommand{\Spec}{\op{Spec}}
\newcommand{\Supp}{\op{Supp}}

\newcommand{\Tor}{\op{Tor}}

\newcommand{\tensor}{\otimes}         

\renewcommand{\to}[1][]{\xrightarrow{\ #1\ }}
\newcommand{\ot}[1][]{\xleftarrow{\ #1\ }}
\newcommand{\onto}[1][]{\protect{\xrightarrow{\ #1\ }\hspace{-0.8em}\rightarrow}}
\newcommand{\into}[1][]{\xhookrightarrow{\ #1\ }}
\newcommand{\To}[1][]{\xRightarrow{\ #1\ }}
\newcommand{\oT}[1][]{\xLeftarrow{\ #1\ }}

\newcommand{\defeq}{\stackrel{\scriptscriptstyle \operatorname{def}}{=}}

\newcommand{\kappatilde}{\widetilde{\kappa}}
\newcommand{\Dual}{\operatorname{D}}

\newcommand{\blank}{{\phantom{N}}}
\newcommand{\ublank}{\underline\blank}
\newcommand{\Coeff}{\Lambda}
\newcommand{\CScheme}{C}
\newcommand{\SuppCrys}{\operatorname{Supp_{crys}}}
\renewcommand{\Fr}{\sigma}

\renewcommand{\frm}{\mathfrak{m}}

\renewcommand{\phi}{\varphi}

\newcommand{\OXFr}[1][X]{\CO_{#1}[\Fr]}

\newcommand{\Frid}{(\Fr \times \id)}
\newbox\dottobox
\setbox\dottobox=\hbox{$\longrightarrow$}
\setbox\dottobox=\hbox to\wd\dottobox{\hss$ \xymatrix@C=.5cm{\ar@{.>}[r]&\\}
                                      $\hss}
\newcommand{\dotto}{\mathrel{\copy\dottobox}}
\newbox\leftdottobox
\setbox\leftdottobox=\hbox{$\longleftarrow$}
\setbox\leftdottobox=\hbox to\wd\leftdottobox{\hss$ \xymatrix@C=.5cm{\ar@{<.}[r]&\\}
                                      $\hss}

\newbox\dotintobox
\setbox\dotintobox=\hbox{$\longrightarrow$} \setbox\dotintobox=\hbox
to\wd\dotintobox{\hss$ \xymatrix@C=.5cm{\ar@{^{ (}.>}[r]&\\}
                                      $\hss}

\let\Longto\Longrightarrow

\newcommand{\UCC}{\underline{\CC}\mathstrut}

\newcommand{\UCF}{\underline{\CF}\mathstrut}

\newcommand{\UCH}{\underline{\CH}\mathstrut}
\newcommand{\UCJ}{\underline{\CJ}\mathstrut}

\newcommand{\UCI}{\underline{\CI}\mathstrut}

\newcommand{\UCM}{\underline{\CM}\mathstrut}
\newcommand{\UCN}{\underline{\CN}\mathstrut}
\newcommand{\UCU}{\underline{\CU}\mathstrut}
\newcommand{\UCV}{\underline{\CV}\mathstrut}

\newcommand{\UCW}{\underline{\CW}\mathstrut}

\newcommand{\NilC}{\mathop{\bf Nil}_\kappa\nolimits}
\newcommand{\LNilC}{\mathop{\bf LNil}_\kappa\nolimits}
\newcommand{\QCohC}{\mathop{{\bf QCoh}}\nolimits_\kappa}
\newcommand{\CohC}{\mathop{{\bf Coh}}\nolimits_\kappa}
\newcommand{\QCrysC}{\mathop{{\bf QCrys}}\nolimits_\kappa}
\newcommand{\CrysC}{\mathop{\hbox{\bf Crys}}\nolimits_\kappa}
\newcommand{\IndCrysC}{\mathop{{\bf IndCrys}}\nolimits_\kappa}
\newcommand{\IndCohC}{\mathop{{\bf IndCoh}}\nolimits_\kappa}
\newcommand{\LNilCohC}{\mathop{\bf LNilCoh}_\kappa\nolimits}
\newcommand{\LNilCrysC}{\mathop{{\bf LNilCrys}}\nolimits_\kappa}

\newcommand{\NilT}{\mathop{\bf Nil}_\tau\nolimits}
\newcommand{\LNilT}{\mathop{\bf LNil}_\tau\nolimits}
\newcommand{\QCohT}{\mathop{{\bf QCoh}}\nolimits_\tau}
\newcommand{\CohT}{\mathop{{\bf Coh}}\nolimits_\tau}
\newcommand{\QCrysT}{\mathop{{\bf QCrys}}\nolimits_\tau}
\newcommand{\CrysT}{\mathop{\hbox{\bf Crys}}\nolimits_\tau}
\newcommand{\IndCrysT}{\mathop{{\bf IndCrys}}\nolimits_\tau}
\newcommand{\IndCohT}{\mathop{{\bf IndCoh}}\nolimits_\tau}
\newcommand{\LNilCohT}{\mathop{\bf LNilCoh}_\tau\nolimits}
\newcommand{\LNilCrysT}{\mathop{{\bf LNilCrys}}\nolimits_\tau}

\newcommand{\QCoh}{\mathop{{\bf QCoh}}\nolimits}
\newcommand{\Crys}{\mathop{\hbox{\bf Crys}}\nolimits_\tau}

\newcommand{\QCrys}{\mathop{{\bf QCrys}}\nolimits_\tau}

\newcommand{\Ltensor}{\mathbin{\stackrel{\raisebox{-.5ex}{$\scriptscriptstyle L$}}{\otimes}}}
\newcommand{\kernel}{\op{ker}}
\newcommand{\CExt}{\mathop{%
   \kern-3pt\CE xt\kern3pt}\nolimits}
\newcommand{\CHom}{\mathop{%
   \kern-1pt\CH om\kern2pt}\nolimits}
\newcommand{\CRHom}{\mathop{%
  \kern-1pt R\CH om\kern2pt}\nolimits}

\newcommand{\wt}[1]{\widetilde{#1}}

\newcommand{\smallinv}{{\scriptscriptstyle -1}}

\newcommand{\SLocQNilInv}{\mathop{\CS_{\rm qnil}^{\smallinv}}\nolimits}
\newcommand{\SLocNilInv}{\mathop{\CS_{\rm nil}^{\smallinv}}\nolimits}

\newcommand{\UOnebig}{\underline{\hbox{\rm1\kern-2.5pt l\kern.9pt}}}

\newcommand{\UOnesmall}{\mathop{\hbox{{\footnotesize\underline{{\rm1\kern-2.2pt%
        l\kern.8pt}}}}}}

\newcommand{\RCHom}{\mathop{R\CHom}\nolimits}

\newcommand{\trace}{\op{Tr}}

\newcommand{\ind}{\operatorname{ind}}
\newcommand{\Uomega}{\underline{\omega}}
\newbox\mybox
\def\arrover#1{\mathrel{
       \setbox\mybox=\hbox spread 1.4em{\hfil$\scriptstyle#1$\hfil}
       \vbox{\offinterlineskip\copy\mybox
             \hbox to\wd\mybox{\rightarrowfill}}}}
\def\larrover#1{\mathrel{
       \setbox\mybox=\hbox spread 1.4em{\hfil$\scriptstyle#1$\hfil}
       \vbox{\offinterlineskip\copy\mybox
             \hbox to\wd\mybox{\leftarrowfill}}}}

\def\ontoover#1{\mathrel{
       \setbox\mybox=\hbox spread 1.4em{\hfil$\scriptstyle#1$\hfil}
       \vbox{\offinterlineskip\copy\mybox
             \hbox to\wd\mybox{\rightarrowfill\hskip-2.8mm
                               $\rightarrow$}}}}
\def\leftontoover#1{\mathrel{
       \setbox\mybox=\hbox spread 1.4em{\hfil$\scriptstyle#1$\hfil}
       \vbox{\offinterlineskip\copy\mybox
             \hbox to\wd\mybox{$\leftarrow$\hskip-2.8mm
                               \leftarrowfill}}}}
\def\longto{\longrightarrow}

\newcommand{\D}{{\bf D}}

\newcommand{\coh}{{\rm coh}}
\newcommand{\lncoh}{{\rm lncoh}}

\newcommand{\lnil}{{\rm lnil}}
\newcommand{\lnilcoh}{{\rm lnilcoh}}
\newcommand{\indcoh}{{\rm indcoh}}
\newcommand{\crys}{{\rm crys}}

\newcommand{\indcrys}{{\rm indcrys}}

\newcommand{\can}{{\rm can}}
\newcommand{\Cech}{\check{C}}

\newcommand{\tr}{\operatorname{tr}}
\newcommand{\ctr}{\operatorname{ctr}}
\newcommand{\ad}{\operatorname{ad}}
\newcommand{\cad}{\operatorname{cad}}

\newcommand{\red}{{\rm red}}

%% file: BB-CartierCrystals_text.tex
\section{Introduction}
For a Noetherian scheme $X$ over the finite field $\BF_p$ and finite $\BF_p$-algebra $\Coeff$, the second author and Pink develop in \cite{BoPi.CohomCrys} a cohomological theory of crystals over function fields and apply their theory to give a completely algebraic proof of Goss's rationality conjecture on $L$-functions of families of $A$-motives (after Anderson). In fact \cite{BoPi.CohomCrys} accomplishes much more in that it introduces an abelian category of $\tau$-crystals, containing families of $A$-motives as a special case, and develops a cohomology theory for them. Roughly speaking a $\tau$-sheaf on $X$ (over $\Coeff$) is a coherent $\CO_{X \times \Coeff}$-module $\CF$ together with a structural map $\tau \colon \CF \to \Frid_*\CF$ where $\sigma: X \to X$ denotes the absolute $q$-power Frobenius on $X$. The $\tau$-sheaves on $X$ form an abelian category, from which the category of $\tau$-crystals is obtained by localizing at the Serre subcategory of nilpotent $\tau$-sheaves.

Another crucial result in \cite{BoPi.CohomCrys} is that, if $\Coeff$ is finite, then the category of $\tau$-crystals is equivalent to the category of constructible sheaves of $\Coeff$-modules on the \'etale site of $X$. This may be viewed as a vast generalization of Artin-Schreier theory and has been anticipated by Deligne in \cite{SGA4.2}

A similar correspondence was developed in \cite{EmKis.Fcrys} by Emerton and Kisin. For a \emph{smooth} scheme of finite type over a field of positive characteristic they consider quasi-coherent $\mathcal{O}_X$-modules $\CM$ with a unit Frobenius semi-linear action. The unit condition means that the linearization $\Fr^* \CM \to \CM$ of $\Fr_{\CM} \colon \CM \to \Fr_* \CM$, given by sending $r \tensor m$ to $r\cdot \Fr_{\CM}(m)$, is an isomorphism. This implies that $\Fr_{\CM}$ is injective, hence $\CM$ has no $\Fr$--nilpotent elements. In addition the unit condition implies that $\CM$ is closed under certain nilpotent extensions. Together, the unit condition can be regarded as a rigidification of the theory of quasi-coherent $\CO_X[\Fr]$--modules. On the level of the derived categories \cite{EmKis.Fcrys} prove a contravariant correspondence between locally finitely generated unit $\mathcal{O}_X$-modules and constructible \'etale $\BF_p$--sheaves. As above they also show their correspondence with a finite coefficient ring $\Coeff$.

The theories in \cite{BoPi.CohomCrys} and \cite{EmKis.Fcrys} may hence be viewed as different approaches to compute \'etale cohomology of $\BF_p$-sheaves in positive characteristic. Each approach is tailored to certain applications and has its own benefits. The present paper grew out of the desire to establish a direct link between these theories for arbitrary essentially finite type coefficients $\Coeff$. The key to such a correspondence ought to be furnished by Grothendieck-Serre duality. It became quickly clear that an important intermediate category would have to be the category dual to the category of $\tau$-crystals. 
The following diagram summarizes the situation
\[
   \xymatrix@C+2pc{
   \{\tau\text{--crystals}\} \ar@{->}[rr]_\cong^{\text{\cite{BoPi.CohomCrys}}}\ar@{<->}[d]_\simeq^{\text{G-S-duality}} && \{\text{constr. }\Coeff\text{--mod on }X_{et}\} \\
   \{\kappa\text{--crystals}\} \ar@{<-}[r]^{\tensor \omega_X}_{\cong} &\{\gamma\text{--crystals}\} \ar@{->}[r]^-{}_-\cong&\{\text{f.g. unit }\mathcal{O}_{X,\sigma}\text{--mod}\}\ar@{<->}[u]^\simeq_{\text{\cite{EmKis.Fcrys}}}
}
\]
What we accomplish in this article (and its sequel \cite{BliBoe.CartierDuality}) is a complete description of the left part of this diagram. That is, for any Noetherian $\Fr$-finite scheme $X$ we introduce the category of $\kappa$-crystals which under Grothendieck-Serre duality corresponds to $\tau$-crystals of \cite{BoPi.CohomCrys}. We also define cohomological operations for $\kappa$-crystals and show in \cite{BliBoe.CartierDuality} that these correspond, via Grothendieck-Serre duality, to the ones defined in \cite{BoPi.CohomCrys}.

The bottom row in the above diagram only holds if $X$ is regular and was treated in a basic form already in \cite{BliBoe.CartierFiniteness} and the intermediate category of $\gamma$-sheaves was thoroughly studied in \cite{blickle_minimal_2008}; the full equivalence in the bottom row, including the compatibility with the cohomological operations on either side, will be treated in another sequel \cite{BliBoe.CartierEmKis} to the present paper.  The maps to the top right corner are results in \cite{BoPi.CohomCrys} and \cite{EmKis.Fcrys}, respectively, and only yield an equivalence for finite coefficient rings $\Coeff$. Note that the horizontal equivalences are covariant equivalences on the underlying abelian categories, whereas the vertical ones are contravariant {\em and} require the passage to the derived category. Furthermore, the theory in the top left corner applies to any Noetherian scheme $X$, whereas the theory in the bottom right corner is only applicable to smooth schemes of finite type over a fixed base field.

During the work on this project it became clear that the category of $\kappa$-crystals has very interesting properties and applications: implicitly in the work of Gabber \cite{Gabber.tStruc} and and explicitly in Anderson \cite{AndersonL} $\kappa$-sheaves are a central concept. The trace formula in the monograph \cite{BoPi.CohomCrys} makes crucial use of \cite{AndersonL} and hence in an indirect way relies on $\kappa$-sheaves as well.
This lead us to establish in \cite{BliBoe.CartierFiniteness} the following important structural result for Cartier crystals in the case that $\Coeff = \BF_q$, \cf \cite[Sec 4]{BliBoe.CartierEmKis}

\begin{theorem*}
Let $X$ be a locally Noetherian scheme over $\mathbb{F}_q$ such that the Frobenius $\sigma$ on $X$ is a finite map. Then within the category of Cartier crystals the following holds.
\begin{enumerate}
\item Coherent Cartier crystals have finite length.
\item The $\Hom$ sets of coherent Cartier crystals are finite dimensional $\mathbb{F}_q$-vetorspaces.
\end{enumerate}
By (a) and (b) coherent Cartier crystals have only finitely many sub-crystals.
\end{theorem*}
In \cite{BliBoe.CartierFiniteness} we derive a number of applications of the above theorem to local cohomology and to Lyubeznik's theory of $F$-finite modules \cite{Lyub}. The latter is a local version of Emerton-Kisin's theory of finitely generated unit $\OXFr$-modules.

Other applications are visible in the work of the second author in \cite{BlickleTestIdealsViaAlgebras} and of Schwede in \cite{schwede_test_2009}. Both works show that the theory of Cartier crystals is related to the theory of test ideals \cite{HaraTakagi.TestIdeal,BliMusSmi.DisRat}; the latter plays a crucial role in the birational geometry of algebraic varieties in positive characteristic. We expect that the formalism, and in particular the cohomological operations, which are developed in the present paper will be useful for the study of test ideals. For example, in some cases, one can realize the test ideal within the category of Cartier crystals as an intermediate extension functor. This functorial viewpoint on test ideals may help to give a conceptual explanation for the transformation rules of test ideals which were obtained recently \cite{SchwedeTuckerTestIdealFiniteMaps,BliSchweTuck.Alterations}. Finally one may speculate, that Cartier Crystals play for test ideals a similar role as $D$-modules play for multiplier ideals in characteristic zero. In characteristic zero, jumping numbers of multiplier ideals have an interpretation as the jumps in the $V$ filtration of a certain $D$-modules \cite{BudurSaito}.  Perhaps in positive characteristic jumping numbers of test ideals have an analogous interpretation in terms of Cartier crystals, which could be related to constructible $p$-torsion sheaves via our duality \cite{BliBoe.CartierDuality} and the equivalence in \cite{BoPi.CohomCrys}.

\subsection{Summary of main results}
Throughout we assume that our schemes are Noetherian and separated over $\BF_q$, the field of $p$ elements. Except when explicitly stated otherwise, we assume throughout the paper that the $q$-power Frobenius $\Fr$ is a finite map. Let $\Coeff$ be an essentially finite type $\BF_q$-algebra which will serve for us as a coefficient ring. We denote $\CScheme=\Spec \Coeff$.

The purpose of this paper is to introduce various categories of $\kappa$-crystals and to develop a theory of cohomological operations for them. As indicated above, we define (quasi-) coherent $\kappa$-sheaves on $X$ over $\Coeff$ as (quasi-) coherent $\CO_{X \times \CScheme}$-modules $\CM$ equipped with a $\CO_{X\times C}$-linear map
\[
\kappa_\CM \colon \Frid_* \CM \to \CM
\]
and denote the category of such by $\CohC(X,\Coeff)$ (respectively, $\QCohC(X,\Coeff)$). The coherent $\kappa$-sheaves form an abelian category which contains the Serre subcategory of nilpotent $\kappa$-sheaves $\NilC(X,\Coeff)$, consisting of those $\CM$ such that some power of $\kappa_\CM$ is zero. The category of $\kappa$-crystals $\CrysC(X,\Coeff)$ is then the localization of $\CohC(X,\Coeff)$ at its subcategory $\NilC(X,\Coeff)$. In order to define functors it is necessary to work in the large ambient category of $\kappa$ quasi-crystals which is defined analogously as the localization of $\QCohC(X,\Coeff)$ at the locally nilpotent (\ie inductive limits of nilpotent) $\kappa$-sheaves $\LNilC(X,\Coeff)$. One of our primary goals is to develop functorial operations on $\CrysC(X,\Coeff)$. These categories, and some more, are described in detail in \autoref{sec.CartierCrys}.

Duality for a finite map applied to the Frobenius endomorphism $\sigma$ says, that to give a $\kappa$-structure on a quasi-coherent $\CO_X$-module $\CM$, \ie a morphism $\kappa \colon \Frid_*\CM \to \CM$, is equivalent to giving a morphism $\kappatilde \colon \CM \to \Frid^! \CM$, where
\[
    \Frid^! \CM = \Frid^{-1}\RCHom_{X\times \CScheme}(\Frid_*\CO_{X \times \Coeff},\CM)
\]
is viewed as an object in $\D^+_\coh(\QCoh(X \times \Coeff))$. These observations suggest that the functors $Rf_*$ and $f^!$ for an essentially finite type morphism $f \colon Y \to X$ known from Grothendieck-Serre duality, induce functors on an appropriate derived category of $\kappa$-crystals. The theory developed in \autoref{CSh-DefFunc} culminates in the following results.

For the push-forward functor on $\kappa$-sheaves in \autoref{GenDefDerIm} we shall prove:
\begin{theorem*}
If $f \colon Y \to X$ is an arbitrary morphism, then the usual push-forward functor $R(f \times \id)_*$ on quasi-coherent sheaves induces an exact $\Coeff$-linear functor
\[
    Rf_* \colon \D^*(\QCrysC(Y,\Coeff)) \to \D^*(\QCrysC(X,\Coeff))
\]
for $* \in \{b,+,-,\emptyset\}$. Furthermore if $f$ is of finite type then this restricts to a functor
\[
    Rf_* \colon \D_\crys^*(\QCrysC(Y,\Coeff)) \to \D_\crys^*(\QCrysC(X,\Coeff))\, .
\]
Here $\D_\crys^*(\QCrysC(\usc,\Coeff))$ denotes the derived category of $\kappa$ quasi-crystals whose cohomology is locally nil coherent (\ie extension of locally nilpotent by coherent) with boundedness condition $*$.
\end{theorem*}
The crucial point to note in this result is that we \emph{do not} assume that $f$ is proper in the second statement. This is in stark contrast to the classical situation where $f_*$ generally does not preserve coherence. What we show, going back to a technique of Anderson in \cite{AndersonL}, is that if $f$ is of finite type, then $Rf_*$ preserves coherence ``up to local nilpotence'', \ie that $Rf_*$ of a coherent $\kappa$-sheaf is an extension of a locally nilpotent by a coherent $\kappa$-sheaf.

The main results for the pull-back for $\kappa$-sheaves are Theorems~\ref{DF-ShriekOnD-Crys} and \ref{CohomDimOfShriek}.
\begin{theorem*}
If $f \colon Y \to X$ is essentially of finite type, then the twisted inverse image functor $(f \times \id)^!$ on quasi-coherent sheaves induces an exact $\Coeff$-linear functor
\[
    f^! \colon \D^{+}_\crys(\QCrysC(X,\Coeff)) \to \D^{+}_\crys(\QCrysC(Y,\Coeff))\, .
\]
of bounded cohomological dimension.

In particular, $f^!$ restricts to a functor on $\D^{b}_\crys(\QCrysC(\usc,\Coeff))$.
\end{theorem*}
At first glance, the boundedness of the cohomological dimension may seem unexpected. It is motivated by the duality between $\tau$- and $\kappa$-sheaves which is treated in \cite{BliBoe.CartierDuality}, which is of finite cohomological dimension, and the fact that the corresponding functor $f^*$ is exact on coherent $\tau$-crystals.

The proof of the above result is based on properties of the functor $f^!$ from Grothen\-dieck-Serre-Duality in its extension by Lipman \cite{LipmanGrothDual} and Nayak \cite{Na.CompEFT}. One difficulty lies in showing that the functor $f^!$ defined on $\kappa$-sheaves preserves nilpotence and hence induces a functor on crystals; cf.~\autoref{CSh-FShriekNilpotenceForEssFinite}. We give two proofs for this crucial statement. One uses Nayak's extension of Nagata compactification to essentially finite type morphisms \cite{Na.CompEFT}, another one uses a notion of stalks for $\kappa$-crystals that we develop in \autoref{sec.Stalks}. Finally, the boundedness of the cohomological dimension is based on the devissage theorem stated next that allows one to decompose $X$ and $Y$ into a disjoint union of locally closed regular subschemes.

As explained above, the two functors $Rf_*$ and $f^!$ we define for $\kappa$-crystals are directly induced from the corresponding functors on the underlying quasi-coherent sheaves. This occupies most of \autoref{sec.CartierCrys} and \autoref{CSh-DefFunc}. This approach has the advantage that many of the compatibilities among these functors follow, often rather directly, from a corresponding well-established statement for the functors on the underlying categories of quasi-coherent sheaves. Some more work is necessary to show the following adjointness result from \autoref{DF-Crys-Adj}.
\begin{prop*}
\begin{enumerate}
\item Let $f \colon Y \to X$ be a proper morphism. Then as functors on categories $\D^+_\crys(\QCrysC(\usc,\Coeff))$ the functor $Rf_*$ is naturally left adjoint to $f^!$.
\item If $j \colon Y \to X$ is an open immersion, then $j^!=j^*$ is naturally left adjoint to $Rj_*$.
\end{enumerate}
\end{prop*}

For an open immersion $j\colon U\into X$ and a closed complement $i\colon Z\into X$ the above adjunction yields natural morphisms $i_*i^!\to\id$ and $\id\to j_*j^*$. In \autoref{ExactTriangleThm} we prove the following technically important result regarding their combination:
\begin{theorem*}
In $\D^+_?(\QCrysC(X,\Coeff))$, for $?\in\{\emptyset,\crys\}$, there is a natural exact triangle
\[
    i_*i^!\to \id \to Rj_*j^* \to i_*i^![1]
\]
\end{theorem*}
This can be used to derive a Kashiwara-type equivalence, roughly saying that for a closed imbedding $i\colon Y \into X$, the functors $i^!$ and $i_*$ are inverse equivalences between the derived category of Cartier crystals on $Y$ and the derived category of Cartier crystals on $X$ which are supported on $Y$; the precise statement can be found in \autoref{thm.DerivedKashiwara}. On the level of the underlying abelian categories and with $\Coeff=\mathbb{F}_p$, this has been shown already in \cite{BliBoe.CartierFiniteness}. The main application we give of the above triangle is its role in the proof of the boundedness for $f^!$, where it allows a devissage argument to reduce to the case of maps between smooth schemes.

We conclude the paper by studying a notion of stalks and their relation to the notion of crystalline support. For a point $i_x \colon x \into X$ of $X$ and a Cartier sheaf $\UCV$ on $X$ we define the stalk of $\UCV$ at $X$ to be $i_x^!\UCV$. A key point is that these stalks can be used to check nilpotence; in \autoref{CSh-LocalNilForIndonStalks} we show that $\UCV$ is locally nilpotent if and only if $i_x^!\UCV$ is cohomologically nilpotent for all (not necessarily closed, if $X$ is not of finite type over a field) points $x \in X$. This pointwise criterion for local nilpotence can in turn be used to give an alternative proof of the fact that $f^!$ preserves cohomological nilpotence by reducing the statement to the case of a map between fields, where it can checked by hand. We then show that the closure of the set of all $x \in X$ such that $i_x^! \UCV$ is not cohomologically nilpotent is in fact a suitable notion of support for the crystal associated to $\UCV$. It is the smallest closed  set $Z$ such that $\UCV|_{X-Z}$ is locally nilpotent, and hence zero as a crystal.

\subsection{Notation and conventions}\label{sec.notation}
For the remainder of this article, we fix the following notation: $p$ is a prime number and $q$ is a power $p^e$ of $p$. 
All schemes are assumed to be Noetherian and separated over the prime field~$\BF_q$. All morphisms and fiber products, and all tensor products of modules and algebras, are taken over~$\BF_q$, unless otherwise specified.

By $X$, $Y$, etc., we denote Noetherian schemes over the field $\BF_q$.  The $q$-power Frobenius morphism on $X$, which acts on functions by $f\mapsto f^q$, is denoted $\Fr_X \colon X\to X$. Mostly we simply write $\Fr$ for~$\Fr_X$.

Except when explicitly said otherwise, we assume throughout that our schemes satisfy the condition that
\begin{center}
\emph{the Frobenius endomorphism $\Fr_X$ is a finite map of schemes.}
\end{center}
This is a crucial assumption in our development our theory of Cartier crystals. Only in \autoref{sec.CartierCrys} we occasionally drop this assumption and point out if a result or construction holds more generally.

As pointed out by Gabber in \cite[Remark 13.6]{Gabber.tStruc}, the finiteness of the Frobenius on any Noetherian separated scheme $X$ implies that $X$ admits a dualizing complex; a proof is given in op.cit.\ in the case that $X$ is affine. Both conditions are clearly satisfied if $X$ is essentially of finite type over a $\Fr$-finite field~$k$, which is our prime case of interest.

We also fix an $\BF_q$-algebra $\Coeff$ which is essentially of finite type over $\BF_q$. By $\CScheme$ we denote the corresponding affine scheme $\Spec\Coeff$. The role of $\Coeff$ will be that of a coefficient ring. Examples of particular interest are that $\Coeff$ is a regular finite type $\BF_q$-algebra or that $\Coeff$ is a local Artin ring. The assumptions on $\CScheme$ imply that $X\times\CScheme$ is again Noetherian for every Noetherian scheme~$X$ over~$k$. This is useful in dealing with coherent sheaves on $X\times \CScheme$.

\begin{rem}
  After the predecessor to this paper \cite{BliBoe.CartierFiniteness} was written the authors became aware of a notational conflict with the notion of a Cartier module in the theory of commutative formal groups which was originally introduced by Manin \cite{Manin.CommutativeFormal}, \cf \cite{Zink.CartierTheory}. We want to stress that our notion is different from theirs and hope that this does not cause confusion for the readers.
\end{rem}

\paragraph{Essentially finite type morphisms.}

Let us recall some classes of morphisms that in the classical literature on algebraic geometry have not been considered much, and where only recently foundational results have been obtained, see~\cite{Na.CompEFT}. These are useful in the formulation and application of Grothendieck-Serre duality.

Let $\phi\colon A\to B$ be a ring homomorphism. If  $B$ is a localization of $A$ and $\phi$ is the canonical localization map then $\phi$ is called a {\em localizing homomorphism}. If $\phi$ factors as $A\to C\to B$ with $A\to C$ of finite type and $C\to B$ localizing, then $\phi$ is called {\em essentially of finite type}.

\begin{defn}\label{sec.EssFT}
Let $f\colon Y \to X$ be a morphism  of Noetherian schemes.
\begin{enumerate}
\item The map $f$ is said to be {\em essentially of finite type} if every point $y\in Y$ has an affine open neighborhood $V= \Spec A$ such that $f^{-1}V\to V$ is covered by finitely many affine open $U_i =\Spec(B_i)$ for which the corresponding ring homomorphisms $\phi_i \colon A\to B_i$ are essentially of finite type.
\item If in (a) each $\phi_i$ is moreover a localizing homomorphism, then one calls $f$ is a {\em localizing morphism}.
\item If $f$ is a localizing morphism and set-theoretically injective (or equivalently separated and injective on generic points) it is called {\em a localizing immersion}.
\end{enumerate}
\end{defn}
In many respects localizing immersions behave like open immersions; any finite type localizing immersion is an open immersion. Each, localizing morphism, localizing immersion and morphism essentially of finite type is preserved under composition and base change. An important example of a localizing immersion is the natural map $\Spec\CO_{Y,y}\to Y$. A natural extension of Nagata's compactification theorem of morphisms of finite type is the following factorization theorem of Nayak \cite[Thm.~4.1]{Na.CompEFT} for essentially finite type morphisms.
\begin{theorem}\label{Mor-FactOfEFT}
Let $f\colon Y\to X$ be a separated essentially finite type morphism of Noetherian schemes. Then $f$ factors as $Y\to[k] Z \to[p] X$ where $k$ is a localizing immersion and $p$ is proper.
\end{theorem}
It follows from \cite[Lem.~3.2]{Na.CompEFT} that the set of such factorizations forms a filtering inverse system, \ie any two such are dominated by a third.

An {\em essentially smooth (essentially \'etale) morphism} of Noetherian schemes is a separated formally smooth (formally \'etale) morphism that is essentially of finite type. Localizing morphisms are  essentially \'etale but the latter class of morphisms is larger.

\subsection{Acknowledgements}
We thank Joe Lipman and Suresh Nayak for answering our questions about some of the more intricate compatibilities of Grothendieck-Serre duality. Their clarifications were crucial during the final stages of the preparation of this manuscript. We are also grateful to Axel Stäbler for a careful reading of an earlier version of this manuscript pointing out numerous typos and inaccuracies.

Both authors were supported by the DFG funded research grant SFB/TRR45 \emph{Periods, moduli, and the arithmetic of algebraic varieties}. The first author was supported by his Heisenberg fellowship.

\section{Cartier sheaves and Cartier crystals}\label{sec.CartierCrys}
In this section we give the definition of Cartier sheaves and Cartier crystals and set up the basic theory. For a more gentle introduction in the case that the coefficient algebra $\Coeff$ is just $\mathbb{F}_p$, we refer to \cite{BliBoe.CartierFiniteness}.

\begin{defn}\label{CSh-Def}
A \emph{Cartier sheaf on $X$ over~$\Coeff$}, or shorter, a \emph{$\kappa$-sheaf
on $X$} is a pair $\UCV = (\CV,\kappa_{\CV})$ consisting of
a quasi-coherent sheaf $\CV$ on $X\times \CScheme$ and an $\CO_{X\times
\CScheme}$-linear homomorphism $\kappa_{\CV}\! \colon \Frid_*\CV \longto \CV$.

A \emph{homomorphism} of $\kappa$-sheaves $\UCV\to\UCW$ on~$X$ is a
homomorphism of the underlying quasi-coherent sheaves $\phi\! \colon \CV\to\CW$ for which
the following diagram commutes:
\[ \xymatrix{
\Frid_*\CV \ar[r]^-{\kappa_{\CV}}
    \ar[d]_{\Frid_*\phi}&
\CV \ar[d]^\phi \\
\Frid_*\CW \ar[r]^-{\kappa_{\CW}} & \CW \rlap{.}\\ } \]
\end{defn}

The sheaf underlying a $\kappa$-sheaf $\UCW$ will always be
denoted~$\CW$, unless otherwise specified.  To avoid cumbersome
notation we will often abbreviate $\kappa = \kappa_{\CW}$ if the
underlying sheaf is clear from the context.

\begin{ex}
If $k$ is a perfect field, then the $p$th root morphism provides a canonical Cartier structure on $k$.

If $k$ is a field such that $(k \colon k^p)$ is finite, then $k$ and $\sigma^!k= \sigma^{-1}\Hom_k(\sigma_* k,k)$ are abstractly isomorphic fields. Any choice of such isomorphism equips $k$ with the structure of a Cartier module, but, unless $k$ is perfect, there is no canonical such choice.
\end{ex}

\begin{ex}\label{ex:CartierOmega}
The prototypical example of a Cartier sheaf (and namesake thereof) is the classical Cartier operator on the dualizing sheaf $\omega_X=\Omega^n_X$ of a smooth $n$-dimensional scheme $X$ of finite type over $\mathbb{F}_p$. It is induced by the Cartier isomorphism (see, for example \cite{EsnVieh.Vanishing})
\[
    H^i(\Fr_*\Omega^\bullet_X) \to[\cong] \Omega^i_X,
\]
which yields after composition with the projection $\Fr_*\Omega^n_X \onto H^n(\Fr_*\Omega^\bullet_X)$ the map
\[
    \kappa=\kappa_X \colon \Fr_* \omega_X \to \omega_X.
\]
In local coordinates $x_1,\ldots,x_n$ it is given by
\[
    x_1^{i_1}\cdots x_n^{i_n}dx_1\wedge \ldots \wedge dx_n \mapsto x_1^{\frac{i_1+1}{p}-1}\cdots x_n^{\frac{i_n+1}{p}-1}dx_1\wedge \ldots \wedge dx_n
\]
where a non-integral exponent renders the whole expression on the right as zero. One verifies that this construction does not depend on the choice of local coordinates. An explicit computation shows that the adjoint of this map $\widetilde{\kappa} \colon \omega_X \to \sigma^!\omega_X$ is an isomorphism.
\end{ex}

\begin{ex}
More generally, if $X$ smooth over a $\sigma$-finite field $k$, then -- after a Cartier structure on $\kappa_k \colon \sigma_{k*} k \to k$ was picked -- by the same formula as above there is a Cartier structure on $\omega_X$ given in local coordinates by having $\kappa_k$ also act on the coefficients from $k$. In particular, if $k$ is perfect, there is a canonical Cartier structure on $\omega_X$ for any finite type smooth scheme $X$ over $k$.

Even more generally, if $X$ is normal and of finite type over $k$, then the Cartier structure on the smooth locus $i \colon U \into X$ on $\omega_U$ induces a Cartier structure on the dualizing sheaf $\omega_X \colonequals i_*\omega_U$ canonically. Also in this case, the adjoint $\omega_X \to \sigma^! \omega_X$ is an isomorphism.
\end{ex}

\begin{defn}\label{CSh-DefCoh}\label{CSh-CatDef}
\begin{enumerate}
\item A  $\kappa$-sheaf $\UCV$ is called \emph{coherent} if its
  underlying sheaf is coherent.
\item A $\kappa$-sheaf $\UCV$ is called \emph{ind-coherent} if it is
  the union (or equivalently the filtered direct limit) of coherent $\kappa$-sheaves.
\end{enumerate}
The category formed by all $\kappa$-sheaves over~$\Coeff$ on~$X$ and with the
above homomorphisms is denoted $\QCohC(X,\Coeff)$.  The full subcategory of
all coherent $\kappa$-sheaves is denoted $\CohC(X,\Coeff)$, that of all
ind-coherent $\kappa$-sheaves $\IndCohC(X,\Coeff)$.
\end{defn}
We begin with a simple, but useful characterization of ind-coherent $\kappa$-sheaves.
\begin{lem}
\label{Csh-IndCohCharacterization}
A $\kappa$-sheaf $\UCV$ is ind-coherent if and only if for every coherent $\CO_{X\times \Coeff}$-submodule $\CU \subseteq \UCV$, the $\kappa$-subsheaf of $\UCV$ generated by $\CU$ is coherent.
\end{lem}
\begin{proof}
If $\UCV$ is ind-coherent, then it is the union of its coherent $\kappa$-sub-sheaves. Since $\CU \subseteq \UCV$ is coherent, it must hence be contained in some coherent $\kappa$-sheaf. It follows that the $\kappa$-subsheaf $\UCU \colonequals  \sum_i \kappa^i \CU$ generated by $\CU$ must also be contained in a coherent $\kappa$-sheaf which implies that $\UCU$ itself is coherent since $X$ is Noetherian.

Conversely, as a quasi-coherent $\CO_{X \times \Coeff}$-module, $\UCV$ is the union of its coherent submodules. If the $\kappa$-subsheaf generated by each of these is also coherent, then, clearly $\UCV$ is the union of coherent $\kappa$-sheaves, hence $\UCV$ is ind-coherent.
\end{proof}

\begin{rem}\label{CSh-Proper}
There exist ind-coherent $\kappa$-sheaves which are not coherent,
e.g. any infinite direct sum of non-zero coherent $\kappa$-sheaves.

More interestingly, there also exist quasi-coherent $\kappa$-sheaves  which are not
ind-coherent: Take $k=\BF_p$, $R=k[x]$, $X=\Spec R$, $\Coeff=k$, and let $\UCV$
be the $\kappa$-sheaf on $X$ corresponding to the $R$-module
$M=\cup_{n\in\BN} k[x^{p^{-n}}]$ with $\kappa$ the inverse of
$\Fr$ which is an isomorphism on $M$. Then
$\UCV$ contains no non-zero coherent $\kappa$-subsheaf.
\end{rem}

Clearly, the categories $\CohC(X,\Coeff)$, $\QCohC(X,\Coeff)$, and $\IndCohC(X,\Coeff)$ are abeli\-an $\Coeff$-linear categories, and, since $\Fr$ is affine and hence $\Fr_*$ is exact, constructions such as kernel, cokernel, etc.\ are the usual ones on the underlying
quasi-coherent sheaves, with the respective structural map $\kappa$ obtained by
functoriality.  In particular, the formation of kernel, cokernel,
image and co-image is preserved under the inclusions $\CohC(X,\Coeff)\subset
\IndCohC(X,\Coeff)\subset \QCohC(X,\Coeff)$.  From this it follows immediately, that $\CohC(X,\Coeff)$ is a Serre subcategory of $\QCohC(X,\Coeff)$, \ie~it is an abelian subcategory that is closed under extensions. This statement holds for $\IndCohC(X,\Coeff)$ only under the additional assumption that the Frobenius $\Fr$ on $X$ is finite:

\begin{prop}\label{CSh-Serre}
$\CohC(X,\Coeff)$ is a Serre sub-category of $\QCohC(X,\Coeff)$. If the Frobenius $\Fr$ is finite on $X$, then $\IndCohC(X,\Coeff)$ is also a Serre sub-category of $\QCohC(X,\Coeff)$.
\end{prop}
\begin{proof}
The only point that is not obvious is that the extension of ind-coherent $\kappa$-sheaves is again ind-coherent. For this, consider a short exact sequence
\[
    0 \to \UCV' \to \UCV \to[\pi] \UCV'' \to 0
\]
of $\kappa$-sheaves, assuming that the ends are ind-coherent. It is enough to show that given a coherent $\kappa$-subsheaf $\UCW''$ of $\UCV''$, there is a coherent $\kappa$-subsheaf $\UCW$ of $\UCV$ that surjects onto $\UCW''$. Let $\CW$ be any coherent subsheaf of $\CV$ that surjects onto $\UCW''$. Note that $\pi(\kappa(\Fr \times \id)_*\CW) = \kappa(\Frid_*(\pi(\CW))) \subseteq \kappa(\Frid_* \UCW'') \subseteq \UCW''$. Hence, $\pi$ induces a surjection
\[
    \pi \colon (\CW+\kappa(\Frid_*\CW)) \to \UCW''.
\]
Denoting by $\CH$ the kernel of this map we have $\kappa( \Frid_* \CW) \subseteq \CW+\CH$. Now, since $\Fr$, and hence $\Frid$ is a finite map by assumption, $\Frid_*\CW$ is coherent. This implies that $\CH$ is a coherent subsheaf of $\UCV'$. Since the latter is assumed to be ind-coherent, $\CH$ is contained in some coherent $\kappa$-subsheaf. After enlarging $\CH$ -- which does not affect the inclusion $\kappa( \Frid_* \CW) \subseteq \CW+\CH$ -- we may hence assume that $\CH$ itself is a coherent $\kappa$-subsheaf of $\UCV'$. Then we have
\[
    \kappa(\Frid_*(\CW+\CH)) = \kappa(\Frid_*\CW)+\kappa(\Frid_*(\CH)) \subseteq \CW+\CH+\CH \subseteq \CW+\CH
\]
showing that $\UCW \colonequals  \CW+\CH$ is a coherent $\kappa$-subsheaf of $\UCV$ mapping epimorphically to $\UCW''$.
\end{proof}
An example for the necessity of the assumption that $\Fr$ is finite in the above proposition is given below in \autoref{CSh-ExampleNotSerreSub}.

Below it will be crucial to know that $\QCohC(X,\Coeff)$ as well as $\IndCohC(X,\Coeff)$ are Grothendieck categories, \ie~they form an abelian category with exact filtered direct limits and a generator.

\begin{theorem}\label{CSh-QisGrothCat}
$\QCohC(X,\Coeff)$ is a Grothendieck category. Its subcategory of ind-coherent objects
$\IndCohC(X,\Coeff)$ is closed under filtered direct limits in $\QCohC(X,\Coeff)$ and is itself a Grothendieck category.
\end{theorem}
\begin{proof}
The proof is almost the same as that of \cite[Thm.~3.2.7]{BoPi.CohomCrys}. Only in showing that $\QCohC(X,\Coeff)$ has a generator is slightly different (in the case that $\Fr$ is not finite). For any cardinal $\lambda$ we say that a $\kappa$-sheaf $\UCV$ is \emph{$\lambda$-coherent} if its underlying sheaf $\CV$ is locally generated by a set of cardinality $\leq \lambda$. Let $\lambda$ be an infinite cardinal such that $\Fr_*\CO_X$ is locally generated by a set of cardinality $\leq \lambda$. Now the proof proceeds as that of \cite[Thm.~3.2.7]{BoPi.CohomCrys} replacing $\omega$ by $\lambda$ throughout.
\end{proof}

\subsection{Nilpotence}
\label{CSh-Nilpotence}
For any $\kappa$-sheaf $\UCV$ we define the iterates $\kappa_{\CV}^n$ of $\kappa_{\CV}$ by setting inductively $\kappa_{\CV}^{0} \colonequals  \id$ and $\kappa_{\CV}^{n+1} \colonequals  \kappa_{\CV} \circ \Frid_*\kappa_{\CV}^n$. Thus they are $\CO_{X\times \Coeff}$-linear homomorphisms
\[
    \kappa_{\CV}^n \colon (\Fr^n\times\id)_*\CV \longto \CV.
\]
Each $\Fr^n_*\UCV \colonequals  \bigl( (\Fr^n\times\id)_*\CV, (\Fr^n\times\id)_*\kappa_{\CV} \bigr)$ is a $\kappa$-sheaf in its own right, and $\kappa_{\CV}^n$ is a homomorphism of $\kappa$-sheaves $\Fr^n_*\UCV \to \UCV$. In particular, its image is a $\kappa$-subsheaf of~$\UCV$. Note that if $\Fr$ is finite then the $\kappa$-sheaves $\Fr^n_*\UCV$ are coherent, whenever $\UCV$ is so.

\begin{defn}\label{CSh-NilDef}
\begin{enumerate}
\item A $\kappa$-sheaf $\UCV$ is called \emph{nilpotent} if
$\kappa^n_{\CV}$ vanishes for some, or equivalently all, $n\gg0$.
\item A $\kappa$-sheaf $\UCV$ is called \emph{locally nilpotent} if
it is a union of nilpotent $\kappa$-subsheaves.
\item By $\LNilC(X,\Coeff)$ we denote the full subcategory of
  $\QCohC(X,\Coeff)$ formed by all locally nil\-po\-tent
  $\kappa$-sheaves, by $\NilC(X,\Coeff)$ we shall denote the intersection
  $\LNilC(X,\Coeff)\cap \CohC(X,\Coeff)$.
\item By $\LNilCohC(X,\Coeff)$ we denote the full subcategory of $\IndCohC(X,\Coeff)$ of objects $\UCM $ which contain a coherent $\UCC \subseteq \UCM$ such that the quotient $\UCM/\UCC$ lies in $\LNilC(X,\Coeff)$. We will see shortly that $\LNilCohC(X,\Coeff)$ is the Serre subcategory of $\IndCohC(X,\Coeff)$ generated by $\CohC(X,\Coeff)$ and $\LNilC(X,\Coeff)$.
\end{enumerate}
\end{defn}
It follows easily from the noetherianess of objects in $\CohC(X,\Coeff)$ that any object in $\NilC(X,\Coeff)$ is nilpotent. However, there are nilpotent $\kappa$-sheaves which are not coherent, and hence not in $\NilC(X,\Coeff)$; for example any non-coherent $\CO_{X \times \CScheme}$-sheaf with zero structural map $\kappa$.

\begin{rem}
\label{Csh-local-nil-affine}
It is important to note that local nilpotence for $\kappa$-sheaves \emph{does not} mean that all local sections are killed by some power of $\kappa$. This is too weak. As an example consider $k=\BF_p$, $X=\Spec R$, $R=k[x]$ and $\kappa \colon R \to R$ given by $x^i \mapsto x^{((i+1)/p)-1}$ where we define $x$ raised to a non-integral exponent as equal to zero. Then clearly, every single element of $r$ is killed by some power of $\kappa$, however $R$ does not have a $\kappa$-submodule which is nilpotent.

The correct condition for local nilpotence, say of a $\kappa$-sheaf $\UCM$ on an affine scheme $X=\Spec R$, is that the \emph{$R \tensor \Coeff$-module generated by} every section $m \in \UCM$ is annihilated by some power of $\kappa$, \ie~for some $e \geq 0$ we have $\kappa^e((R \tensor \Coeff)m)=0$. This is made precise in the following two lemmata.
\end{rem}

\begin{lem}
\label{CSh-CharNilLocNil}
For a $\kappa$-sheaf $\UCM$ we define $\UCM_e$ as the largest $\kappa$-subsheaf on which $\kappa^e$ is zero. We let $\UCM_{nil}$ be the largest locally nilpotent $\kappa$-subsheaf of $\UCM$.
\begin{enumerate}
\item $\UCM_{nil}=\bigcup_{e \in \BN} \UCM_e$.
\item \label{CSh-it1} $\UCM_e$ is the sheaf
\[
    U \mapsto \{ m \in \UCM(U) | \kappa^e\Frid^e_*(\CO_{X \times \CScheme}\cdot m)=0 \}.
\]
\item \label{CSh-it2} $\UCM_{nil}$ is the sheaf
\[
    U \mapsto \{ m \in \UCM(U) | \exists e \in \BN : \kappa^e\Frid^e_*(\CO_{X \times \CScheme}\cdot m)=0 \}.
\]
\end{enumerate}
\end{lem}
\begin{proof}
Clearly, since $\kappa$ is additive, the sum of two subsheaves on which $\kappa^e$ acts as zero also has this property, so $\UCM_e$ is well defined. Equally clearly, $\UCM_{nil}$ is equal to the union of all nilpotent subsheaves which implies the equality $\UCM_{nil}=\bigcup_{e \in \BN} \UCM_e$, and this shows (a). To prove (b), let us denote the sheaf defined in (b) by $\UCN_e$. Then $\kappa^e(\UCN_e)=0$ and hence $\UCN_e \subseteq \UCM_e$. The reverse inclusion can be checked (affine) locally where it is clear by definition. A similar argument applies to \autoref{CSh-it2}.
\end{proof}
\begin{cor}
\label{CSh-CharLocNilCoh}
A $\kappa$-sheaf $\UCM$ is locally nilpotent if and only if for all coherent $\CO_{X \times \CScheme}$-subsheaves $M$ there is a $e \geq 0$ such that $\kappa^e(\Frid^e_*(M))=0$.
\end{cor}

Since every quasi-coherent $\CO_{X \times \CScheme}$-module is the union of its coherent submodules it follows that $\LNilC(X,\Coeff) \subseteq \IndCohC(X,\Coeff)$. Moreover one can easily show that $\LNilC(X,\Coeff)$ is closed under filtered direct limits.

\begin{prop}\label{CSh-LNilIsSerre}
All categories in the following diagram
\begin{equation}\label{CSh-CCatDiag}
{\parbox{0cm}{{
\xymatrix@C-12pt@R-30pt{
&\LNilC(X,\Coeff)\ar@{^{ (}->}[dr]&&\\
\NilC(X,\Coeff) \ar@{^{ (}->}[ur] \ar@{_{ (}->}[dr] &
&  \LNilCohC(X,\Coeff)\ar@{^{ (}->}[r] &\IndCohC(X,\Coeff)\rlap{,}\\
&\CohC(X,\Coeff) \ar@{_{ (}->}[ur] &&\\}
}}}
\end{equation}
are Serre sub-categories of
$\IndCohC(X,\Coeff)$, and $\NilC(X,\Coeff)$, $\CohC(X,\Coeff)$ are Serre sub-categories of $\QCohC(X,\Coeff)$.

If the Frobenius $\Fr$ on $X$ is finite, then all categories in diagram \autoref{CSh-CCatDiag} are Serre sub-categories of $\QCohC(X,\Coeff)$.\side{is needed to guarantee that $\LNilCohC$ is a Serre subcategory of $\QCohC$}\side{}
\end{prop}

\begin{proof}
In \autoref{CSh-Serre} above we already observed that $\CohC(X,\Coeff)$ is a Serre subcategory of $\QCoh(X,\Coeff)$, and that $\IndCohC(X,\Coeff)$ is a Serre subcategory of $\QCohC(X,\Coeff)$ under the assumption that $\Fr$ is finite. The only non-trivial part in each claim is to show that the respective sub-category is closed under extensions. For this let
\[
0 \to \UCV' \to  \UCV \to \UCV'' \to 0
\]
be a short exact sequence of $\kappa$-sheaves. We first verify that $\NilC(X,\Coeff)$ is closed under extensions in $\QCoh(X,\Coeff)$. For this assume that the outer terms of the above sequence lie in $\NilC(X,\Coeff)$ and the middle in  $\QCoh(X,\Coeff)$. The coherence of the middle term $\UCV$ is clear, since $\CohC(X,\Coeff)$ is a Serre subcategory which contains $\NilC(X,\Coeff)$. Choose $n,m\in\BN$ such that
$\kappa^m_{\CV'}=0$ and $\kappa^n_{\CV''}=0$. Then $\UCV$ is nilpotent since $\kappa^{m+n}_{\CV}=0$.

To check that $\LNilC(X,\Coeff)$ is closed under extensions in $\IndCohC(X,\Coeff)$ assume that the outer terms of the sequence above are in $\LNilC(X,\Coeff)$ and the middle lies in $\IndCohC(X,\Coeff)$. In particular, $\UCV$ is a union of its coherent $\kappa$-sub-sheaves $\UCW$. For any such there is a short exact sequence $0\to \UCW\cap\UCV'\to\UCW\to \UCW/(\UCW\cap\UCV')\to0$. Its outer terms lie in $\NilC(X,\Coeff)=\LNilC(X,\Coeff)\cap\CohC(X,\Coeff)$. By the previous case $\UCW$ must be nilpotent, and so $\UCV$ lies in $\LNilC(X,\Coeff)$.

Finally, if the outer terms lie in $\LNilCohC(X,\Coeff)$ we have coherent $\UCC' \subseteq \UCV'$ and $\UCC'' \subseteq \UCV''$ such that the both quotients are in $\LNilC(X,\Coeff)$. Since $\UCV \in \IndCohC(X,\Coeff)$ we can find a coherent $\kappa$-subsheaf $\UCC \subseteq \UCV$ surjecting onto $\UCC''$ and containing $\UCC'$. Possibly enlarging $\UCC'$ such that there is a short exact sequence $\UCC' \into \UCC \onto \UCC''$ we see that the quotient $\UCV/\UCC$ is an extension of the locally nilpotent $\UCV'/\UCC'$ and $\UCV''/\UCC''$. By the preceding case it follows that $\UCV/\UCC \in \LNilC(X,\Coeff)$, hence $\UCV \in \LNilCohC(X,\Coeff)$.
\end{proof}

\begin{ex}\label{CSh-ExampleNotSerreSub}
If $\Fr$ is not finite, then neither $\IndCohC(X,\Coeff)$, $\LNilC(X,\Coeff)$, nor $\LNilCohC(x,\Coeff)$ is closed under extensions in $\QCohC(X,\Coeff)$. For an example take $k$ a field with $[k:k^p]$ countably infinite (\eg~$k$ is a countable, purely transcendental extension of $\BF_p$), $X = \Spec k$, and $\Coeff=k$. Let $\{e_i\}_{i \in \mathbb{N}}$ be a basis of $k$ over $k^p$ and let $\CV=\oplus_{i \in \mathbb{N}} k \cdot g_i$ be a countable direct sum of copies of $k$. We define a $\kappa$-structure on $\CV$ as follows:
\begin{align*}
    \kappa(e_ig_1) &= e_1g_{i+1} \text{ and for $j \geq 2$ we set}\\
    \kappa(e_ig_j) &= \begin{cases} e_{i+1}g_j &\text{if $i<j$} \\ 0 &\text{if $i\geq j$} \end{cases}
\end{align*}
Then for $j \geq 2$, the coherent $\kappa$-subsheaf $k g_j$ is nilpotent of order $j$, hence $\UCV'=\oplus_{j \geq 2} k g_j$ is a locally nilpotent ind-coherent $\kappa$-subsheaf of $\UCV$. The quotient $\UCV/\UCV'$ is coherent and has zero structural map (since $\kappa(\UCV) \subseteq \UCV'$). However $\UCV$ cannot be locally nilpotent as this would require that $k \cdot g_1$ is annihilated by a single power of $\kappa$. But $\kappa(k \cdot g_1)=\UCV'$ by construction, and $\UCV'$ is clearly not annihilated by a single power of $\kappa$, this is not possible. Similarly, $\UCV$ cannot be ind-coherent, since any $\kappa$-subsheaf that has nonzero intersection with $k g_1$ must contain $k g_1$, and hence contains $\kappa(k g_1)=\UCV'$. Since $\UCV'$ is not coherent it follows that $\UCV$ cannot be ind-coherent. Hence $\UCV$ is an example for an extension of ind-coherent locally nilpotent $\kappa$-sheaves that is itself neither ind-coherent nor locally nilpotent. Clearly, $\UCV$ is also not in $\LNilCohC(X,\Coeff)$ even though both $\UCV'$ and $\UCV''$ are.
\end{ex}

\subsection{\'Etale pullback and localization}
For an essentially \'etale map of schemes $f\colon Y \to X$ the pullback $f^*$ of quasi-coherent sheaves induces a pullback functor for $\kappa$-sheaves. With this at hand we will show that all of the categories introduced above localize well, so that we will be able to check membership in each of these locally on an affine cover. This will be convenient in several cases below. As we will explain later, for an arbitrary essentially of finite type map it is the extraordinary inverse image functor $f^!$ from \cite{HartshorneRD} that induces a functor for Cartier sheaves. Since for an essentially \'etale map $f$ one has $f^* \cong f^!$ this will recover the more elementary construction that follows now. We start to recall a well known lemma.
\begin{lem}
\label{CSh-FrobEtale}
Let $j \colon U \to X$ be a an essentially \'etale\footnote{It is conceivable, that this lemma is valid if $j$ is only formally \'etale. As Brian Conrad pointed out to us, for maps of \emph{Noetherian} schemes formally \'etale does imply flatness (this is not true if one drops Noetherian), so the flat base change still applies.} map of schemes, then
\[
\xymatrix{   U \ar[r]^j \ar[d]_\Fr & X \ar[d]^\Fr \\
             U \ar[r]^j & X
             }
\]
is a cartesian square. Furthermore, the natural transformation of functors $j^*\Fr_* \to \Fr_* j^*$ is an isomorphism.
\end{lem}
\begin{proof}
We need to show that the natural map $\Fr_{U/X} \colon U \to U \times_{X,\Fr} X$ is an isomorphism. Since it is a bijection on the level of underlying topological spaces, we may localize to check that the map is an isomorphism. Then one reduces to the cases that $j$ is finite and \'etale (see for example \cite[Page 60]{HH90}), $j$ is an open immersion or that $j$ is a localization at some multiplicative set -- both cases are straightforward.

For the second statement, we apply to the unit of adjunction $\id \to j_*j^*$ the functor $\Fr_*$ to get $\Fr_* \to \Fr_*j_*j^* \cong j_* \Fr_* j^*$. Adjunction for $j_*$ and $j^*$ again yields the natural transformation $j^* \Fr_* \to \Fr_* j^*$, and flat base change implies that this is an isomorphism. Of course we use the fact that essentially \'etale implies flat, since both, \'etale morphisms as well as localizations are flat.
\end{proof}
\begin{defn}
\label{CSh-EtalePullback}
Let $j \colon U \to X$ be an essentially \'etale map and $\UCV$ a $\kappa$-sheaf on $X \times \CScheme$. We define $j^* \UCV$ as the $\kappa$-sheaf whose underlying sheaf is $(j \times \id)^* \CV$ with structural morphism given by the following composition
\[
    \Frid_* (j \times \id)^* \CV \cong (j \times \id)^* \Frid_* \CV \to[(j \times \id)^*\kappa] (j \times \id)^* \CV
\]
where the first isomorphism is the one from \autoref{CSh-FrobEtale}.
\end{defn}
Of particular interest to us are the case where $U \subseteq X$ is an open immersion. This is covered by the above, hence  $\UCV|_{U\times \CScheme} \colonequals  f^*\UCV$ is naturally a $\kappa$-sheaf on $U$ (with coefficients in $\Coeff$). Similarly, if $\CS$ is an arbitrary sheaf of multiplicative sets on $X$, then $S^{-1}\UCV$ is naturally a $\kappa$-sheaf. In this case, the structural map $\kappa$ is locally given by $\kappa(v/s)\colonequals \kappa(vs^{q-1})/s$. If $x \in X$ is a point, then we denote by $\UCV_x$ the $\kappa$-sheaf on $\Spec \CO_{X,x}$ (with coefficients in $\Coeff$) whose underlying $\CO_{X,x} \tensor \Coeff$-module is $(j \times \id)^*\CV$ where $j \colon \Spec \CO_{X,x} \to X$ is the natural map.

It follows immediately from the functoriality and exactness that pullback along essentially \'etale maps preserves (local) nilpotence and other properties of a $\kappa$-sheaf:

\begin{lem}\label{CSh-EtalePullbackPreservesNil}
Let $j \colon Y \to X$ be essentially \'etale.  Then the exact functor $j^*$ preserves nilpotence, local nilpotence, coherence, local nil-coherence and ind-coherence of $\kappa$-sheaves.
\end{lem}
\begin{proof}
Nilpotence is preserved by functoriality. Coherence is clear. Since $j^*$ is exact and commutes with filtered direct limits the other assertions are immediate.
\end{proof}
Here is a converse of this observation that any of these properties may be checked on an open cover.
\begin{lem}\label{Csh-NilpoenceLocal}
Let $\UCV$ be a $\kappa$-sheaf on $X$ (Noetherian) and $\{U_i\}_{i \in I}$ a (Zariski) open covering of $X$. Then $\UCV$ is nilpotent if and only if for every $U_i$ the restriction $\UCV|_{U_i}$ is nilpotent, i.e. nilpotence can be checked on Zariski coverings.

Furthermore, $\UCV$ is locally nilpotent if and only if for all $x \in X$ the stalk $\UCV_x$ is locally nilpotent, i.e. local nilpotence can be checked on stalks (and hence on Zariski open coverings).
\end{lem}
\begin{proof}
From \autoref{CSh-CharNilLocNil} it follows that the formation of the (locally) nilpotent part $\UCV_e$ for $e \geq 0$ and $\UCV_{\op{nil}}$ commute with the restriction to an open set and with taking stalks. Now, $\UCV$ is nilpotent of order $e$ (resp. locally nilpotent) if and only if for some $e \geq 0$ the inclusion $\UCV_e \subseteq \UCV$ is equality (resp. the inclusion $\UCV_{\op{nil}} \subseteq \UCV$ is equality). From this it immediately follows that $\UCV$ is nilpotent of order $e$ (resp. locally nilpotent) if and only if for all $x \in X$ the stalks $\UCV_x$ are nilpotent of order $e$ (resp. locally nilpotent). This shows all statements except the implication that if all $\UCV|_{U_i}$ are nilpotent (of some order $e_i$) then so is $\UCV$. But this follows since we always assume that $X$ is Noetherian which allows us to resort to a finite cover and to take $e$ as the maximum of the finitely many $e_i$. Then it follows that $\UCV$ is nilpotent of order~$e$.
\end{proof}
Next we point out that membership in the categories in Diagram \autoref{CSh-CCatDiag} can also be checked locally.
\begin{prop}
\label{CSh-PropertiesLocalize}
\begin{enumerate}
\item Membership in each of the categories of Diagram \autoref{CSh-CCatDiag}, $\CohC$, $\NilC$, $\LNilC$, $\LNilCohC$, or $\IndCohC$, can be checked on Zariski open covers.
\item For $\UCV \in \CohC$ membership in $\NilC$ can be checked on stalks, \ie after pullback to the spectrum of the local rings of the points of $X$.
\end{enumerate}
\end{prop}
\begin{proof}
For (a) note that coherence is local for Zariski coverings, which shows the statement for $\CohC$. For $\NilC$ and $\LNilC$ the statement follows from the preceding \autoref{Csh-NilpoenceLocal}. Using the characterization in \autoref{Csh-IndCohCharacterization} that $\UCM \in \IndCohC(X,\Coeff)$ if and only if the $\kappa$-subsheaf generated by any coherent $\CO_{X \times \Coeff}$-submodule of $\UCM$ is also coherent, assertion (a) follows for $\IndCohC(X,\Coeff)$ as well, since coherence for $\CO_{X \times \CScheme}$-modules is checked locally. It remains to show (a) for $\LNilCohC$. For this let $U_1, \ldots, U_n$ be an open affine cover of $X$ such that for each $i$ the restriction $\UCV|_{U_i}$ is in $\LNilCohC$. That is we have for each $i$ a short exact sequence
\[
    0 \to \UCC_i \to \UCV|_{U_i} \to \UCN_i \to 0
\]
where $\UCC_i$ is a coherent and $\UCN_i$ a nilpotent $\kappa$-sheaf on $U_i$. Extend each $\UCC_i$ to a coherent subsheaf $\CC'_i$ of $\UCV$. Since $\UCV$ is ind-coherent, by \autoref{Csh-IndCohCharacterization} the $\kappa$-subsheaf $\UCC$ of $\UCV$ which is generated by the sum of all $\CC_i$ is coherent. By (a) for $\LNilC$ it follows that $\UCV/\UCC$ is locally nilpotent and hence $\UCV \in \LNilCohC$.

For (b) observe that for coherent $\kappa$-sheaves nilpotence is the same as local nilpotence. Hence the assertion follows from \autoref{Csh-NilpoenceLocal}.
\end{proof}

\subsection{Cartier crystals}

Our aim is to study Cartier sheaves \emph{up to nilpotence}. This is made precise by the process of localization of the category of coherent Cartier sheaves $\CohC(X,\Coeff)$ at its Serre subcategory of nilpotent Cartier sheaves $\NilC(X,\Coeff)$. We follow the setup for localization as recalled in \cite[\S~2.2]{BoPi.CohomCrys}. The process is very similar to the passage from the category of (bounded) complexes of an abelian category to the derived category by inverting quasi-isomorphisms, but simpler. Namely, the Serre sub-categories of a Grothendieck category are in one to one correspondence to the saturated multiplicative systems. Starting with the Serre subcategory of nilpotent Cartier sheaves, the corresponding saturated multiplicative system consists precisely of those maps $\phi$ whose kernel and cokernel are locally nilpotent. We define:

\begin{defn}
\label{CSh-NilisoDef}
A homomorphism of $\kappa$-sheaves is called a \emph{nil-iso\-morph\-ism} if both its kernel and cokernel are locally
  nilpotent.
\end{defn}

Since a coherent $\kappa$-sheaf is nilpotent if and only it is locally nilpotent, a map of coherent $\kappa$-sheaves is a nil-iso\-morph\-ism if and only if its kernel and cokernel are \emph{nilpotent} (and not merely locally nilpotent). The following simple characterization of nil-isomorphisms will be useful. We omit the straightforward proof.

\begin{prop}\label{CSh-NilIso}
    Kernel and cokernel of a homomorphism $\phi\! \colon \UCV\to\UCW$ of $\kappa$-sheaves are nilpotent if and only if there exist $n\ge0$ and a homomorphism of $\kappa$-sheaves $\alpha$ making the following diagram commute:
    \[
        \xymatrix{ \Fr^n_*\UCV \ar[r]^-{\kappa^n} \ar[d]_{\Fr^n_*\phi}& \UCV \ar[d]^\phi\\ \Fr^n_*\UCW \ar[ur]^\alpha \ar[r]^-{\kappa^n} & \UCW \rlap{.} \\
        }
    \]
    In particular, if such a diagram exists, then $\phi$ is a nil-isomorphism. If $\phi\! \colon \UCV \to \UCW$ is a homomorphism of \emph{coherent} $\kappa$-sheaves then the existence of the diagram is equivalent to $\phi$ being a nil-isomorphism.
\end{prop}

\begin{cor}\label{CSh-TauIsNilIso}
For every $\kappa$-sheaf \,$\UCV$ and for every $n\ge0$, the homomorphism $\kappa^n_{\CV}\! \colon \Fr^n  _*\UCV \to \UCV$ is a nil-isomorphism.
\end{cor}
\begin{proof}
Apply \autoref{CSh-NilIso} to $\phi= \kappa^n_{\CV}$ and $\alpha = \id_{\Fr^n_*\CV}$.
\end{proof}

By \autoref{CSh-LNilIsSerre} the category $\NilC(X,\Coeff)$ is a Serre
subcategory of $\CohC(X,\Coeff)$. By \cite[Prop.~2.3.2]{BoPi.CohomCrys}
the corresponding saturated multiplicative system is the
class of nil-isomorphisms.  In the remainder of this section, we will
study basic properties of the associated localized category.

Our first concern is existence.  Recall first that $\QCohC(X,\Coeff)$ is a
Grothendieck category by \autoref{CSh-QisGrothCat}.  Thus it is locally
small, and the same follows for any full subcategory.  From \cite[Prop.~2.3.4]{BoPi.CohomCrys} we deduce that localization of $\CohC(X,\Coeff)$ with
respect to nil-isomorphisms yields a well defined abelian category.

\begin{defn}\label{CSh-QuotDef}
The localization of $\CohC(X,\Coeff)$ at the multiplicative system of
nil-isomorphism is denoted by $\CrysC(X,\Coeff)$.
We refer to the objects of $\Crys(X,\Coeff)$ as \emph{Cartier-$\Coeff$-crystals
  on}~$X$, or just as Cartier crystals, or $\kappa$-crystals.
\end{defn}
In order to define certain functors on the category of $\kappa$-crystals it is necessary to give a slightly different presentation of $\CrysC(X,\Coeff)$ within a bigger ambient category. Recall the natural inclusion of subcategories
\begin{equation}
    \CohC(X,\Coeff) \subseteq \LNilCohC(X,\Coeff) \subseteq \IndCohC(X,\Coeff) \subseteq \QCohC(X,\Coeff)
\end{equation}
By localization at the multiplicative system of nil-isomorphisms we obtain a sequence of inclusions\side{The existence of $\QCrysC$ requires $\Fr$-finite, since only then is $\LNilCohC$ a Serre subcategory. To remedy this, one could close $\LNilCohC$ up under extensions in $\QCohC$ and localize at the resulting Serre subcategory. Since $\LNilCohC$ is a Serre-subcategory the statements in this section are all still true.}
\begin{equation}\label{Csh-eq_incl}
    \CrysC(X,\Coeff) \subseteq \LNilCrysC(X,\Coeff) \subseteq \IndCrysC(X,\Coeff) \subseteq \QCrysC(X,\Coeff).
\end{equation}
Note that in order for the localization to exist we must know that the localizing subcategory $\LNilC$ is a Serre-subcategory. Hence in the case of $\QCohC$ we need to assume that the Frobenius on $X$ is finite, \cf~\autoref{CSh-LNilIsSerre}. One has the following immediate consequence.
\begin{prop}\label{CSh-Crys=NilCrys}
The inclusions in \autoref{Csh-eq_incl} are fully faithful and $$\CrysC(X,\Coeff) \subseteq \LNilCrysC(X,\Coeff)$$ is an equivalence of categories.
\end{prop}
\begin{proof}
The full faithfulness of $\CrysC(X,\Coeff) \subseteq \LNilCrysC(X,\Coeff)$ follows from 
$\NilC(X,\Coeff)=\LNilC(X,\Coeff) \cap \CohC(X,\Coeff)$ which is true by definition. For the other of inclusions of \autoref{Csh-eq_incl} this assertion is obvious. By definition it follows that the essential image of $\CrysC(X,\Coeff)$ in $\IndCrysC(X,\Coeff)$ is equal to $\LNilCrysC(X,\Coeff)$. Hence the second claim follows.
\end{proof}
However $\CrysC(X,\Coeff)$ is not closed under isomorphism in $\IndCrysC(X,\Coeff)$, hence is not a Serre subcategory. But its essential image $\LNilCrysC(X,\Coeff)$ is a Serre subcategory of $\IndCrysC(X,\Coeff)$. If $\Fr$ is finite, both are Serre subcategories of $\QCrysC(X,\hspace{-1pt}\Coeff)$.

\begin{rem}\label{CSh-Rem}
    Recall that a localized category possesses the same class of objects as the original one, but with different homomorphism sets. To avoid confusion one should keep in mind that the categorical properties of an object depend solely on the homomorphisms.  Thus in the interest of clarity we will always speak of \emph{the $\kappa$-crystal associated to} a $\kappa$-sheaf, and of \emph{the $\kappa$-sheaf underlying} a $\kappa$-crystal, although these objects are really `the same'.  This rule will help to clarify whether a property refers to the category $\CohC(X,\Coeff)$ or to $\CrysC(X,\Coeff)$.

Based on the equivalence of Serre subcategories and saturated multiplicative systems, e.g.~\cite[Prop.~2.3.2]{BoPi.CohomCrys}, we say that a $\kappa$-crystal is zero if and only if its underlying $\kappa$-sheaf is nilpotent.  To distinguish morphisms in the quotient category $\CrysC(X,\Coeff)$ from those in $\CohC(X,\Coeff)$ we often denote them by dotted arrows $\dotto$.  If a morphism in $\CrysC(X,\Coeff)$ is the image of a morphism in $\CohC(X,\Coeff)$, by abuse of notation we also denote it by a solid arrow. We use double arrows $\Longto$ to denote nil-isomorphisms.
\end{rem}


From \autoref{CSh-QisGrothCat} and \cite[Prop.~2.4.3, Prop.~2.4.9]{BoPi.CohomCrys} we deduce:
\begin{theorem}\label{CSh-GrothDirLim}\side{Again, existence of $QCrysC$ requires $\Fr$-finite unless we close $\LNilCohC$ up under extension in $\QCohC$.}
$\QCrysC(X,\Coeff)$ is a Grothendieck category, and $\IndCrysC(X,\Coeff)$ is a full subcategory which is closed under filtered direct limits and hence itself a Grothendieck category. $\CrysC(X,\Coeff)$ is a full subcategory of both. In particular, all three are locally small abelian categories.
\end{theorem}

As with nil-isomorphisms between coherent $\kappa$-sheaves, there is a
standard way to represent homomorphisms of crystals.

\begin{prop}\label{Crys1-ReprMorph}
Any homomorphism $\phi\! \colon \UCV\dotto\UCW$ in $\CrysC(X,\Coeff)$ can be
represented for suitable $n$ by a diagram
\[
    \UCV \stackrel{\ \kappa^n}\Longleftarrow \Fr^n_*\UCV \longto \UCW\, .
\]
\end{prop}

\begin{proof}
Any morphism in the localized category $\CrysC(X,\Coeff)$ can be represented
by a diagram $\UCV\stackrel{s}{\Longleftarrow} \UCV' \stackrel{\psi}\longto
\UCW$ in $\CohC(X,\Coeff)$ where $s$ is a nil-isomorphism.  Applying
\autoref{CSh-NilIso} to~$s$, we obtain a commutative diagram
\[ \xymatrix{
\UCV & \UCV' \ar@{=>}[l]_s \ar[r]^{\psi} & \UCW \\
& \Fr^n_*\UCV \ar@{=>}[ul]^{\kappa^n}
\ar[u]_\alpha \ar[ur]_{\psi\,\alpha} & \\}\]
where $\kappa^n$ is again a nil-isomorphism by
\autoref{CSh-TauIsNilIso}.  Thus $\phi$ is also represented by
the lower edge of this diagram, which proves the assertion.
\end{proof}

\section{Cohomological operations for Cartier sheaves}\label{CSh-DefFunc}

For \emph{left} $\OXFr$--modules (called $\tau$-sheaves in the terminology of \cite{BoPi.CohomCrys}) the authors of \cite{BoPi.CohomCrys} construct the following cohomological functors: (derived) pullback $f^*$ for any morphism $f \colon Y \to X$ of schemes, (derived) proper push-forward $f_!$ for morphisms of finite type and (derived) tensor product one expects that the dual operations, namely $f^!$, $f_*$ and a twisted tensor product $\boxtimes$ are naturally defined for \emph{right} $\OXFr$--modules, \ie for Cartier sheaves. In essence this will turn out to be the case, except that for $f^!$ we require $f$ to be essentially of finite type. This seems a natural condition, since already in Grothendieck-Serre duality for coherent sheaves one has such a restriction on~$f^!$.


In this section we will give the constructions of $f^!$, $f_*$ on $\QCrysC(\ldots)$ and describe the main properties and compatibilities of these functors; see \autoref{CSh-Pushforward} and \autoref{CSh-sec-InverseImage}. A further functor $\boxtimes$, a kind of tensor product, is postponed past the main duality theorem which is treated in \cite{BliBoe.CartierDuality}. As we set things up, it will be clear that the functors $f^!$ and $f_*$ for Cartier sheaves agree with the corresponding functors on the underlying quasi-coherent sheaves. This is convenient as we can use the full machinery of \cite{HartshorneRD} and \cite{LipmanGrothDual} which frees us from checking many compatibilities which are already verified for $f^!$ and $f_*$ in the classical context. In this sense our approach is similar to the one of Emerton and Kisin in \cite{EmKis.Fcrys} who also use the close relation of their functors for unit $\CO_{F,X}$-modules to the classical ones on the underlying $\CO_X$-modules. The precise relationship between our and Emerton and Kisin's theory will be explained in detail in \cite{BliBoe.CartierEmKis}.\m?{Aside: It might be possible to follow the more modern "Bousfield" techniques as in Neeman, to conclude form our existence of $Rf_*$ on the unbounded derived category, the existence of a right adjoint $f^\times$. Form this one can piece together an $f^!$ by setting $f^!=f^\times$ for proper and $f^!=f^*$ for essentially etale. If I can work this out I could include an extended remark below. Nevertheless we still use the explicit nature of $f^!$ so this would not bypass a lot of work, I think}

An important step to transfer the construction of the derived functors $f_*$ and $f^!$ for quasi-coherent sheaves to Cartier sheaves is a basic but fundamental comparison result in \autoref{CSh-CompDCI} between certain derived categories, namely that $\D^*(\QCohC(X,\Coeff))$ is equivalent to  $\D^*(\QCoh(X\times\CScheme))_\kappa$ -- here the latter denotes the category of pairs consisting of an object $\CV^\bullet$ in the derived category of quasi-coherent sheaves together with a homomorphism $\Fr_*\CV^\bullet\to\CV^\bullet$ in \emph{the derived category of quasi-coherent sheaves}. For $\D^*(\QCoh(X\times\CScheme))_\kappa$ it is immediate that the desired functors $f^!$ and $f_*$ are just the restrictions of the corresponding functors on the derived category of quasi-coherent sheaves $\D^*(\QCoh(X\times\CScheme))$ on $X\times\CScheme$. 

 It is also beneficial to have a direct construction of functors as derived functors on derived categories of Cartier sheaves or Cartier crystals. For $\tau$-sheaves this is the approach taken in \cite{BoPi.CohomCrys}. In \autoref{CSh-Pushforward} we include a treatment along these lines for the functor $Rf_*$ based on the use of \v{C}ech resolutions, which shows that $Rf_*$ is the derived functor of $f_*$ on $\D^*(\QCohC(X,\Coeff))$. This very much simplifies the proof that $Rf_*$ preserves nilpotence as one can use a spectral sequence computation. For $f^!$ we do not have such a description and we have to work significantly more to obtain the same result.

In the remainder of this article, both of the two viewpoints represented by the categories $\D^*(\QCohC(X,\Coeff))$ and  $\D^*(\QCoh(X\times\CScheme))_\kappa$ serve their purpose: On the one hand, the category $\D^*(\QCoh(X\times\CScheme))_\kappa$ is taken under Grothendieck-Serre duality to a category which can be compared with a derived category of $\tau$-sheaves. This will be important in \cite{BliBoe.CartierEmKis}, the second part of this work, where we shall treat the subtleties of such a duality for derived  categories of quasi-crystals with bounded cohomology in crystals between the $\kappa$- and $\tau$-side. On the other hand, the category $\D^*(\QCohC(X,\Coeff))$ admits a direct passage to the derived category $\D^*(\QCrysC(X,\Coeff))$ of quasi-crystals, allowing us to transfer the derived functors $Rf_*$ and $f^!$ to quasi-crystals as well.

In addition we follow the viewpoint of \cite{BoPi.CohomCrys} and work out much of our theory for the derived subcategory $\D^*(\IndCohC(X,\Coeff))$ of $\D^*(\QCohC(X,\Coeff))$ built out of ind-objects. In this setting, we construct $Rf_*$ directly and in \cite{BliBoe.CartierEmKis} we shall also construct $f^!$ for morphisms $f\colon Y\to X$ with $X$ regular. There we shall also see that the ind-categories will be technically useful in the passage from bounded $\kappa$-complexes with coherent cohomology to derived categories of $\kappa$-coherent sheaves.

\subsection{Comparison of derived categories}\label{CSh-CompDCI}

In this subsection we prove the comparison result on derived categories which was alluded to above. It is crucial to the construction of our cohomological operations (which are described below) and also crucial to the duality between $\kappa$-sheaves and the $\tau$-sheaves of \cite{BoPi.CohomCrys} which is treated in \cite{BliBoe.CartierDuality}.

For $*\in\{b,+,-,\emptyset\}$ and $?\in\{\lnil,\coh,\lncoh,\ind,\emptyset\}$,
\[
\kappa=\kappa_{\CV^\bullet}\colon \Frid_*\CV^\bullet\longto \CV^\bullet
\]
is a homomorphism in $\D^*(\QCoh(X\times\CScheme))$, such that for all $i\in\BZ$ the induced pair
\[
\Big( H^i(\CV^\bullet), \ H^i(\kappa_{\CV^\bullet}) \colon \Frid_*H^i(\CV^\bullet) \cong H^i(\Frid_*\CV^\bullet) \to[H^i\kappa] H^i(\CV^\bullet) \Big)
\]
lies in $\LNilC(X,\Coeff)$ for $? = \lnil$, in $\CohC(X,\Coeff)$ for $? = \coh$, in $\LNilCohC(X,\Coeff)$ for $? = \lncoh$, in $\IndCohC(X,\Coeff)$ for $? = \ind$, or without further condition if $?=\emptyset$. A  \emph{morphism} from $(\CV^\bullet,\kappa_{\CV})$ to $(\CW^\bullet,\kappa_{\CW})$ is a homomorphism $\phi\!:\CV^\bullet\to\CW^\bullet \in \D^*_?(\QCoh(X\times\CScheme))$, such that the following diagram commutes in $\D^*(\QCoh(X\times\CScheme))$
\begin{equation}
\label{CompDCI-TwoDefs}
{\parbox{0cm}{{\xymatrix@C+2pc{
\Frid_*\CV^\bullet\ar[r]^-{\kappa_{\CV^\bullet}}\ar[d]_{\Frid_*\phi}& \CV^\bullet\ar[d]^\phi \\
\Frid_*\CW^\bullet\ar[r]_-{\kappa_{\CW^\bullet}} & \CW^\bullet\rlap{.}
\\} }}}
\end{equation}
There is an obvious canonical inclusion
\begin{equation}\label{CompDCI-CanIncl-kappa-InAndOut}
\D^*_?(\QCohC(X,\Coeff))\into \D^*_?(\QCoh(X\times \CScheme))_\kappa
\end{equation}
where the former denotes the derived category of complexes of $\kappa$--sheaves with cohomological bounded-ness condition $*$ and whose cohomology lies in $?$. 
The main result of the present subsection is the following comparison theorem which is essential to the construction of functors. 

\begin{theorem}\label{CompDCI-EquivOfTwoCats}
For $* \in \{\emptyset, +,-,b\}$ and $? \in \{\coh,\lncoh,\ind,\emptyset\}$, the natural inclusion \autoref{CompDCI-CanIncl-kappa-InAndOut}
is an equivalence of categories.
\end{theorem}

\begin{proof}
In this proof we denote quasi-isomorphisms by a double arrow ``$\Longleftarrow$''.
We begin by showing that the inclusion \autoref{CompDCI-CanIncl-kappa-InAndOut} is essentially surjective.  
Suppose that the diagram
\[
  \kappa_{\CV^\bullet}\colon \Frid_*\CV^\bullet\to[\phi_0] \CV^\bullet_1\stackrel{\psi_0}{\Longleftarrow} \CV^\bullet.
\]
represents an element in $ \D^*_?(\QCoh(X\times \CScheme))_\kappa$. We need to show that, up to quasi-isomorphisms, this object lies in the image of $\D^*_?(\QCohC(X,\Coeff))$. Setting $\CV_0^\bullet \colon = \CV^\bullet$, below we shall inductively construct a commutative diagram
\[
\xymatrix{& \ar[d]^{\phi_{0}}\Frid_*\CV^\bullet_0\ar@{=>}[r]&\ar[d]^{\phi_1} \Frid_*\CV^\bullet_1\ar@{=>}[r]&\ar[d]^{\phi_2}\Frid_*\CV^\bullet_2\ar@{=>}[r]&\ldots\\
 \CV^\bullet_0\ar@{=>}[r]& \CV^\bullet_1\ar@{=>}[r]&\CV^\bullet_2\ar@{=>}[r]&\CV^\bullet_3\ar@{=>}[r]&\ldots\\}
 \]
in $ \D^b_\coh(\QCoh(X\times \CScheme))$ of complexes such that the horizontal maps are quasi-isomorphisms. Since $\Frid_*$ is exact and commutes with direct limits this map of directed systems induces on the direct limit $\CV^\bullet_\infty$ a (honest) map of complexes
\begin{equation}
\label{CompDCI-KappaInf}
\kappa_\infty\!:\Frid_*\CV^\bullet_\infty\longto \CV^\bullet_\infty
\end{equation}
which makes the following diagram commutative showing that $(\CV_\infty,\kappa_\infty)$ is quasi-isomorphic to $(\CV^\bullet,\kappa_{\CV^\bullet})$:
\begin{equation}\label{CompOfDCI-QuisToDKappa}
\xymatrix{\Frid^*\CV^\bullet  \ar@{..>}[d]_{\kappa_{\CV^\bullet}} \ar@{=>}[r] & \Frid^*\CV^\bullet_\infty \ar[d]^{\kappa_\infty}\\
\CV^\bullet \ar@{=>}[r] & \CV^\bullet_\infty \\}
\end{equation}
By the construction of $\CV^\bullet_\infty$ being quasi-isomorphic to $\CV^\bullet$, it satisfies the same constraints on cohomology ($\emptyset$, $\lnil$, $\coh$, $\lncoh$ or $\ind$). If moreover $\CV^\bullet$ is concentrated in degrees $[a,b]$, then the cohomology of $\CV^\bullet_\infty$  is concentrated in this interval. Thus by standard truncation functors, we may assume that $\CV^\bullet_\infty$ itself is concentrated in $[a,b]$, and hence satisfies the same boundedness conditions as $\CV^\bullet$.

We now give the rather straightforward construction of the directed system announced above.

Consider the following diagram where the left portion represents our given morphism $\kappa_{\CV^\bullet}$.
\[
\xymatrix{&\Frid_*\CV^\bullet_0\ar[d]^{\phi_0}\ar@{=>}[rr]^{\Frid_*\psi_0}&&\ar@{-->}[d]^{\phi_1} \Frid_*\CV^\bullet_1 \\
 \CV^\bullet_0\ar@{=>}[r]^{\psi_0}& \CV^\bullet_1\ar@{==>}[rr]^{\psi_1}&&\CV^\bullet_2
 }
 \]
Since $\Fr$ is affine, the functor $\Frid_*$ is exact and hence the top map is still a quasi-isomorphism. Now the dashed arrows exist and form a commuting square by the axiom of localization changing a left fraction into a right fraction. Iterating this procedure we obtain in the limit \autoref{CompDCI-KappaInf}.

To show full faithfulness consider a morphism $\CW^\bullet \dotto \CV^\bullet $ in  $ \D^*_?(\QCoh(X\times \CScheme))_\kappa$. By essential surjectivity \autoref{CompOfDCI-QuisToDKappa} we may assume that we have a diagram of complexes
\[
\xymatrix @C
+2pc {
\Frid_*\CV^\bullet  \ar[d]_{\kappa_{\CV^\bullet}} \ar@{=>}[r]^{\Frid_*\alpha} & \Frid_*\CU^\bullet \ar[d]_{\kappa_{\CU^\bullet}} & \ar[l]_{\Frid^*\phi}   \Frid^*\CW^\bullet \ar[d]_{\kappa_{\CW^\bullet}}  \\
\CV^\bullet \ar@{=>}[r]^{\alpha} & \CU^\bullet & \CW^\bullet \ar[l]_{\phi} \rlap{,} \\}
\]
 where all vertical maps are maps of complexes and both squares \emph{commute in the derived category} $ \D^*_?(\QCoh(X\times \CScheme))$, \ie the squares composed with a suitable quasi-isomorphism $\CU^\bullet \xRightarrow{\ \phi\ } \CT^\bullet$ commute as maps of complexes to $\CT^\bullet$.

Now observe that $\Frid^*\CT^\bullet   \oT[\Frid^*\phi]  \Frid^*\CU^\bullet \to[\kappa_{\CU^\bullet}] \CU^\bullet  \xRightarrow{\ \phi\ }  \CT^\bullet$ together with $\CT^\bullet$ defines an object of $ \D^*_?(\QCoh(X\times \CScheme))_\kappa$, which by essential surjectivity is quasi-isomorphic to an object in the  image of $\D^*_?(\QCohC(X,\Coeff))$.  Replacing $(\CU^\bullet,\kappa_{\CU^\bullet})$ by this object, we may assume that the previous diagram commutes as a diagram of complexes. This proves that the given morphism $\CV^\bullet \dotto \CW^\bullet $ arises from a morphism in $\D^*_?(\QCohC(X,\Coeff))$ under the natural inclusion \autoref{CompDCI-CanIncl-kappa-InAndOut}. An entirely similar argument shows that the image of a morphism under \autoref{CompDCI-CanIncl-kappa-InAndOut} is zero in $ \D^b_\coh(\QCoh(X\times \CScheme)_\kappa$ if and only if it was zero in $\D^*_?(\QCohC(X,\Coeff))$. Hence the proof of full faithfulness is complete.
\end{proof}

In anticipation of the duality proven in \cite{BliBoe.CartierDuality} we now record a result for $\tau$-sheaves of \cite{BoPi.CohomCrys} analogous to \autoref{CompDCI-EquivOfTwoCats} for $\kappa$-sheaeves above: By \cite[Ch.~3]{BoPi.CohomCrys} a $\tau$-sheaf on $X$ over $\Coeff$ is a pair $\UCF=(\CF,\tau)$ consisting of a quasi-coherent $\CO_{X\times\CScheme}$-module $\CF$ together with an $\CO_{X\times\CScheme}$-linear map $\tau\colon \Frid^*\CF\to\CF$. As in \autoref{sec.CartierCrys} one introduces for such pairs categories of $\tau$-sheaves $\CohT$, $\NilT$, $\LNilT$, $\LNilCohT$, $\IndCohT$, $\QCohT$ and, by localization, of crystals $\CrysT$, $\LNilCrysT$, $\IndCrysT$,  $\QCrysT$ for any argument $(X,\Coeff)$. In the same way as at the beginning of this section for $\kappa$-sheaves, one can form the derived categories $D^*_?(\QCohT(X,\Coeff))$ and $D^*_?(\QCoh(X,\Coeff))_\tau$, where the latter denotes pairs $(\CF^\bullet,\tau)$ consisting of $\CF^\bullet \in D^*_?(\QCoh(X,\Coeff))_\tau$ and a \emph{morphism in the derived category} $\tau \colon \CF^\bullet \to \Frid_*\CF^\bullet$. The following analog of Theorem~\ref{CompDCI-EquivOfTwoCats} holds in this context:

\begin{theorem}\label{CompDCI-EquivOfTauCats}
For $* \in \{\emptyset, +,-,b\}$ and $? \in \{\lnil,\coh,\lncoh,\ind,\emptyset\}$, the natural inclusion
\[
\D^*_?(\QCohT(X,\Coeff)) \into \D^*_?(\QCoh(X\times \CScheme))_\tau
\]
is an equivalence of categories.
\end{theorem}
\begin{proof}
The proof proceeds in exactly the same way as that for Theorem~\ref{CompDCI-EquivOfTwoCats} with one change. The basic objects in that proof where morphisms $\Frid_*\CV^\bullet\to\CV^\bullet$ while the basic objects here will be morphisms $\CF^\bullet\to\Frid_*\CF^\bullet$. 
\end{proof}

\subsection{Pushforward}\label{CSh-Pushforward}

Throughout this subsection let $f\colon Y \to X$ be a morphism of schemes. We first define a derived functor $Rf_*$ on Cartier sheaves. Then we explain how to transfer this functor to Cartier crystals. Both results are rather straightforward but serve as a model for other derived functors on Cartier crystals such as $f^!$ for $f$ a finite morphism. Next we present a more delicate result who's proof goes back to an idea of Anderson in \cite{AndersonL}. It says that if $f$ is of finite type, then $Rf_*$ preserves (up to equivalence of categories) the derived category of Cartier crystals. 

\begin{defn}
\label{DF-pushforward}
For any $\UCW\in\QCohC(Y,\Coeff)$ we define $f_* \UCW$ as the $\kappa$-sheaf on $X$ over $\Coeff$ whose underlying sheaf is $(f\times\id)_* \CW$ and whose structural morphism is the composite
\[
 \kappa_{f_*\UCW} \colon  \Frid_* (f \times \id)_* \CV \to[\simeq] (f \times \id)_* \Frid_* \CV \to[(f \times \id)_*\kappa_\CV] (f \times \id)_* \CV \ .
\]
\end{defn}
The construction of the derived functor of $f_*$ is facilitated by the fact that the usual \v{C}ech complex of a $\kappa$-sheaf associated to an open covering is naturally a complex of $\kappa$-sheaves. This is simply a consequence of the \'etale pullback for $\kappa$-sheaves constructed in \autoref{CSh-EtalePullback}.

Concretely, for any finite cover $\FU$ of $Y$ and any $\UCV\in\QCohC(Y,\Coeff)$ we denote by $\Cech^\bullet_\FU(\UCV)$ the \v{C}ech complex of $\UCV$, cf.~\cite[III.4]{Hartshorne} or \cite[5.4]{BoPi.CohomCrys}. It is a resolution of $\UCV$ in $\D^*(\QCohC(Y,\Coeff))$ via the natural co-augmentation homomorphism $\UCV\to \Cech_\FU^\bullet(\UCV)$. The construction of the \v{C}ech resolution is functorial in $\UCV$ and thus extends to complexes $\UCV^\bullet$. By $\Cech_\FU(\UCV^\bullet)$ we denote the total complex of the bicomplex thus obtained\footnote{Since we are working on noetherian schemes and restrict to finite covers for our \v{C}ech complexes this works even for unbounded complexes}. It is a resolution of~$\UCV^\bullet$ in $\D^*(\QCohC(Y,\Coeff))$. Following  \cite[Thm.~I.5.1]{HartshorneRD} or \cite[Thm.~6.4.1]{BoPi.CohomCrys}, one deduces:

\begin{theorem}\label{DF-RegIm}
For any $*\in\{b,+,-,\emptyset\}$ the functor $f_*$ possesses a right derived functor
$$ Rf_*\! \colon \D^*(\QCohC(Y,\Coeff)) \longto \D^*(\QCohC(X,\Coeff)).$$
If $\FU$ is a finite affine cover of $X$, then $Rf_*$ can be defined as $\UCV^\bullet \mapsto  f_*\big(\Cech_\FU(\UCV^\bullet)\big)$.

If moreover $f$ is proper then $Rf_*$ restricts to a functor
\[
    Rf_* \colon \D^*_\coh(\QCohC(Y,\Coeff)) \longto \D^*_\coh(\QCohC(X,\Coeff))
\]
\end{theorem}
\begin{rem}
The fact that we may compute $Rf_*$ on $\kappa$-sheaves via a \v{C}ech resolution associated to a finite cover shows that we have for any complex $\UCV^\bullet \in \D^*_\coh(\QCohC(Y,\Coeff))$ a convergent spectral sequence of $\kappa$-sheaves
\[
    E_2^{i,j} = R^if_*(H^j\UCV^\bullet) \To R^{i+j}f_*\UCV^\bullet.
\]
\end{rem}
The usual functor $R(f\times\id)_*$ on complexes of quasi-coherent sheaves also defines a functor
\begin{eqnarray*}
\lefteqn{R(f\times\id)_*\colon \D^*(\QCoh(Y\times\CScheme))_\kappa\longto \D^*(\QCoh(X\times\CScheme))_\kappa},\\
&&(\CV^\bullet,\kappa_{\CV^\bullet}) \longmapsto  \big( R(f\times\id)_*\CV^\bullet, R(f\times\id)_*\big( \kappa_{\CV^\bullet} \big) \big) .
\end{eqnarray*}
which agrees with $Rf_*$ for $\kappa$-sheaves via the equivalence of categories in \autoref{CompDCI-EquivOfTwoCats} This means, for a $\kappa$-sheaf $\UCV$ on $Y$ we may compute $R^if_*\UCV$ as the pair consisting of the sheaf $R^i(f\times\id)_*\CV$ together with the map $R^i(f\times\id)_*(\tau_\CV)$. This is again a consequence of the observation that the same \v{C}ech resolution computes both, $Rf_*$ and $R(f \times \id)_*$:
\begin{theorem}\label{DF-DerImages}
There is a commutative diagram
\[
\xymatrix@C+2pc{
\D^*(\QCohC(Y,\Coeff)) \ar[d]^\simeq\ar[r]^{Rf_*}& \D^*(\QCohC(X,\Coeff))\ar[d]^\simeq\\
\D^*(\QCoh(Y\times\CScheme))_\kappa\ar[r] ^{R(f\times\id)_*} & \D^*(\QCoh(X\times\CScheme))_\kappa
}
\]
with the vertical equivalences from Theorem~\ref{CompDCI-EquivOfTwoCats}.
\end{theorem}

To prepare the passage of a functor from a derived category of Cartier sheaves to a derived category of Cartier crystals, we make the following definition.
\begin{defn}\label{DF-DefCohomNilp}
Complexes in $\D^*_\lnil(\QCoh(Y,\Coeff))_\kappa$ are called {\em cohomologically nil\-potent}.

A morphism of complexes $\UCV^\bullet\to\UCW^\bullet$ in $\D^*(\QCohC(Y,\Coeff))$ is called a \emph{nil-quasi-isomorphism} if its cone is cohomologically nilpotent. The set of all such morphisms is denoted~$\SLocQNilInv$.
\end{defn}\side{Apparently, this uses $F$-finiteness for the existence of $\QCrys$, but it might be possible to save this by enlarging $\LNilCohC$.}

$\D^*_\lnil(\QCohC(Y,\Coeff))$ is a full triangulated subcategory of $\D^*(\QCohC(Y,\Coeff))$.
By the general theory of such, see~\cite{Verdier.CatDer}, the nil-quasi-isomorphisms form a saturated multiplicative system of $\D^*(\QCohC(Y,\Coeff))$ and thus the localization $\SLocQNilInv \D^*(\QCohC(Y,\Coeff))$ is defined. It is natural to compare this localized category with the derived category obtained from the localization of $\QCohC$ at nil-isomorphisms. Clearly one has a homomorphism of derived categories
$$\D^*(\QCohC(Y,\Coeff))\longto \D^*(\QCrysC(Y,\Coeff))$$ which maps all nil-quasi-isomorphism to isomorphisms. Indeed one has the following formal result:
\begin{theorem}\label{DF-CompDCII}\side{Again, for the $\QCohC$ case one needs $\Fr$-finiteness. For $\IndCohC$ and $\LNilCohC$ one gets away without.}
For $*\in\{\emptyset,+,-,b\}$, the natural morphism
\[
\SLocQNilInv \D^*(\QCohC(Y,\Coeff))\to  \D^*(\QCrysC(Y,\Coeff))
\]
is an equivalence of categories. The analogous result holds if $(\QCohC,\QCrysC)$ is replaced by either of $(\IndCohC,\IndCrysC)$ or $(\LNilCohC,\LNilCrysC)$.
\end{theorem}
\begin{proof}
For quasi- and ind-objects, the first two parts of the theorem follows from  \cite[Thm.~2.6.2]{BoPi.CohomCrys} since by \autoref{CSh-GrothDirLim} the relevant categories are Grothendieck categories and the locally nilpotent ones a Serre subcategory closed under filtered direct limits.

To show the equivalence
\[
\SLocQNilInv \D^*(\QCohC(Y,\Coeff)) \to  \D^*(\LNilCrysC(Y,\Coeff))
\]
claimed in the third statement, one cannot appeal directly to \cite[Theorem~2.6.2]{BoPi.CohomCrys} since $\LNilCohC(Y,\Coeff)$ is not closed under infinite direct sums and hence not a Grothendieck category. Instead, for $*\in\{b,-,+\}$ the result is a special case of \cite[Thm.~3.2]{Miyachi}. It is a challenging exercise (left to the reader) to check that the argument of op.cit.\ can be extended to $*=\emptyset$
\end{proof}

Since localization at $\SLocQNilInv$ preserves cohomology up to local nilpotence, we obtain the following consequence:
\begin{cor}\label{DF-CorCompDCII}
For $*\in\{\emptyset,+,-,b\}$, the isomorphism of \autoref{DF-CompDCII} restricts to an equivalence of categories
\[
 \SLocQNilInv \D^*_\lncoh(\QCohC(Y,\Coeff))\to  \D^*_\crys(\QCrysC(Y,\Coeff)),
\]
where the index $\crys$ (resp $\lncoh$) indicates that the cohomology is in $\LNilCrysC$ (resp. $\LNilCohC$). The same result holds, if we either replace $(\lncoh,\crys)$ by $(\indcoh,\indcrys)$ or if we replace $(\QCohC,\QCrysC)$ by $(\IndCohC,\IndCrysC)$.
\end{cor}
\begin{proof}
We sketch the proof of the first statement and leave the analogous remaining statements to the reader. First one observes that the inclusion
\[
    \SLocQNilInv\D^*_\lncoh(\QCohC(Y,\Coeff)) \into \SLocQNilInv\D^*(\QCohC(Y,\Coeff))
\]
establishes the former as a full subcategory of the latter. For this note that $\D^*_\lncoh(\QCohC(Y,\Coeff))$ is a full subcategory of $\D^*(\QCohC(Y,\Coeff))$. Now let \[
\phi \colon \UCU^\bullet \Longleftarrow \UCW^\bullet \to \UCV^\bullet
\]
represent a morphism in $\SLocQNilInv\D^*(\QCohC(Y,\Coeff)$ with $\UCU^\bullet, \UCV^\bullet \in \D^*_\lncoh(\QCohC(Y,\Coeff)$. We have to show that $\UCW^\bullet$ is also in $\D^*_\lncoh(\QCohC(Y,\Coeff)$. But this follows easily since $\UCU^\bullet \Longleftarrow \UCW^\bullet$ is a nil-quasi-isomorphism, \ie its cone lies in $\LNilC$. The long exact cohomology sequence of the cone-triangle shows that the cohomology of $\UCW^\bullet$ is an extension of $\LNilC$ and $\LNilCohC$, hence also lies in $\LNilCohC$ since the latter is a Serre subcategory and hence closed under extensions by \autoref{CSh-LNilIsSerre}.

Now consider the diagram
\[
\xymatrix@C+2pc{
\SLocQNilInv\D^*(\QCohC(Y,\Coeff))\ar[r]^\simeq& \D^*(\QCrysC(Y,\Coeff))\\
\SLocQNilInv\D^*_\lncoh(\QCohC(Y,\Coeff)) \ar@{^(->}[u]\ar@{-->}[r]^\simeq & \D^*_\crys(\QCrysC(Y,\Coeff))\ar@{^(->}[u]
}
\]
where the upwards arrows are full subcategories and the top is the equivalence of \autoref{DF-CompDCII}. Since $\SLocNilInv \LNilCohC = \LNilCrysC$ it is now just an unraveling of definitions to verify that under the equivalence of the top row, the bottom two full subcategories correspond to each other, we leave this to the reader.
\end{proof}

\begin{rem}\label{DF-RemCompDCII}
In \cite{BliBoe.CartierDuality} we shall prove a deep refinement of \autoref{DF-CorCompDCII} for $*\in\{b,+\}$ and cohomology in $\coh$ on the left hand side.
\end{rem}

The relevance of \autoref{DF-CompDCII} and its corollary for the construction of functors on derived categories of crystals is explained by the following:
\begin{prop}\label{DF-ConstrDFonCrys}\g?{I have to check again whether this proposition and the above two results are sufficient for all construction below \\ ok for $Rf_*$.}
    Suppose $F\colon \D^*(\QCohC(Y,\Coeff)) \to \D^*(\QCohC(X,\Coeff))$ is an exact functor of triangulated categories, \eg a derived functor. If $F$ preserves cohomological nilpotence, i.e., if $F\big(\D^*_\lnil(\QCohC(Y,\Coeff))\big)\subset  \D^*_\lnil(\QCohC(X,\Coeff))$, then $F$ maps nil-quasi-isomorphism again to such and it induces an exact functor (denoted by the same name):
$$F\colon \D^*(\QCrysC(Y,\Coeff)) \to \D^*(\QCrysC(X,\Coeff)).$$
An analogous statement also holds if $(\QCohC,\QCrysC)$ is replaced by $(\IndCohC,\IndCrysC)$ or $(\LNilCohC,\CrysC)$.

Moreover both of the above assertions hold with additional conditions on the cohomology: If the given functor preserves cohomology in $\LNilCohC$, or $\IndCohC$, respectively, then the resulting functor preserves cohomology in $\crys$ or $\indcrys$.

\end{prop}
\begin{proof}
Consider the composition of the given functor $F$ with the localization functor $ \D^*(\QCohC(X,\Coeff)) \to \SLocQNilInv \D^*(\QCohC(X,\Coeff))$. By the hypothesis and the universal property of localization there exists a unique functor $$ \wt F:\SLocQNilInv \D^*(\QCohC(Y,\Coeff)) \to  \SLocQNilInv \D^*(\QCohC(X,\Coeff))$$
whose composite with the localization $ \D^*(\QCohC(Y,\Coeff)) \to \SLocQNilInv \D^*(\QCohC(Y,\Coeff))$ is equal to the above composition. The first assertion now follows by an application of \autoref{DF-CompDCII} to domain and range of the functor.
The same proof applies to $(\IndCohC,\IndCrysC)$ and $(\LNilCohC,\CrysC)$ instead of $(\QCohC,\QCrysC)$. For the remaining assertions one uses \autoref{DF-CorCompDCII} instead of \autoref{DF-CompDCII}.
\end{proof}

The following lemma is useful for verifying the conditions of the preceding \autoref{DF-ConstrDFonCrys}.

\begin{lem}\label{DF-LemmaOnNil}
Suppose $F=Rf$ is the right derived functor of a left exact functor
\[f \colon \QCohC(Y,\Coeff) \to \QCohC(X,\Coeff).\]
Let $\CA$ (resp. $\CA'$) be a Serre subcategory of $\QCohC(Y,\Coeff)$ (resp. $\QCohC(X,\Coeff)$), and let $\UCV^\bullet$ be a complex in $\QCohC(Y,\Coeff)$. Assume that
\begin{enumerate}
\item each $R^if=H^i(Rf)$ maps $\CA$ into the Serre subcategory $\CA'$, and
\item there exists a convergent spectral sequence for complexes in $\D^+(\QCohC(Y,\Coeff))$\footnote{At this point, we have not yet established that the categories $ \QCohC(\ldots)$ have enough injectives. So we cannot say anything general about existence or convergence.}
\[ E_2^{i,j} = R^if(H^j\UCV^\bullet) \To R^{i+j} f \UCV^\bullet. \]
\end{enumerate}
Then $F\big(\D^+_\CA(\QCohC(Y,\Coeff))\big)\subset \D^+_{\CA'}(\QCohC(X,\Coeff))$.

If moreover $F$ is of finite cohomological dimension and (b) holds for unbounded complexes, then the conclusion holds for any of $b$, $-$, $\emptyset$ in place of~$+$.
\end{lem}
\begin{proof}
The convergence of the spectral sequence (together with the boundedness of the cohomological dimension of $F$ if $\UCV^\bullet$ is unbounded) implies that the limit term $R^{n} f \UCV^\bullet$ is a finite successive extension of $\kappa$-subquotients of the $R^if(H^j\UCV^\bullet)$ for $i+j=n$. By assumption each of these is in $\CA$, hence the result follows since $\CA'$ is a Serre subcategory
\end{proof}

Our first application of the above lemma is to $Rf_*$.
\begin{cor}\label{DF-DerImAndNil}
Let $f\!:Y\to X$ be any morphism of noetherian schemes with $\Fr_X$ finite. Then $Rf_*\big(\D^*_\lnil(\QCohC(Y,\Coeff))\big)\subset \D^*_\lnil(\QCohC(X,\Coeff))$.

If moreover $f$ is proper, then $Rf_*\big(\D^*_\lncoh(\QCohC(Y,\Coeff))\big)\subset \D^*_\lncoh(\QCohC(X,\Coeff))$.
\end{cor}
\begin{proof}
As pointed out after \autoref{DF-RegIm} we have a convergent spectral sequence
\[
    E_2^{i,j} = R^if_*(H^j\UCV^\bullet) \To R^{i+j}f_*\UCV^\bullet.
\]
By \autoref{CSh-LNilIsSerre} the subcategory $\LNilC$ (resp. $\LNilCohC$) is a Serre subcategories of $\QCohC$. Hence, in order to conclude with the preceding \autoref{DF-LemmaOnNil} we only have to check that $R^if_*$ preserves $\LNilC$ (resp. $\LNilCohC$). For $\LNilC$, since $R^if_*$ commutes with direct limits, it is enough to show that $R^if_*$ preserves nilpotence. But this follows since $R^if_*$ is a functor: If $(\UCV,\kappa)$ is nilpotent, then $\Frid^n_*\CV \to[\kappa^n] \CV$ is zero for some $n$. This implies that $R^if_*(\kappa^n)$ is also zero, and hence $R^if_*\UCV$ is nilpotent.

For $\LNilCohC$, we note that for proper $f$ all $R^if_*$ also preserve coherence. Since, by definition a nil-coherent $\kappa$-sheaf is an extension of coherent and locally nilpotent, it follows that $R^if_*$ preserves nil-coherence.
\end{proof}

Combining \autoref{DF-ConstrDFonCrys} and \autoref{DF-DerImAndNil} we deduce:
\begin{cor}\label{DF-DerImOnCCrysI}
The functor $Rf_*$ on Cartier sheaves induces a functor
$$ Rf_*\colon \D^*(\QCrysC(Y,\Coeff))\longto \D^*(\QCrysC(X,\Coeff)).$$
If $f$ is proper, it restricts to a functor
$$ Rf_*\colon \D^*_\crys(\QCrysC(Y,\Coeff))\longto \D^*_\crys(\QCrysC(X,\Coeff)).$$
Recall, the subscript $\crys$ indicates that the cohomology lies in $\LNilCrysC$.
\end{cor}\label{DF-RkOnDerImOnCCrys}
\begin{rem}
The functor $Rf_*$ on quasi-crystals is in fact a derived functor: the \v{C}ech resolution also exists for Cartier sheaves, and hence for Cartier quasi-crystals. It is clearly functorial and thus extends to complexes. Therefore the analog of \autoref{DF-RegIm} holds for Cartier quasi-crystals as well. The explicit construction shows that this derived functor agrees with the functor constructed in \autoref{DF-DerImOnCCrysI} via \autoref{DF-CompDCII} and \autoref{DF-DerImAndNil}.
\end{rem}

For a general morphism $f \colon Y \to X$ push-forward $f_*$ does not preserve coherence. The following result shows that for Cartier sheaves and morphisms of finite type, morally the situation is better.
\begin{theorem}\label{DF-DerImOnLNilCoh}
Let $f\!:Y\to X$ be a morphism of finite type and $*\in\{b,-,+,\emptyset\}$. Then
$$ Rf_*\big(\D^*(\LNilCohC(Y,\Coeff)) \big)\subset \D^*(\LNilCohC(X,\Coeff)).
$$
\end{theorem}
\begin{proof}
The proof is an adaption of an argument of Anderson from \cite{AndersonL}.
The assertion is local, since $Rf_*$ may be defined via $f_*$ applied to \v{C}ech resolutions. Thus it suffices to prove the theorem for affine morphisms $f\colon \Spec S\to\Spec R$ of finite type given by $\phi \colon R\to S$ in the following two cases: (a) $f$ is a closed immersion, (b) $\phi$ is the canonical homomorphism $R\to R[X]$. Since now $f$ is affine, the functor $Rf_*$ has vanishing higher cohomology and thus in both cases we need to show that $$f_*\big(\LNilCohC(\Spec S,\Coeff) \big)\subset \LNilCohC(\Spec R,\Coeff).$$
Case (a) is trivial since any closed immersion of finite type is proper. To show (b) we need to verify the following
\begin{claim*}\label{claim1}
Let $(M,\kappa)$ be a finitely generated $R[X]\otimes\Coeff$-module and $\kappa\!:M\to M$ be $(\Fr_{R[X]}\otimes \id_\Coeff)^{-1}$-linear map giving $M$ the structure of a Cartier module. Then $M$ contains a finitely generated $R\otimes \Coeff$-submodule $N$ preserved under $\kappa$ such that $(M/N,\kappa)$ is a locally nilpotent $R[\Fr]\otimes\Coeff$-module.
\end{claim*}
On $R[X]$ we consider the filtration by the degree in $X$, i.e., we set
\[R[X]_{\le d}\colonequals \{f\in R[X]\mid \deg_Xf\le d\}.\]
Next we choose generators $m_1,\ldots,m_t$ of $M$ as an $R[X]\otimes \Coeff$-module. They yield a non-canonical exhaustive filtration
\[M_{\le d}\colonequals \sum_{i=1}^t R[X]_{\le d}\otimes\Coeff\cdot m_i\]
of $M$ by finitely generated $R\otimes \Coeff$-submodules. The key observation in  \cite{AndersonL} is that the Cartier semi-linear operator $\kappa$ on $M$ has a strong contractivity property which can be described as follows: Choose, typically non-unique, elements $a_{i,j,d}\in R[X]$ such that
$$\kappa(X^dm_i)=\sum_j a_{i,j,d} m_j$$
for each $d\in\{0,\ldots,q-1\}$ and $i\in\{1,\ldots,t\}$ and set
$$C\colonequals \max\{\deg_X a_{i,j,d} \mid i,j\in\{1,\ldots,t\},d\in\{0,\ldots,q-1\}\}.$$
\begin{claim*}\label{claim2}
For $\ell>Cq/(q-1)$ one has $\kappa(M_{\le \ell})\subset M_{\le \ell-1}$.
\end{claim*}
If this claim is shown then the first claim above is immediate for $N=M_{\le \ell_0}$ with $\ell_0=\lfloor Cq/(q-1)\rfloor$, the largest integer less or equal to~$Cq/(q-1)$.

To prove the claim, fix $\ell\ge Cq/(q-1)$. Then for any $d\le \ell$ we have $\frac{d}q +C\le \frac{\ell}q +C<\ell$. Let $d'\in\{0,\ldots,q-1\}$ be the remainder of $d$ by division by~$q$. For $i\in\{1,\ldots,t\}$ it follows that
\[
     \kappa( X^d m_i)=\kappa( X^{d-d'}X^{d'} m_i)=X^{\frac{d-d'}q} \sum_j a_{i,j,d'} m_j\in M_{\le \frac{d-d'}q+C}\subset M_{\le \ell-1}
\]
finishing the proof of the claim and thereby the proof of the theorem.
\end{proof}

\begin{cor}\label{DF-DerImOnIndCoh}
Let $f\!:Y\to X$ be essentially of finite type and $*\in\{b,-,+,\emptyset\}$. Then
$$
    Rf_*\big(\D^*(\IndCohC(Y,\Coeff)) \big)\subset \D^*(\IndCohC(X,\Coeff)).
$$
\end{cor}
\begin{proof}
As in the proof of \autoref{DF-DerImOnLNilCoh} the assertion is local, and so it suffices to prove it when $X=\Spec R$ and $Y=\Spec S$ are affine schemes. In that case we need to prove that $ f_*\big(\IndCohC(Y,\Coeff) \big)\subset \IndCohC(X,\Coeff)$. Since all ind-Cartier sheaves are direct limits of coherent Cartier sheaves, it suffices to show that
\begin{equation}\label{InclusionForDerImAndInd}
 f_*\big(\CohC(Y,\Coeff) \big)\subset \IndCohC(X,\Coeff).
\end{equation}
For $f$ of finite type this follows from \autoref{DF-DerImOnLNilCoh} due to the inclusions in \autoref{CSh-CCatDiag}. 
By \autoref{Mor-FactOfEFT}  any essentially finite type morphism $f$ factors as the composite of a finite type morphism and a localized immersion, and so it suffices to show \autoref{InclusionForDerImAndInd} for $\Spec S\to \Spec R$ induced from a localization homomorphism $\alpha\colon R\to S=\CS^{-1}R$, $\CS\subset R$ a multiplicative system.

Let $I$ be the collection of all multiplicatively closed subszstems $\CS_0$ of $\CS$ which are finitely generated. Let $(M,\kappa)$ be a finitely generated $S\otimes\Coeff$-module and $\kappa\!:M\to M$ be $(\Fr_{S}\otimes \id_\Coeff)^{-1}$-linear. We first construct a finitely generated $\CS_0^{-1}R\otimes \Coeff$-module $N_0\subset M$ such that $\kappa(N_0)\subset N_0$ for some $\CS_0\in I$: Consider a presentation
$$
\CS^{-1}R^m\otimes\Coeff\to[A] \CS^{-1}R^n\otimes\Coeff\to M \to 0.
$$
The coefficients of the presentation matrix $A$ involve only finitely many denominators. Thus it is defined already over some $\CS_1\in I$. Let $N_1$ denote the $\CS_1^{-1}R\otimes \Coeff$-module which is the cokernel of the map $A$ over $\CS_1^{-1}R\otimes\Coeff$. By enlarging $\CS_1$ if necessary, we may assume $N_1\subset M$. Since $\Fr_R$ is finite by hypothesis, $\Frid_* N_1$ is finitely generated over $\CS_1^{-1}R\otimes \Coeff$. Any finite set of $\CS_1^{-1}R\otimes \Coeff$-generators of $\Frid_* N_1$ is also a set of $S\otimes \Coeff$-generators of $\Frid_*M$. Hence there exists $\CS_0\in I$ containing $\CS_1$ such that $\kappa(N_1)\subset M$ is contained in $N_0\colonequals \CS_0^{-1}N_1$. From $\kappa(s^{-1}m)=\kappa(s^{-q} s^{q-1}m)=s^{-1}\kappa(s^{q-1}m)$ for $s\in S$, one deduces $\kappa(N_0)\subset N_0$, as desired.

Now observe that
\[
(M,\kappa) =\CS^{-1}(N_0,\kappa)=\dirlim_{\CS_0 \subseteq \CT \in I} \CT^{-1}(N_0,\kappa).
\]
By \autoref{DF-DerImOnLNilCoh}, the Cartier sheaves attached to $\CT^{-1}(N_0,\kappa)$ are locally nil-coherent. It follows that their direct limit $(M,\kappa)$ lies in $\IndCohC(\Spec R,\Coeff)$, and this completes the proof of the corollary.
\end{proof}
\begin{rem}
The corollary is not true if one replaces $\IndCohC$ by $\LNilCohC$ while keeping $f$ not of finite type, as the following simple example shows: let $k$ be an algebraically closed field, let $\alpha\!:k\to k(x)$ be the natural homomorphism and let $M=k(x)\mathrm{d}x$ be the module of differentials on $k(x)$ with the usual Cartier operator $C$ given by sending $x^ndx$ to $x^{(n+1)/p+1}$ and taking $p$th roots on $k$. Then all elements of the form $m_c\colonequals \frac{\mathrm{d}x}{x-c}$, $c\in k$ are fixed points of $C$:
\[
    C(\frac{dx}{x-c})=(x-c)^{-1}C((x-c)^{p-1}dx) = (x-c)^{-1}C(x^{p-1}dx)=\frac{dx}{x-c}
\]
using that $C(x^idx)=0$ for $0 \leq i \leq p-2$. In particular $C$ acts not nilpotently on the $m_c$'s. As $(m_c)_{c\in k}$ is linearly independent, $\alpha_*M$ is not an element of $\LNilCohC(\Spec k,k)$.\
\end{rem}

\begin{cor}\label{GenDefDerImQCohC}
Suppose $*\in\{b,+,-,\emptyset\}$ and $?\in\{\lnil,\lncoh,\indcoh\}$. If $f$ is of finite type, then
$Rf_*\big(\D^*_?(\QCohC(Y,\Coeff)) \big)\subset \D^*_?(\QCohC(X,\Coeff)) $. If $f$ is essentially of finite type, the conclusion still holds for~$?=\ind$.
\end{cor}
\begin{proof}
Since $f$ is of finite type, the derived functor $Rf_*$ has finite cohomological dimension. This ensures the convergence of the spectral sequence in \autoref{DF-LemmaOnNil} for any choice of~$*$. Condition (a) of \autoref{DF-LemmaOnNil} was established in \autoref{DF-DerImAndNil}, \autoref{DF-DerImOnLNilCoh} and \autoref{DF-DerImOnIndCoh}, respectively. Therefore the proof is an application of~\autoref{DF-LemmaOnNil}.
\end{proof}

\autoref{GenDefDerImQCohC} together with \autoref{DF-CompDCII} and \autoref{DF-ConstrDFonCrys} directly yield:
\begin{theorem}\label{GenDefDerIm}
Suppose $*\in\{b,+,-,\emptyset\}$. If $f$ is of finite type, then
\[
\xymatrix@C+2pc@R-.5pc{
\D^*(\LNilCrysC(Y,\Coeff)) \ar@{^{ (}->}[d] \ar[r]^{Rf_*}& \D^*(\LNilCrysC(X,\Coeff))\ar@{^{ (}->}[d]\\
\D^*_?(\IndCrysC(Y,\Coeff)) \ar@{^{ (}->}[d] \ar[r]^{Rf_*}& \D^*_?(\IndCrysC(X,\Coeff))\ar@{^{ (}->}[d]\\
\D^*_?(\QCrysC(Y,\Coeff)) \ar[r]^{Rf_*}& \D^*_?(\QCrysC(X,\Coeff))\\
}
\]
commutes for $?\in\{\crys,\emptyset\}$. If $f$ is essentially of finite type, the lower square commutes for $?\in\{\ind,\emptyset\}$.
Moreover via the equivalence $\CrysC(Y,\Coeff) \to[\simeq] \LNilCrysC(Y,\Coeff)$ from  \autoref{CSh-Crys=NilCrys}, for $f$ of finite type the first row induces a functor
\[
\xymatrix@C+2pc@R-.5pc{
\D^*(\CrysC(Y,\Coeff))  \ar[r]^{Rf_*}& \D^*(\CrysC(X,\Coeff))\\
}
\]
\end{theorem}

\side{A reasonably thorough reading up to here seems to indicate that we need $\Fr$-finite whenever we speak of $\QCrysC$. However I have the feeling that this can be resolved by replacing $\LNilCohC$ by the smallest Serre subcategory $\overline{\LNilCohC}$ of $\QCohC$ it is contained in. One still has $\LNilCohC = \overline{\LNilCohC} \cap \IndCohC$. To make this work we have to check, that $R^if_*$ preserves $\overline{\LNilCohC}$ which should follow by the long exact cohomology sequence and the fact that $R^if_*$ preserves $\LNilCohC$.\\
However here are many details that need to be checked and a lot of proofs need to be adjusted also down the row. Presently I don't think this is worth it. And I am more than happy to add $\sigma$-finiteness as a blanket assumption for the whole paper. }

\begin{rem}
Regarding the vertical arrows in \autoref{GenDefDerIm}, it is shown in \cite[Reference !]{BliBoe.CartierDuality} that for a regular scheme with finite $\Fr_X$ the natural inclusion
\[
    \D^*_?(\IndCrysC(X,\Coeff))\into \D^*_?(\QCrysC(X,\Coeff))
\]
is an equivalences of categories for $*\in \{\emptyset,+,-,b\}$, $?\in\{\indcrys,\crys\}$. For $*\in \{-,b\}$ the same holds for
\[
    \D^*(\LNilCrysC(X,\Coeff))\into \D^*_\crys(\IndCrysC(X,\Coeff)).
\]
This can be combined with the equivalence $\CrysC(X,\Coeff) \to[\simeq] \LNilCrysC(X,\Coeff)$ to obtain an equivalence of categories $\D^{-/b}(\CrysC(X,\Coeff)) \to[\simeq]   \D^{-/b}_\crys(\QCrysC(X,\Coeff))$. This observation can be used to define for proper $f \colon Y \to X$ and $X$ regular the pushforward $Rf_* \colon \D^{-/b}(\CrysC(Y,\Coeff)) \to \D^{-/b}(\CrysC(X,\Coeff))$ without appealing to \autoref{DF-DerImOnIndCoh}. This yields an alternative construction of $Rf_*$ of \autoref{GenDefDerIm} in this case.

It is unclear to us if the regularity assumption in the equivalence of categories above is really necessary; in the context of $\tau$-crystals it is not, as shown in \cite[Chapter 5]{BoPi.CohomCrys}.
\end{rem}

\subsection{Twisted inverse image}
\label{CSh-sec-InverseImage}

We first recall the key properties of the twisted inverse image functor for quasi-coherent sheaves $f^!$ for \emph{essentially} finite type morphisms from \cite[Cor.~VII.3.4]{HartshorneRD}, \cite[Thm.~4.8.1]{LipmanGrothDual} and  \cite[5.2]{Na.CompEFT}. Let $\FC$ be the category of Noetherian schemes with essentially finite type morphisms (see \autoref{sec.EssFT}).
\begin{theorem}\label{CSh-PropFShriek}
On the category $\FC$, there exists a pseudofunctor $(\usc)^!$ which to each $f\colon Y\to X$ in $\FC$ assigns a functor
\[f^!\colon \D^+(\QCoh(X))\longto \D^+(\QCoh(Y))\]
that is uniquely determined up to isomorphism by the following three properties :
\begin{enumerate}
\item On the subcategory of Noetherian schemes with essentially \'etale \hbox{morphisms}, $f\mapsto f^!$ is naturally isomorphic to the inverse image pseudofunctor~$f\mapsto f^*$.
\item On the subcategory  of Noetherian schemes with proper morphisms, $f\!\mapsto \!f^!$ is na\-turally isomorphic to the right adjoint of the direct image \hbox{pseudofunctor $\!f\!\mapsto \!Rf_*$}.
\item \label{CSh-PropFShriek-d}Functorially over~$\FC$, for any base change diagram of Noetherian schemes
\[
    \xymatrix{U\ar[r]^j \ar[d]^g &X\ar[d]^f\\
V\ar[r]^i &Y\\}
\]
with $f$ proper and $i$ flat there is a flat base change isomorphism $j^*f^!\stackrel{\sim}{\longto} g^!i^*$. If in addition $i$ and $j$ are essentially \'etale, then the flat base change isomorphism agrees with the natural isomorphism
\[
    j^*f^!\stackrel{\sim}{\longto}  j^!f^! \stackrel{\sim}{\longto} (fj)^!=(ig)^!\stackrel{\sim}{\longto} g^!i^!\stackrel{\sim}{\longto} g^!i^*.
\]
\end{enumerate}
On the subcategory $\FC_S^p$ of Noetherian schemes essentially of finite type over a base scheme $S$ that possesses a dualizing complex $\omega_S^\bullet$ and where all morphisms are proper, one has the following additional results:
\begin{enumerate}\setcounter{enumi}{3}
\item Let $\Dual_X$ denote the dualizing functor on $X\in\FC_S^p$ induced from $\omega_S^\bullet$, then $f\mapsto Rf_*$ and $f\mapsto \Dual_Y\circ \,Rf_*\circ \Dual_X$ are naturally isomorphic pseudo-functors.
\item The adjoint pair $(Lf^*,Rf_*)$ yields an adjoint pair $(\Dual_X\!\circ Rf_*\circ\Dual_Y,\Dual_Y \!\circ Lf^*\circ\Dual_X)$ and in particular $f\mapsto f^!$ and $f\mapsto \Dual_Y \circ\, Lf^*\circ\Dual_X$  are naturally isomorphic pseudofunctors.
\item
For each $f\colon Y\to X$ in $\FC_S$ one has $f^!\big(D^+_\coh(\QCoh(X))\big)\subset D^+_\coh(\QCoh(Y))$.
\end{enumerate}
\end{theorem}
\begin{rem}\label{RemOnfShriek}
The abstract existence result of $f^!$ in (b) is attributed by Deligne to Verdier, see~\cite[App.~IV]{HartshorneRD}. The construction of $f^!$ for general $f$ proceeds as follows: By Nagata's compactification theorem, any finite type morphism $f$ can be factored as the composite $f=\bar f\circ j$ of an open immersion $j$ and a proper morphism $\bar f$. The extension by Nayak recalled in \autoref{Mor-FactOfEFT}, yields an analogous factorization for essentially finite type morphisms. One then verifies that $f^!=j^!\circ\bar f^!$ using (a) and (b). In \cite{HartshorneRD} $f^!$ is actually defined via such a factorization and a great deal of work goes into showing that this is independent of the factorization.

In the presence of a compatible system of dualizing complexes, as in $\FC^p_S$, the approach indicated in (d) and (e) gives an alternative construction to $f^!$ for proper morphisms. This was carried out in \cite{HartshorneRD} based on notes by Grothendieck. The compatibility of the two approaches is explained in \cite[App.~4.10]{LipmanGrothDual}.

Since an adjoint can only be defined up to natural isomorphisms of pseudofunctors, in a setting as in (d)-(f), one may simply take $f^!=\Dual_Y \circ \, Lf^*\circ\Dual_X$ as the definition.
\end{rem}

In general there is no explicit description of the functor $f^!$, however in many important cases the situation is better:

\begin{ex}
If $f$ is a \emph{finite} morphism then $f^!(\ublank)\cong f^{-1}\CRHom(f_*\CO_X,\ublank)$; in particular $f^!$ is a right derived functor in this case. If in addition $f$ is flat, then $f^!(\ublank)\cong f^{-1}\CHom(f_*\CO_X,\ublank)$ and one can verify the identity $f^!(\ublank) \cong f^! \CO_X \tensor f^*(\ublank)$ by hand, \cf \cite[Lemma 5.7]{BliBoe.CartierFiniteness}.
\end{ex}

More generally, in the case of a smooth morphism $f \colon Y \to X$ one has the explicit relationship between $f^!$ and $Lf^*$ given by $f^!(\usc) = f^!\CO_X \Ltensor_{\CO_Y} Lf^*(\usc)$, and $f^!\CO_X=\omega_f[-d]$ is the relative dualizing sheaf for $f$ placed in degree $d=$ relative dimension of $f$. This can be generalized as follows:

\begin{rem}\label{EssPerfBounded}
A morphism $f\colon Y\to X$ is called {\em essentially perfect} if it is essentially of finite type, if $Rf_*\CO_Y\in\D^b_\coh(\QCoh(X))$ and if there exist integers $m\le n$ such that for all $y\in Y$ the module $\CO_{Y,y}$ is quasi-isomorphic to a complex of flat $\CO_{X,f(y)}$-modules concentrated in degrees between $m$ and~$n$. This notion is a relative version of being of finite $\Tor$-dimension, \cf~\cite[II.4]{HartshorneRD}. As shown in \cite{Na.CompEFT} the key points are:
\begin{enumerate}
\item A morphism $f$ is essentially perfect if and only if $f^!\CO_X$ is bounded and there is a natural isomorphism of functors
\begin{equation}\label{ExplicitUpperShriek}
f^!(\ublank) \cong f^!\CO_X \Ltensor_{\CO_Y} Lf^*(\ublank).
\end{equation}
\item If $f$ is essentially perfect, then $f^!$ is bounded.
\end{enumerate}
The isomorphism \autoref{ExplicitUpperShriek} is obtained as follows: First assume that $f$ is proper, then $Rf_*\big( Lf^*(\ublank)  \Ltensor_{\CO_Y}  f^!\CO_X \big) \to[\mathrm{Proj.\,formula}]  (\ublank) \Ltensor_{\CO_X}  Rf_* f^!\CO_X   \to[\trace_f] (\ublank)$ is just the adjoint of \autoref{ExplicitUpperShriek} via the adjunction $(Rf_*,f^!)$. In general one uses the factorization of $f$ into $h \circ i$ where $i$ is essentially \'etale and $h$ is proper, see \cite[Section 5.7]{Na.CompEFT} for details.

Since $f$ factors as $f = h \circ i$ where $i$ is localizing and $h$ is perfect, and $i^!=i^*$ is exact, we may assume that $f$ is perfect. This means $\CO_X$ is $f$-perfect, \ie there is $m \in \mathbb{N}$ such that for all $x \in X$ one has that $\CO_{X,x}$ has a flat $\CO_{Y,f(x)}$-resolution of length less than $m$. This immediately implies that $Lf^*$ is bounded. Hence by the formula \autoref{ExplicitUpperShriek} and the fact that for essentially perfect $f$ one has $f^!\CO_Y$ is $f$-prefect \cite[Theorem 5.9]{Na.CompEFT}, \cite[Theorem 4.9.4]{LipmanGrothDual} it follows that $f$ is bounded.
\end{rem}
\begin{ex}
Examples for essentially perfect morphisms are \cite{Na.CompEFT}:
\begin{enumerate}
\item Any locally complete intersection morphism, as can be seen from the Koszul complex.
\item Any Cohen-Macaulay morphism $f$: Such $f$ are flat by definition, so that $Lf^*=f^*$. Moreover they satisfy $f^!\CO_Y\cong\omega_f[-d]$, where $\omega_f$ is the dualizing sheaf relative to $f$ and $d$ is the relative dimension (locally constant on $X$). Special cases are Gorenstein morphisms, where in addition $\omega_f$ locally free of rank~$1$, and essentially smooth morphism in which case $\omega_f=\bigwedge^d \Omega_{Y/X}$ for $\Omega_{Y/X}$ is the relative sheaf of differentials.
\item Any finite type morphism $f$ with $X$, $Y$ regular; this will be discussed in detail in \autoref{LemOnFTMorOfRegSchemes} below.
\end{enumerate}
\end{ex}

In order to show that the functor $f^!$ descends to $\kappa$-sheaves we use the duality for a finite morphism applied to the Frobenius $\Frid$. This states that the functors $\Frid_*$ and $\Frid^!$ form an adjoint pair of functors. Observing \autoref{CSh-PropFShriek}, it says that in the derived category $\D^+_?(\QCoh(X\times\CScheme))$ giving a map $\kappa \colon \Frid_*\CV^\bullet \to \CV^\bullet$ is equivalent to giving a map $\wt{\kappa}:\CV^\bullet \to \Frid^! \CV^\bullet$. To state this observation, let us define (ad hoc) the category $\D^{+/b}_?(\QCoh(X \times \CScheme))_{\wt{\kappa}}$ as the category whose objects are pairs $(\CV^\bullet, \wt{\kappa})$ with $\CV^\bullet \in \D^{+/b}_?(\QCoh(X \times \CScheme))$ and $\wt{\kappa} \colon \CV^\bullet \to \Frid^! \CV^\bullet$. The homomorphisms are the ones which commute with $\wt{\kappa}$.
\begin{lem}\label{LemOnFrobAdj}
The derived category $\D^{+/b}_?(\QCoh(X \times \CScheme))_\kappa$ is equivalent to $\D^{+/b}_?(\QCoh(X \times \CScheme))_{\wt{\kappa}}$.  The equivalence is given by sending $(\CV^\bullet,\kappa) \in \D^+_?(\QCoh(X \times \CScheme))_\kappa$ to $(\CV^\bullet, \wt\kappa)$ where $\wt\kappa$ and $\kappa$ correspond to one another via the adjunction (duality of finite maps)
\[
    \Hom_{\D^+}(\CV^\bullet,\Frid^!\CV^\bullet) \cong \Hom_{\D^+}(\Frid_*\CV^\bullet,\CV^\bullet).
\]
recalled in \autoref{CSh-PropFShriek}.
\end{lem}
It is clear how for $f \colon Y \to X$ essentially of finite type the functor $(f \times \id)^!$ from \autoref{CSh-PropFShriek} induces a functor on $\D^+_?(\QCoh(X \times \CScheme))_{\wt{\kappa}}$, just by the compatibility of $(\usc)^!$ with composition.  By the preceding \autoref{LemOnFrobAdj} this induces an $(f \times \id)^!$ on $\D^+_?(\QCoh(X \times \CScheme))_\kappa$ for $?\in\{\emptyset,\coh\}$. Concretely, given an element $\UCV^\bullet \in \D^+_?(\QCoh(X \times \CScheme))_\kappa$ with its adjoint  $\wt\kappa \colon \CV^\bullet \to \Frid^! \CV^\bullet$ we apply $(f \times \id)^!$ and obtain a functorial morphism
\begin{equation}\label{ShriekDef0}
    (f\times \id)^!\CV^\bullet \to[(f\times \id)^!(\wt{\kappa})] (f\times \id)^!\Frid^!\CV^\bullet \cong \Frid^!((f\times \id)^!\CV^\bullet) .
\end{equation}
This naturally gives $(f \times \id)^! \CV^\bullet$ the structure of an element of $\D^+_?(\QCoh(X \times \CScheme))_\kappa$ and hence defines a functor
\begin{equation}\label{DefFshriekOnD()kappa}
    f^!\colon \D^+_?(\QCoh(X\times\CScheme))_\kappa \longto \D^+_?(\QCoh(Y\times\CScheme))_\kappa .
\end{equation}
\begin{defn}\label{DF-ShriekDef}
From \autoref{DefFshriekOnD()kappa}, for $?\in\{\coh,\emptyset\}$\footnote{In \autoref{CSh-FShriekNilpotenceForEssFinite} we shall prove the same for $?\in\{\lnil,\lncoh\}$ if $f$ is essentially of finite type.}, we define the functor
$$f^!\colon \D^{+}_?(\QCohC(X,\Coeff))\longto \D^{+}_?(\QCohC(Y,\Coeff))$$
via the equivalence $\D^{+}_?(\QCohC(\ublank,\Coeff))\to  \D^+_?(\QCoh(\ublank\times\CScheme))_\kappa$ of \autoref{CompDCI-EquivOfTwoCats}.
\end{defn}

If $f^!$ has bounded cohomological dimension, then by viewing the datum of (\ref{ShriekDef0}) as an object $\CV^\bullet \in \D^b_?(\QCoh(Y \times \CScheme))_\kappa$ one easily verifies that $f^!\CV^\bullet$ lies in $\D^b_?(\ldots)_\kappa$. In particular, by \autoref{EssPerfBounded}, we obtain:

\begin{prop}\label{ShriekForRegXY}\label{RemOnFShriekBounded}
Suppose $f\colon Y\to X$ is essentially perfect. Then for $?\in\{\coh,\emptyset\}$ one has
$$f^!\big( \D^{b}_?(\QCohC(X,\Coeff))\big) \subset \D^{b}_?(\QCohC(Y,\Coeff)).$$
This holds in particular if $X$ and $Y$ are regular.
\end{prop}
\begin{rem}\label{FirstRemOnfShriekCohomBdd}
For general essentially finite $f$ one cannot expect that $f^!$ is bounded, since this already fails for a closed immersion $\Spec R/I \into \Spec R$ if $R/I$ is not of finite projective dimension over $R$. However, in \autoref{CohDimOfShriek} we shall prove that the induced functor $f^!$ on complexes of crystals is bounded. In light of a duality theory that is the content of \cite{BliBoe.CartierDuality} such a boundedness result is expected. The duality relates $\kappa$-crystals to $\tau$-crystals of \cite{BoPi.CohomCrys} and the functor corresponding to $f^!$ is $Lf^*$ which is shown to be exact in \cite[Prop.~6.1.5, Thm.~5.3.1(c)]{BoPi.CohomCrys}.
\end{rem}

\begin{rem}\label{DF-RemShriekDef}
In \autoref{ShriekDef0} we defined the structural map for $f^!$ by a twofold application of the adjunction of pairs $(\CV^\bullet,\kappa)$ and $(\CV^\bullet,\wt\kappa)$. If one goes through the definitions, then the sheaf underlying $f^!\UCV^\bullet$ is $(f\times\id)^!\CV^\bullet\in \D^+_\coh(\QCoh(Y,\Coeff))$ and the homomorphism $\kappa_{f^!\CV^\bullet}$ is defined by the following commutative diagram:
 \begin{equation}\label{KappafShriek}
  \xymatrix@C+20pt @R-.8pc{
  \Frid_* (f\times\id)^! \CV^\bullet
  \ar[dd]^-{\kappa_{\scriptstyle f^!\CV^\bullet}} \ar[r]_-{\id\to\Fr^!\Fr_*}&
  \Frid_* (f\times\id)^! \Frid^!\Frid_*\CV^\bullet\ar[d]^\cong\\
  & \Frid_* \Frid^!
  (f\times\id)^!\Frid_*\CV^\bullet\ar[d]^{\Fr_*\Fr^!\to\id} \\
  (f\times\id)^!\CV^\bullet
  &(f\times\id)^!\Frid_* \CV^\bullet\ar[l]_{(f\times\id)^!(\kappa)}\rlap{.}\\}
  \end{equation}
For later use, we call the map from top left to bottom right a base change map:
\begin{equation}\label{DefSigmaFbasechange}
 \Frid_* (f\times\id)^! \to[\mathrm{base\,change}] (f\times\id)^!\Frid_* .
\end{equation}
\end{rem}

\begin{rem} \m?{Expanded this remark to explain the base change morphism in this case}
For finite maps $f$ this base change morphism \autoref{DefSigmaFbasechange} can be made more explicit as follows. Let us first assume that $X= \Spec R$ and $Y = \Spec S$ are affine, and $f$ is given by a finite map of rings $f \colon R \to S$. To
\[
    \xymatrix{
        R \ar[r]^{\Fr_R}\ar[d]_{f} & \Fr_*R \ar[d]^{f} \\
        S \ar[r]^{\Fr_S} & \Fr_*S }
\]
and a single Cartier module $V$ over $R$ the base change morphism \autoref{DefSigmaFbasechange} on zeroth cohomology is given by the composition
\begin{equation*}
\begin{split}
     \Hom_{\Fr_*R}(f_*\Fr_*S,V) & \to \Hom_{\Fr_*R}(f_*\Fr_*S,\Hom_R(\Fr_*R,\Fr_*V)) \to \Hom_R(f_*\Fr_*S,\Fr_*V)  \\ &\to \Hom_S(\Fr_*S,\Hom_R(f_*S,\Fr_*V)) \to \Hom_R(f_*S,\Fr_*V)\, ,
\end{split}
\end{equation*}

Disregarding the first map for a moment, the middle two steps are given by $\Hom$-$\tensor$ adjunction and express the commutation of $f^!\Fr^! \cong \Fr^! f^!$ and the last is just evaluation at 1. But tracing through these last three steps one observes that their composition is just given by precomposition with the Frobenius $\Fr_S$ and post-composition with the evaluation at 1 on $\Hom_R(\Fr_*R,\Fr_*V) \to \Fr_*V$.\footnote{The second map is given by $\psi \mapsto [ \alpha \mapsto (\psi(\alpha))(1)]$. The third map by $\phi \mapsto [\gamma \mapsto [\beta \mapsto \phi(\gamma\Fr_S(\beta))]]$ and the third is evaluation at $1$. Combining these we get that $\psi$ under this composition is a map that sends $s$ to $\psi(\Fr_S(s))(1)=ev_1 \circ \psi \circ \Fr_S(s)$ as claimed.} But the first map is just given by post-composition with the natural map $V \to \Hom_R(\Fr_*R,\Fr_*V)$ which composed with the evaluation at $1$ is the identity. It follows that the whole composition of all four maps is just given by precomposition with $\Fr_S$.
This explicit computation in the affine case shows that more generally, for any finite map $f \colon Y \to X$ and a complex of Cartier sheaves $\UCV^\bullet$ the base change
\[
    \Fr_* f^! \UCV^\bullet = \Fr_* f^{-1} \CRHom(f_*\CO_X,\UCV^\bullet) \to f^{-1}\CRHom(f_*\CO_X,\Fr_*\UCV^\bullet)=f^! \Fr_* \UCV^\bullet
\]
is given on underlying $\CO_X$-modules and after choosing some injective resolution $\CI^\bullet$ of $\CV^\bullet$ by the composition
\[
    \sigma_*f^{-1}\CHom(f_*\CO_X,\CI^\bullet) \to f^{-1}\CHom(f_*\Fr_*\CO_X,\sigma_*\CI^\bullet) \to f^{-1}\CHom(f_*\CO_X,\Fr_*\CI^\bullet)
\]
where the first map sends $\phi$ to $\sigma_*\phi$ and the second is composition with the Frobenius action $\Fr_X$ on $\CO_X$. Therefore, the extraordinary pullback $f^!\UCV^\bullet$ is given by the pair consisting of the complex of sheaves
$$ (f\times\id)^{-1}\CRHom((f\times\id)_*\CO_{X\times \CScheme},\CV^\bullet) $$
and the morphism of complexes
\begin{eqnarray*}
\lefteqn{\Frid_* (f\times\id)^{-1}\CRHom((f\times\id)_*\CO_{X\times \CScheme},\CV^\bullet) }\\
&\to &  (f\times\id)^{-1}\CRHom((f\times\id)_*\Frid_*\CO_{X\times C},\Frid_*\CV^\bullet) \\
&\to &  (f\times\id)^{-1}\CRHom((f\times\id)_*\CO_{X\times \CScheme},\CV^\bullet),
\end{eqnarray*}
where the second morphism uses on the right argument the structural homomorphism $\kappa_\CV\colon\Frid_* \CV^\bullet\to\CV^\bullet$ and on the left simply the Frobenius morphism $\Frid \colon \CO_{X \times C} \to \Frid_*\CO_{X \times C}$.
\end{rem}

\begin{ex}
If $i \colon Y \into X$ is a hypersurface in affine $n$-space given by the vanishing of a regular element $f \in R=k[x_1,\ldots,x_n]$.    The Koszul complex $K_\bullet(f,R) = 0 \to R \to[f] R \to 0$ is a $R$-free resolution of $S= R/f$. The Frobenius on $R/f$ lifts uniquely to the Koszul complex (homological degree on the right is zero)
\[
\xymatrix{
    0 \ar[r] & R \ar[r]^f \ar[d]_{f^{p-1}\Fr} & R  \ar[d]^\Fr\\
    0 \ar[r] & \Fr_*R \ar[r]^{\Fr_*f}  & \Fr_*R
    }
\]
yielding a map $\tau \colon K_\bullet(f,R) \to \Fr_* K_\bullet(f,R)$ which is compatible with the Frobenius action $\sigma \colon R/f \to \Fr_* R/f$.  Using this free resolution of $R/f$ we compute that for any Cartier module $(M,\kappa)$ the restriction $i^!M$ is represented by the complex $\Hom_R(K_\bullet(f,R),M)$. As by the above recipe we may use the lift $\tau$ of the Frobenius to the Koszul complex to compute that the structural map on $i^!M$ is given by
\[
\Fr_* \Hom_R(K_\bullet(f,R),M) \to \Hom_R(\Fr_* K_\bullet(f,R),\Fr_* M) \to[\kappa \circ \usc \circ \tau] \Hom_R(K_\bullet(f,R),M).
\]
Concretely, identifying $i^!M = \Hom_R(K_\bullet(f,R),M)$ with the complex $M \to[f] M$ via evaluation at $1$, the structural map $i^!\kappa$ is given by
\[
\xymatrix{
     \Fr_* M \ar[r]^{\Fr_*f}\ar[d]_\kappa & \Fr_* M\ar[d]^{\kappa f^{p-1}} \\
      M \ar[r]^{f} &  M \\
}
\]
More generally, a similar description can be obtained for any complete intesection $Y \into X$ by noting that the Koszul complex $K(f_1,\ldots,f_t,R)$ on a regular sequence is just the tensor product of the Koszul complexes $K(f_i,R)$ for $i= 1,\ldots, t$. In this case, if the Cartier module $(M,\kappa)$ is locally free, then $i^!M$ is equal to $(M/(f_1,\ldots,f_t)M)[-t]$ and the structural map is given by $\kappa((f_1\cdot \ldots \cdot f_t)^{p-1} \cdot \usc)$.
\end{ex}

Next we establish that the previously defined functor $j^*$ for essentially \'etale morphisms $j \colon U \to X$ agrees with $j^!$ here.
\begin{prop}\label{jUpperShriekIsjUpperStar}
Suppose $j \colon Y \to X$ is essentially \'etale, then $j^*$ as defined in \autoref{CSh-EtalePullback} is isomorphic to $j^!$ as defined above.
\end{prop}
\begin{proof}
This comes down to verifying that the flat base change isomorphism $(j \times \id)^*\Frid_* \to\Frid_*(j \times \id)^*$ is inverse to the base change isomorphism $\Frid_*(j \times \id)^!  \to[\simeq] (j \times \id)^! \Frid_*$ from \autoref{DefSigmaFbasechange} under the functorial isomorphism $(j\times \id)^* \to[\simeq] (j \times \id)^!$ of \autoref{CSh-PropFShriek}. To ease notation we omit $\times \id$ until the end of this proof. The two morphisms are the top and bottom row of the diagram
\[
\xymatrix{
j^*\Fr_* \ar[r]\ar[dd]^\cong & \Fr_*\Fr^*j^*\Fr_* \ar[r]^\cong & \Fr_*j^*\Fr^*\Fr_* \ar[r]\ar[d] &\Fr_*j^*\ar[dd]^\cong\ar[ld] \\
&\Fr_*\Fr^!j^*\Fr_* \ar[lu]\ar[d]^\cong &\Fr_*j^*\Fr^!\Fr_* \ar[l]\ar[d]^\cong & \\
j^!\Fr_* & \Fr_*\Fr^!j^!\Fr_* \ar[l] & \Fr_*j^!\Fr^!\Fr_* \ar[l]^\cong &\Fr_*j^! \ar[l]
}
\]
By \autoref{CSh-PropFShriek}, the vertical isomorphisms are just the functorial identifications of $j^* \to j^!$ and hence the bottom parts of the diagram commute. Therefore one has to check that the composition around the top-left pentagon starting at the top-left vertex is the identity. This will be done locally where one can write down all maps explicitly: let $j:R \to S$ be an essentially \'etale map of rings. Denote $\Fr \colon R \to R'$ the Frobenius and similarly for $S$. We have by \autoref{CSh-FrobEtale} that $S \tensor_R R' \cong S'$ which is needed to identify the bottom isomorphism below as given by sending $1 \tensor \phi$ to $\id_S \tensor \phi$. With $M=\Fr_*N$ an $R$-module we can now make the top left diagram more explicit:
\[
\xymatrix{
S \tensor_R M \ar[r] & \Fr_*(S' \tensor_S S \tensor_R M) \ar[r]^\cong &\Fr_*(S'\tensor_{R'} R' \tensor_R M) \ar[d] \\
& \Fr_*\Hom_S(S \tensor_R R',S \tensor_R M) \ar[lu]^{\operatorname{ev}_1} & \Fr_* (S'\tensor_{R'} \Hom_R(R',M)) \ar[l]_-{\cong}
}
\]
Since the maps are $S$-linear, it is enough to check on generators, \ie on elements of the form $1 \tensor m$, that the total composite is the identity:
\[
    1 \tensor m \mapsto 1 \tensor 1 \tensor m \mapsto 1 \tensor 1 \tensor m \mapsto 1 \tensor [\phi_m \colon r \mapsto rm] \mapsto \id_{S} \tensor \phi_m \mapsto 1 \tensor m
\]
This shows our claim.
\end{proof}
\begin{rem}
The compatibility of the two base changes for the Frobenius under an essentially \'etale morphism is a consequence of a more general result about cartesian squares: Suresh Nayak explained to us \cite{Nayak.letter} that for any cartesian square
\[
\xymatrix{
    Y' \ar[r]^{j'} \ar[d]_{s'} & Y \ar[d]^{s} \\
    X' \ar[r]^{j} & X
}
\]
with $j$ essentially \'etale and $s$ proper, the composition
\[
   j^*s_* \to s'_*{s'}^*j^*s_* \to s'_*{j'}^*{s}^*s_* \to s'_*{j'}^* \to s'_*{j'}^*s^!s_* \to s'_*{s'}^!j^*s_* \to j^*s_*
\]
is the identity. Here composition of the first three maps is the base change under $(\usc)^*$ and the composition of the last three maps the base change under $(\usc)^!$ after the identification of $j^*$ with $j^!$ by \autoref{CSh-PropFShriek}.
\end{rem}

Our next goal is to show that $f^!$ preserves cohomological nilpotence. Since $f^!$ was not constructed as a derived functor, for arbitrary essentially finite type $f$ this does not follow from a spectral sequence as was the case with $Rf_*$ before. One way to bypass this problem, is to factor $f$ in a convenient way and reduce thereby to cases which can either be handled directly or by a spectral  sequence argument. To prepare for this, we give some basic results on injective $\kappa$-sheaves and on adjunction. Let us point out, that in  \autoref{sec.Stalks} we will develop a theory of stalks. This will an give alternative proof of cohomological nilpotence by reduction to points. Then morphisms correspond to inclusion of fields, and cohomological nilpotence will be obvious.

\paragraph{Adjunction for $(j^*,Rj_*)$ and injective $\kappa$-sheaves}

\begin{prop}\label{DF-UpperStar-Adj-New}
Let $j\colon  U\into X $ be an essentially \'etale morphism. Then the functor
\[
    j^*=j^! \colon \D^+(\QCohC(X,\Coeff)) \to \D^+(\QCohC(U,\Coeff))
\]
of \autoref{jUpperShriekIsjUpperStar} is naturally left adjoint to the functor
\[
    Rj_* \colon \D^+(\QCohC(U,\Coeff)) \to \D^+(\QCohC(X,\Coeff))
\]
 of \autoref{DF-RegIm}.
\end{prop}
\begin{proof}
There is a well-known adjunction between $Rj_*$ and $j^*$ on the derived category of quasi-coherent sheaves. The adjunction is given by the natural transformations $\ad\colon j^*Rj_*\to\id$ and $\cad\colon\id\to Rj_*j^*$ satisfying the usual compatibilities, see~\cite[Thm.~VI.1.2]{MacLane.Cat}. We need to show that both of these transformations are compatible with the $\kappa$-structure. As we shall see, this will come down to the verification of natural isomorphism $\ad\circ\Fr_*\cong\Fr_*\circ\ad$ and $\cad\circ\Fr_*\cong\Fr_*\circ\cad$. We now give details. To simplify notation, we omit $\times \id$ from it until the end of the proof.

For the natural transformation $\ad$ we need to verify that the trapezoid in the following diagram commutes:
\[\xymatrix{ \Fr_* \ar[d]_\kappa & \Fr_*j^* Rj_*  \ar[drr]_{\kappa_{j^*Rj_*}}\ar[l]_-{\Fr_*\ad} & \ar[l]_-{\mathrm{b.c.}}^{\sim}  j^*\Fr_*Rj_* &  \ar[l]^{~\sim}_{\can}  j^* Rj_*\Fr_* \ar[d]_{j^*Rj_*\kappa} \ar[r]^{\ad\Fr_*}& \Fr_* \ar[d]^\kappa \\
\id &&&\ar[lll]^{\ad}j^*Rj_* \ar[r]^\ad &\id \rlap{.}}\]
Here by $\mathrm{b.c.}\colon j^*\Fr_*\to \Fr_* j^*$ we denote the base change isomorphism from \autoref{CSh-FrobEtale} and by $\can\colon \Fr_* Rj_* \to[\sim] Rj_*\Fr_*$ the canonical isomorphism. By functoriality of the adjunction, the square on the right is commutative. The triangle in the middle is the definition of~$\kappa_{j^*Rj_*}$ as $(j^*Rj_*\kappa)\circ \can^{-1}\circ(\mathrm{b.c.})^{-1}$.

We claim that the two maps from $j^* Rj_*\Fr_*$ to $\Fr_*$ in the top row are equal; this will complete the proof of commutativity of the above diagram. For the claim, consider
\begin{equation}\label{Eqn:DiagForAd}
\xymatrix@C+2pc@R-.5pc{ & \Fr_*j^* Rj_* \ar[dl]_{\Fr_*\ad} &  & Rj_* \Fr_*j^* Rj_* \ar[dl]_{Rj_*\Fr_*\ad}  \\
\Fr_*&  j^*\Fr_* Rj^* \ar[u]_{\mathrm{b.c.}}    & Rj_* \Fr_*&  \Fr_* Rj^* \ar[u]_{\ad(\mathrm{b.c.})}    \\
&\ar[ul]^{\ad\Fr_*} j^* Rj_* \Fr_*  \ar[u]_{j^*\can}      &  &\ar@{=}[ul]   Rj_* \Fr_*  \ar[u]_{\can}     \rlap{,} \\ }
\end{equation}
where the right diagram is the adjoint of the left one. It suffices to verify the commutativity on the right. By the definition of $\mathrm{b.c}$, its adjoint $\ad(\mathrm{b.c})$ is the composite $$\Fr_*Rj_* \to[\Fr_*\cad Rj_*] \Fr_* Rj_* j^* Rj_* \to[\can] Rj_*\Fr_*  j^* Rj_* .$$ The commutativity of \autoref{Eqn:DiagForAd} now follows from the fact that the composite
$$Rj_* \to[ \cad Rj_*] Rj_* j^* Rj_* \to[ Rj_* \ad] Rj_*$$
is the identity. (Switch all $\Fr_*$ in the right diagram of \autoref{Eqn:DiagForAd} to the left using $\can$.)

We leave it to the reader to draw the diagram that needs to commute for the compatibility of $\kappa$ with $\cad$. Here the proof comes down to the commutativity of
\begin{equation}\label{Eqn:DiagForCad}
\xymatrix@C+2pc@R-.5pc{ & Rj*j^*\Fr_*  \ar[d]^{\mathrm{b.c.}} \\
\Fr_*\ar[ur]^{\cad\Fr_*} \ar@{-->}[r]^\delta \ar[dr]^{\Fr_*\cad}  & Rj_*\Fr_*  j^*\ar[d]^{\can}\\
&j^* Rj_* \Fr_*  \rlap{,}\\}
\end{equation}
where $\delta$ is defined as $\delta= \can^{-1} \circ \Fr_*\cad$. This time the base change map is the left adjoint of $\delta$, hence $\delta$ is the right adjoint of $\mathrm{b.c}$. But the right adjoint of $\mathrm{b.c}$ is precisely the composite around the top triangle, and hence that triangle commutes.
\end{proof}

The adjunction in \autoref{DF-UpperStar-Adj-New} yields some basic results on injectives:
\begin{lem}\label{BasicInj}
Let $j\!:U\into X$ be an open immersion. If $\UCI\in\QCohC(U,\Coeff)$ is injective, then so is $j_*\UCI\in \QCohC(X,\Coeff)$. The same assertion holds for $\IndCohC$ in place of $\QCohC$. Finally if $\UCI$ is any $\kappa$-sheaf in $\QCohC(U,\Coeff)$ whose underlying sheaf is injective in $\QCoh(U\times \CScheme)$, then $Rj_*\UCI=j_*\UCI$ and the sheaf underlying $j_*\UCI$ is injective in $\QCoh(X\times \CScheme)$.
\end{lem}
\begin{proof}
By \cite[2.3.10]{Weib.hom} a functor that is right adjoint to an exact functor, preserves injectives. In the case at hand the functor $j_*$ is right adjoint to the exact functor $j^*$ since $j$ is an open immersion. The first assertion is then immediate; for the second assertion, one also needs to invoke \autoref{DF-DerImOnIndCoh}. The last assertion follows in the same way.
\end{proof}

\begin{prop}\label{EnoughInj}
Suppose that $\{U_i\}_{i\in I}$ is a finite affine cover of a scheme $X$ with  corresponding open immersions $j_i\!:U_i\into X$. Then $\QCohC(X,\Coeff)$ contains enough injectives of the form $\bigoplus_{i\in I}j_{i*}\UCI_i$ where the $\UCI_i\in\QCohC(U_i,\Coeff)$ are injective.

In particular $\QCohC(X,\Coeff)$ contains enough injectives.
\end{prop}

\begin{proof}
By the previous lemma, the object $\bigoplus_{i\in I}j_{i*}\UCI_i$ is injective in $\QCohC(X,\Coeff)$. Next observe that the defining property of sheaves implies that the natural transformation $\id\into \bigoplus_{i\in I}j_{i*}j_i^{*}$ is a monomorphism. Hence it suffices to prove that each category $\QCohC(U_i,\Coeff)$ for $U_i=\Spec R_i$ has enough injectives. This is clear, however, since $\QCohC(U_i,\Coeff)$ is equivalent to the category of right $R_i[\Fr]\otimes \Coeff$-modules, and as a module category, the latter has enough injectives, e.g.~\cite[\S~3D]{Lam.Rings}.
\end{proof}

\begin{lem}\label{KappaInjIsInj}
Let $\UCI$ be an injective object of $\QCohC(X,\Coeff)$. Then its underlying sheaf is injective as a sheaf of $\CO_{X\times\CScheme}$-modules.
\end{lem}
\begin{proof}
Let $\phi\colon\CI\into \CJ$ be an injective homomorphism of $\CO_{X\times\CScheme}$-modules with $\CJ$ injective, and consider the diagram
$$\xymatrix@C+1pc{ \Frid_* \CI \ar[r]^-{\Frid_*\phi} \ar[d]_{\kappa_\CI}& \Frid_*\CJ\ar@{-->}[d]_{\kappa_\CJ}\\
\CI \ar[r]^-{\phi} & \CJ\\
}$$
with the dashed arrow missing at first. Since $\Fr$ is finite and affine, the morphism $\Frid_*\phi$ is injective. The injectivity of $\CJ$ as an $\CO_{X\times\CScheme}$-module now yields the dashed morphism, making $\UCJ:=(\CJ,\kappa_{\CJ})$ into a $\kappa$-sheaf and $\phi\colon \UCI\to\UCJ$  into a monomorphism of $\kappa$-sheaves. Because $\UCI$ is an injective $\kappa$-sheaf, $\phi$ is split and thus $\CI$ is a direct summand of $\CJ$ as a  $\CO_{X\times\CScheme}$-module. As direct summands of injectives are injective, the proof  is now complete.
\end{proof}

\subsubsection*{Cohomological nilpotence of $f^!$}

We begin with a consequence of the above results on injective $\kappa$-sheaves.
\begin{prop}\label{RightDerivedOfFlat}
Suppose $f\colon Y\to X$ is a finite morphism. Then the pullback functor $f^\flat=f^{-1}\CHom_Y(f_*\CO_X,\ublank)$ on $\QCohC(X,\Coeff)$ is left exact.  It induces a right derived functor on $\D^+(\QCohC(X,\Coeff))$ that agrees with $f^!$. In particular there exists a convergent spectral sequence
\[ E_2^{i,j} = R^if^! (H^j\UCV^\bullet) \To R^{i+j} f^! \UCV^\bullet. \]
\end{prop}
\begin{proof}
Clearly $f^\flat$ is left exact. Since by \autoref{EnoughInj} the category $\QCohC(X,\Coeff)$ possesses enough injectives, the functor $f^\flat$ induces a right derived functor $Rf^\flat$ on $\D^+(\QCohC(X,\Coeff))$. By standard arguments, it follows that $Rf^\flat$ has a convergent spectral sequence as claimed for $f^!$. It remains to show that $f^!\cong Rf^\flat$.

The functor $f^!$ is defined as the functor $(f\times \id)^!$ on $\D^+(\QCoh(X\times \CScheme))$ with an induced endomorphism (of complexes), see \autoref{DF-RemShriekDef}. It therefore suffices to show that $f^!$ can be computed by resolutions by injective $\kappa$-sheaves. But this follows from \autoref{KappaInjIsInj} which asserts that the sheaves underlying an injective $\kappa$-sheaf are injective in $\QCoh(X\times\CScheme)$.
\end{proof}

The next two results clarify in particular cohomological nilpotence for~$f^!$.
\begin{lem}\label{CSh-FShriekNilpotenceFiniteEssEtaleFinite}
Let $?$ be in $\{\lnil,\lncoh,\indcoh\}$. If $f\colon Y \to X$ is essentially \'etale or finite, then $f^!\big( \D^+_?(\QCohC(X,\Coeff)) \big)\subset \D^+_?(\QCohC(Y,\Coeff))$.
\end{lem}
\begin{proof}
If $f$ is essentially \'etale it is in particular flat and therefore $f^! \cong f^*$ is exact so that it commutes with taking cohomology. The result now follows from \autoref{CSh-EtalePullbackPreservesNil}.

If $f$ is finite, we apply \autoref{DF-LemmaOnNil} with $\CA=\CA'=?$, and so we need to verify conditions (a) and (b) of that lemma. The required spectral sequence (b) was established in \autoref{RightDerivedOfFlat} above. For (a), we need to show that the functors $R^if^!=R^if^\flat=f^{-1}R^i\CHom_Y(f_*\CO_X,\ublank)$ preserve local nilpotence and local nil-coherence. By functoriality, $R^if^\flat$ preserves nilpotence: if the $n$-th iterate of $\kappa_\CV$ is zero, the same holds for $\kappa_{R^if^\flat\CV}$. The preservation of coherence under the $R^i f^\flat$ is clear, since it holds already in the absence of a $\kappa$-structure. Moreover the functors $R^if^\flat$ commute with direct limits, since they are constructed via injective resolutions. From the three properties just explained, condition (a) of  \autoref{DF-LemmaOnNil}  is immediate.

\end{proof}
\begin{prop}\label{CSh-FShriekNilpotenceForEssFinite}
Let $?$ be in $\{\lnil,\lncoh,\indcoh\}$. If $f\colon Y \to X$ is essentially of finite type, then $f^!\big( \D^+_?(\QCohC(X,\Coeff)) \big)\subset \D^+_?(\QCohC(Y,\Coeff))$.
\end{prop}
\begin{proof}
By the factorization result \autoref{Mor-FactOfEFT} of Nayak, and the previous lemma, it suffices to prove the result when $f$ is a morphism of finite type. Let $\FU$ be an affine cover of $X$ and $\FV$ an affine cover of $Y$ refining the open cover $f^{-1}(\FU)$. Since by \autoref{CSh-PropertiesLocalize} containment in $?$ is a local property, it suffices to prove the proposition for the affine finite type morphisms $V\to U$ for $V\in\FV$ and $U\in\FU$ which arise by restriction from~$f$. Such a morphism factors as $V\into \BA^n_U\to U$ for a closed immersion $ V\to\BA^n_U$ and the canonical projection $\pr\colon \BA^n_U\to U$. By the previous lemma, it remains to prove the assertion for the smooth morphism~$\pr$. By \autoref{ExplicitUpperShriek} we find that $f^!=f^*\otimes \omega_f$ with $\omega_f\cong \CO_{\BA^n_X}$ is an exact functor. It clearly preserves nilpotence and coherence and it commutes with filtered direct limits. As in the previous lemma, these three properties suffice to complete the proof.
\end{proof}

Together with \autoref{DF-CompDCII} and \autoref{DF-ConstrDFonCrys}, the previous proposition~yields:
\begin{theorem}\label{DF-ShriekOnD-Crys}
If $Y\to[f] X$ is essentially of finite type, then the functor $f^!$ from \autoref{DF-ShriekDef} induces an exact $\Coeff$-linear functor
\[
    f^!\!:\D^+_\crys(\QCrysC(X,\Coeff))\longto \D^+_\crys(\QCrysC(Y,\Coeff))\, .
\]
\end{theorem}
In \autoref{CohomDimOfShriek} we shall prove the stronger result that $f^!$ maps $\D^b_\crys(\QCrysC(X,\Coeff))$ into $\D^b_\crys(\QCrysC(Y,\Coeff))$, which means that $f^!$ is a bounded functor on crystals.

\paragraph{Adjunctions for crystals.}

\begin{prop}\label{DF-Trace-Adj}
Let $f\colon  Y\to X$ be a proper morphism. As functors on categories of the form $\D^+_?(\QCohC(\ldots)) $ for $?\in\{\lncoh,\ind,\emptyset\}$, the functor $f^!$ of \autoref{CSh-FShriekNilpotenceForEssFinite}  is naturally right adjoint to the functor $Rf_*$ of \autoref{GenDefDerImQCohC}.
\end{prop}

\begin{proof}
The proof is parallel to that of \autoref{DF-UpperStar-Adj-New}. The adjunction between $Rf_*$ and $f^!$ on the derived category of quasi-coherent sheaves can be formulated in terms of the trace homomorphism $\tr\colon Rf_*f^!\to \id$ and its bi-adjoint $\ctr\colon \id\to f^!Rf_*$; see \cite[p.~383 or p.~417]{HartshorneRD} or \cite[4.1]{LipmanGrothDual}.\footnote{Note that for proper $f$ neither Deligne's appendix to \cite{HartshorneRD} or \cite{LipmanGrothDual} require complexes with coherent cohomology. Note also that since our schemes are separated, we may stay within the derived categories of quasi-coherent sheaves.} To extend this adjunction to functors on $\D^+_\coh(\QCohC(\ldots)) $, one has to show that $\tr$ and $\ctr$ are compatible with the structural map $\kappa$. This can be expressed as the commutativity of diagrams like \autoref{Eqn:DiagForAd} and \autoref{Eqn:DiagForCad} with $\ad$ and $\cad$ replaced by $\tr$ and $\ctr$. As with \autoref{DF-UpperStar-Adj-New}, the argument only uses formal properties of adjunction. Details are left to the reader.
\end{proof}

\begin{cor}\label{DF-Crys-Adj}
Let $f \colon Y \to X$ be proper and $j \colon U \to X$ be essentially \'etale. Then the adjoint pairs $(Rf_*,f^!)$ from \autoref{DF-Trace-Adj} and $(j^*,Rj_*)$ from \autoref{DF-UpperStar-Adj-New} induce adjuctions on the corresponding categories $\D^+:=\D^+_\crys(\QCrysC(\ldots))$\footnote{The adjunction $(j^*,Rj_*)$ also holds in the unbounded derived category by the same argument.} of $\kappa$-crystals. Concretely, there are natural functorial isomorphisms
\[
    \Hom_{\D^+}(Rf_*\ublank,\ublank) \cong \Hom_{\D^+}(\ublank,f^!\ublank)
\]
and
\[
    \Hom_{\D^+}(j^*\ublank,\ublank)\cong \Hom_{\D^+}(\ublank,Rj_*\ublank)
\]
\end{cor}

\begin{proof}
Let $(F,G)$ be either $(j^*,Rj_*)$ for an essentially \'etale morphism $j\colon U\into X$, $(Rf_*,f^!)$ for a proper morphism $f\colon Y\to X$.
By \autoref{DF-UpperStar-Adj-New} and \autoref{DF-Trace-Adj}, we have natural transformations $a\colon FG\to \id$ and $c\colon\id\to GF$ on $\D^+_\lncoh(\QCohC(\ldots))$ such that $F\to[Fc] FGF \to[aF] F$ and $G\to[cG] GFG\to[Ga] G$ are the identity. Under localization at nil-isomorphisms we obtain induced natural transformations on categories $\D^+_\crys(\QCrysC(\ldots))$ and the formulae yielding $\id_F$ and $\id_G$, respectively, are preserved; see~\cite[Prop.~2.2.5]{BoPi.CohomCrys}.
\end{proof}

\section{Support, Kashiwara equivalence, and applications}

In this section we shall establish a basic exact triangle which allows us to decompose complexes of Cartier crystals on a scheme $X$ into their restriction to an open subscheme $U$ and a closed complement $Z$. Applications of this triangle are given in the subsections. First we show in \autoref{CohDimOfShriek} that $f^!$ on Cartier crystals has bounded cohomological dimension. In \autoref{CrysSupp} we obtain some basic results on the crystalline support of a complex of Cartier crystals.

\subsection{An exact triangle and Kashiwara equivalence}
\label{ExactTriangleSubsec}
Throughout this subsection, we fix an open immersion $j\colon U\into X$ and a complement $i\colon Z\into X$, i.e. a closed immersion $i\colon Z\into X$ such that $U=X\setminus i(Z)$ on points. Due to \autoref{CorOnReducedSubscheme} there is no harm in assuming that $Z$ carries its reduced scheme structure.

The following is the first main result of this subsection, see also \cite[Section 3.3]{BliBoe.CartierFiniteness}:
\begin{theorem}\label{ExactTriangleThm}
In $\D^+_?(\QCrysC(X,\Coeff))$, for $?\in\{\emptyset,\crys\}$, there is a natural exact triangle
\[
    i_*i^!\to \id \to Rj_*j^* \to i_*i^![1]
\]
\end{theorem}
\begin{proof}
Let $\Fa\subset\CO_X$ denote the ideal sheaf corresponding to $i\colon Z\into X$. Since $i$ is finite, the functor $i^!$ is the right derived functor $i^{-1}R\CHom_{\CO_{X\times C}}(\CO_{Z\times C},\ublank)$. Moreover the homomorphism $ i_*i^!\to \id$ is the evaluation at $1$ map
\[ R\CHom_{\CO_{X\times C}}(\CO_{Z\times C},\ublank) \longto \ublank.\]
Explicitly, if one is given a complex $\UCI^\bullet$ of injective $\kappa$-sheaves, this is the map of complexes which in each degree $i$ is the evaluation at the $1$ section
\[ \CHom_{\CO_{X\times C}}(\CO_{Z\times C},\UCI^i) \longto \UCI^i.\]
Next note that if $\UCI^i$ is an injective $\kappa$-sheaf, then by \autoref{KappaInjIsInj} its underlying sheaf $\CI^i$ is injective as an $\CO_{X\times C}$-module. Hence by \cite[Exer.~III.3.6]{Hartshorne} the sheaf $\CI^i$ is flasque and hence so is $j^*\CI^i$. Thus $Rj_* j^* \UCI^\bullet$ is given by $j_* j^* \UCI^\bullet$ since $Rj_*$ can be computed via flasque resolutions. Since we may compute all terms in the asserted triangle of the theorem via injective resolutions of $\kappa$-sheaves,  to prove the theorem it suffices to show that the natural short sequence of complexes
\[ \CHom_{\CO_{X\times C}}(\CO_{Z\times C},\UCI^\bullet) \longto \UCI^\bullet \to  j_*j^* \UCI^\bullet \]
is degree wise exact as a complex of $\kappa$ quasi-crystals. To verify this we may assume that $X=\Spec R$ is affine and hence $\Fa$ is an ideal of $R$. We shall simply write $I$ for the injective $\kappa$-module on $R\times\Coeff$ and $\wt I$ for the associated sheaf on $\Spec R$. Denoting by $I[\Fa] = \Hom (R/\Fa,I)$ the $\Fa$-torsion submodule of $I$, we need to show that
\begin{equation} \label{SESforFundTriangle}
0 \to I[\Fa] \to I \to \Gamma(U,j_*j^*\wt I) \to 0
\end{equation}
is exact in the category of $\kappa$ quasi-crystals. By \cite[Exer.~III.3.7]{Hartshorne} the right hand term is given by a formula of Deligne as $ \Gamma(U,j_*j^*\wt I) = \dirlim_n\Hom_{R\times\Coeff}(\Fa^n\times\Coeff,I)$. The injectivity of $I$ and the exactness of $0\to \Fa^n\to R\to R/\Fa^n\to 0$ imply that for all $n$
\[ 0 \to I[\Fa^n] \to I \to \Hom_{R\times\Coeff}(\Fa^n\times\Coeff,I) \to 0\]
is a short exact sequence, and hence so is the direct limit over this. The key observation now is that for all $n$ the inclusion
\[ I[\Fa]\into I[\Fa^n]\]
is a local nil-isomorphism: this is so, because $\kappa^e(I[\Fa^{[p^e]}])\subset I[\Fa]$ for all $e\in\BN$ (\cf the proof of \autoref{CSh-LemmaNilStalksLocal}) and because the powers $\Fa^n$ and $\Fa^{[p^e]}$ generate the same topology on~$R$, since in the ordinary- and bracket-powers are co-final within one another. It follows that \autoref{SESforFundTriangle} is exact in $D^+(\QCrysC(X,\Coeff)$ and hence the proof of the theorem is complete for $?=\emptyset$. For $?=\crys$ we only need to observe that the triangle preserves the subcategory $D^+_\crys(\QCrysC(X,\Coeff)$.
\end{proof}

As an immediate consequence we record here an analog of the Kashiwara equivalence. As above let $i \colon Y \into X$ be a closed subset with open complement
\[
    Y \into[i] X \xhookleftarrow{\ j\ } U\, .
\]
We define, that a $\kappa$-crystal $\UCM$ on $X$ \emph{is supported in $Y$} if and only if its restriction to $U$ is zero (as a crystal), \ie $j^!\UCM=j^*\UCM=0$. Likewise, a complex $\UCM^\bullet \in \D^+(\QCrysC(X,\Coeff))$ \emph{is supported in $Y$} if $j^*\UCM^\bullet$ is a acyclic as a complex of $\kappa$-crystals.\footnote{Below in \autoref{CrysSupp} we will introduce stalks for $\kappa$-crystal and a notion of \emph{crystalline support} of a $\kappa$-sheaf. The condition that $\UCM^\bullet$ is supported in $Y$ means then precisely that its crystalline support is contained in $Y$, justifying the \emph{ad hoc} definition above.} Since $j^*$ is exact for open immersions, this is equivalent to the cohomology $H^i(\UCM^\bullet)$ of $\UCM^\bullet$ being supported in $Y$ for all $i$. We obtain the following equivalence of derived categories of crystals on $Y$ with crystals on $X$ which are supported on $Y$.

\begin{theorem}[Kashiwara equivalence]\label{thm.DerivedKashiwara}
Let $i \colon Y \into X$ be a closed immersion. For $? \in \{\emptyset,\crys\}$ and $*\in\{b,+\}$, the essential image of the functor
\[
    i_* \colon \D^*_?(\QCrysC(Y,\Coeff))   \to \D^*_?(\QCrysC(X,\Coeff))
\]
is the full triangulated subcategory $\D^*_{?,Y}(\QCrysC(X,\Coeff))$ of $\D^*_?(\QCrysC(X,\Coeff)$ consisting of complexes which are \emph{supported in $Y$}. The inverse to $i_*$ is given by $i^!$ which hence yield a pair of inverse equivalences of categories
\[
\xymatrix{
      \D^*_?(\QCrysC(Y,\Coeff)) \ar@<.5ex>[r]^-{i_*} & \D^*_{?,Y}(\QCrysC(X,\Coeff)) \, . \ar@<.5ex>[l]^-{i^!}
}
\]
\end{theorem}
\begin{proof}
Let $j \colon U \into X$ be the open complement of $Y$ in $X$. For $\UCN^\bullet \in \D^+_?(\QCrysC(Y,\Coeff))$ clearly $j^*i_*(\UCN^\bullet)=0$ since this is already the case on underlying quasi-coherent sheaves. Hence $i_*\UCN^\bullet$ is supported in $Y$. Likewise, the natural transformation $\id \to i^!i_*$ is an isomorphism since this is already the case on the underlying quasi-coherent sheaves.\footnote{We have
\[
i^!i_*(\usc)=i^{-1}\RCHom(i_*\CO_Y,i_*\usc) \cong \RCHom(Li^*i_*\CO_Y,\usc) \ot[\cong] \RCHom(\CO_Y,\usc)=\id(\usc)
\]
where the last isomorphism is induced from $Li^*i_*\CO_Y \to \CO_Y$ which is an isomorphism for any closed immersion.}
Now, for $\UCM \in \D^+_?(\QCrysC(X,\Coeff))$ consider the triangle from \autoref{ExactTriangleThm}
\[
      i_*i^!\UCM^\bullet \to \UCM^\bullet \to Rj_*j^* \UCM^\bullet \to i_*i^![1]
\]
and note that by definition, if $\UCM^\bullet$ is supported in $Y$, then the term $Rj_*j^* \UCM^\bullet$ is zero, and hence we have a quasi-isomorphism $i_*i^!\UCM^\bullet \to \UCM^\bullet$, showing essential surjectivity of $i_*$. This shows the claimed equivalence of categories for $*=+$.

The case $*=b$ follows from the case $*=+$ since $i_*$ and $i^!$ both have bounded cohomological dimension on crystals by \autoref{GenDefDerIm} and \autoref{CohomDimOfShriek}.\footnote{The proof of \autoref{CohomDimOfShriek} below only uses \autoref{ExactTriangleThm} and not \autoref{thm.DerivedKashiwara}, so there is no circular reasoning.}
\end{proof}
As an immediate corollary we obtain:
\begin{cor}\label{CorOnReducedSubscheme}
Let $i \colon X_\red\to X$ be the closed immersion of the reduced subscheme of a scheme. Let  $?$ be in $\{\emptyset,\crys\}$ and $*\in\{b,+\}$. Then $i^!$ and $i_*$ define natural isomorphisms
\[\xymatrix{  \D^*_?(\QCrysC(X_\red,\Coeff))   \ar@<.5ex>[r]^-{i_*} &   \D^*_?(\QCrysC(X,\Coeff)) \ar@<.5ex>[l]^-{i^!} } \]
\end{cor}
\begin{proof}
We apply the above with $Y=X_\red$ and hence $U = \emptyset$.
\end{proof}
From this equivalence on the level of the derived category we immediately obtain an equivalence on the underlying abelian category which was already shown in the case of $\CrysC$ in \cite[Prop. 3.19]{BliBoe.CartierFiniteness}
\begin{cor}
Let $i \colon Y \to X$ be a closed immersion. Denote by ${\QCrysC}_{,Y}(X,\Coeff)$ the category of Cartier crystals supported on $Y$. Then $i^!$ and $i_*$ define natural isomorphisms
\[\xymatrix{  \QCrysC(Y,\Coeff)  \ar@<.5ex>[r]^-{i_*} &   {\QCrysC}_{,Y}(X,\Coeff) \ar@<.5ex>[l]^-{i^!} } \]
and in particular $R^ai^!$ is zero on ${\QCrysC}_{,Y}(X,\Coeff)$ for $a \neq 0$. The same statement hold for $\CrysC$.
\end{cor}
\begin{proof}
For a single (coherent) $\kappa$-crystal $\UCM$ \autoref{thm.DerivedKashiwara} shows that $i^!i_*\UCM \cong \UCM$. In particular, the higher cohomology $R^ai^!$ vanishes for $a \neq 0$. Now all claims follow from \autoref{thm.DerivedKashiwara}.
\end{proof}

\subsection{Finite cohomological dimension of the twisted inverse image}\label{CohDimOfShriek}

As noted before in \autoref{FirstRemOnfShriekCohomBdd}, given a morphism $f\colon Y\to X$ of essentially finite type, the functor $f^!$ on quasi-coherent sheaves may have infinite cohomological dimension. However, in \cite{BliBoe.CartierDuality} we show that Grothendieck-Serre duality induces a duality between $\D^b_{\crys}(\QCrysC(X,\Coeff))$ and the corresponding category of $\tau$-crystals of \cite{BoPi.CohomCrys}; the functor $f^!$ then corresponds to $f^*$ which is shown to be exact for $\kappa$-crystals in op. cit. Hence $f^!$ has finite cohomological dimension of $\kappa$-crystals. The bounds on the cohomological dimension of $f^!$ one obtains via this reasoning are however not optimal.

In fact, we show in this section directly that for $\kappa$-crystals the cohomlogical width of $f^!$ on $\D^+_\crys(\QCrysC((\ldots))$ satisfies the same bounds as one would expect if $X$ and $Y$ were both smooth. In the smooth case one has -- already on the underlying sheaves -- that $f^!(\CM)$ for $\CM$ a sheaf has cohomology only in the range $[-\dim Y, \dim X]$; this is well known and explained below in some detail.
\begin{theorem}\label{CohomDimOfShriek}
Let $f\colon Y\to X$ be a morphism essentially of finite type and suppose that $X$ has finite Krull dimension $\dim X$. Then for any $\UCV$ in $\CrysC(X,\Coeff)$, the complex $ f^!\UCV$ has cohomology supported in degrees $[-\dim Y,\dim X]$.
\end{theorem}
The proof is a typical application of our triangle \autoref{ExactTriangleThm}: One chooses $U$ to be a dense open subset with desirable properties (in our case a dense affine regular subset). On $U$ the statement is known. Then the triangle \autoref{ExactTriangleThm} allows to reduce the general case to the corresponding statement on the complement of $U$, which is dealt with by induction on dimension. In the present situation this type of argument is applied twice, on the base $X$ and then on $Y$. But before we carry out the proof in detail, let us recall the situation when $X$ and $Y$ are regular:
\begin{lem} \label{LemOnFTMorOfRegSchemes}
Let $f\colon Y \to X$ be a finite type morphism of equidimensional regular schemes with $X$ (and hence $Y$) of finite Krull dimension. Then $f^!\CO_X$ is quasi-isomorphic to a complex $\omega_f[\dim Y-\dim X]$ with $\omega_f$ locally free of rank $1$. Furthermore, one has a functorial isomorphism
\[  f^!(\usc) = f^!\CO_X \Ltensor Lf^*(\usc), \]
so in particular $f$ is perfect.
\end{lem}
\begin{proof}
To prove the assertion on the shape of $f^!\CO_X$ it suffices to show that its representing complex is acyclic outside degree $\dim X-\dim Y$ and has a locally free sheaf of rank $1$ as the cohomology in that degree. Since this is a local assertion, we may factor $f$ (locally but without changing the notation) as
\[ Y \to[i] \BA^n\times X \to[\pr_2] X\]
with $i$ a closed immersion. Because $Y$ and $ \BA^n\times X$ are regular, $i$ is a locally complete intersection morphism. Via the Koszul complex and the explicit formula for $i^!$ for the finite morphism $i$, the latter implies that $i^!\CO_{\BA^n\times X}$ is represented by a complex of the form $\omega_i[\dim Y-\dim X - n]$ as stated. The assertion on $f^!\CO_X$ follows since $\pr_2^!\CO_X=\CO_{\BA^n\times X}[n]$. The remaining assertions follow from \cite{Na.CompEFT}, see also \autoref{ExplicitUpperShriek} on page \pageref{ExplicitUpperShriek}.
\end{proof}
This description of $f^!(\usc)$ in the case of a finite type morphism between regular schemes immediately implies the following result which shows \autoref{CohomDimOfShriek} in this case.
\begin{cor}\label{ShriekForFTMorBetweenReg}
Let $f\colon Y \to X$ be a finite type morphism of equi-dimensional regular schemes with $X$ (and hence $Y$) of finite Krull dimension. Then for any quasi-coherent sheaf $\CV$ on $X \times \CScheme$, the complex $(f \times \id)^!\CV$ has cohomology concentrated in degrees $[-\dim Y,\dim X-\dim Y]$.
\end{cor}
\begin{proof}
Observe that the argument in the proof of the previous lemma applies to $f \times \id_{\CScheme}$ where $C$ is any essentially finite type $\BF_q$-scheme $\CScheme$. As a finite type morphism between regular schemes, $f$ and hence $f \times \id_{\CScheme}$ are essentially perfect, due to the fact that essentially perfect is preserved by base change. It yields that $(f \times \id)^!\CO_{X\times\CScheme}$ is quasi-isomorphic to $\omega_{f \times \id}[\dim Y-\dim X]$ for some locally free sheaf $\omega_{f\times \id}$ of rank one on $Y\times \CScheme$ and that
\[ (f \times\id)^!(\usc) = \omega_{f\times \id} \otimes_{\CO_{X\times\CScheme}}  L(f\times\id)^*(\usc)[\dim Y-\dim X] .\]
Thus it remains to prove that $L(f\times\id)^*$ has cohomology concentrated in degrees $[-\dim X, 0]$ which is a local assertion. So suppose $X=\Spec R$, $Y=\Spec S$ for regular rings $R,S$, $f$ is a ring homomorphism $f\colon R\to S$ and the module underlying $\CV$ is $M$ on $R\otimes \Coeff$. Then $L^{i}(f\times\id)^*\CV$ is the sheaf associated to the module
\[\Tor_{-i}^{R\otimes\Coeff}(S\otimes\Coeff,M) \cong \Tor_{-i}^R(S,M). \]
Using that $R$ is regular of dimension $\dim X$, the right hand side shows that $L^{i}$ vanishes if $i$ is not in $[-\dim X,0]$.
\end{proof}
If one drops the assumption of equi-dimensionality in the preceding Corollary we can only conclude that the cohomology of $(f \times \id_{\CScheme})^!(\CV)$ is concentrated in degrees $[-\dim Y,\dim X]$. The following lemma prepares the induction step of our main theorem:
\begin{lem}\label{IndStepCohomDimShriek} Let $X$ be a scheme of finite Krull dimension. Let $j\colon U\into X$ be an \emph{affine} open immersion with $i\colon Z\to X$ the closed complement such that $\dim Z < \dim X$.
\begin{enumerate}
\item The right derived functor of
\[R^0i^!\colon \CrysC(X,\Coeff) \to \CrysC(Z,\Coeff) \]
has cohomological dimension at most one.
\item For a morphism $f\colon Y\to X$ essentially of finite type consider the fiber squares
\[\xymatrix{
U\times_XY \ar[r]^{\wt j} \ar[d]_{f_U}& Y \ar[d]_f&\ar[l]_{\wt i} \ar[d]^{f_Z} Y\times_XZ\\
U\ar[r]^j & X & \ar[l]_i Z\rlap{.}
}\]
If $f^!_U$ and $f_Z^!$ take Cartier-crystals to complexes whose cohomology is supported in degrees $[a,b]$ and $[a,b-1]$, respectively, then $f^!$ takes Cartier-crystals to complexes whose cohomology is supported in degrees $[a,b]$.
\end{enumerate}
\end{lem}
\begin{proof}
Since the map $j$ is affine, the functor $j_*$ and hence $j_*j^*$ is exact, as a composite of exact functors on  $\CrysC(X,\Coeff)$. In the exact triangle $i_*i^! \to \id \to j_*j^*$ of functors on the category $\D^+_\crys(\QCrysC(X,\Coeff))$, and for $\UCV$ in $\CrysC(X,\Coeff)$ it follows that both, $\id\UCV=\UCV$ and $j_*j^*\UCV$ are concentrated in degree zero. Hence $i_*i^!\UCV$ is concentrated in degrees $0$ and  $1$ only.  Since $i_*$ is exact and faithful part (a) follows.

For part (b) we use the same triangle applied to $f^!(\UCV)$
\[  (\wt i_*\wt i^!f^!(\UCV) \to f^!(\UCV) \to \wt j_*\wt j^*f^!(\UCV)) \cong (\wt i_*f_Z^!i^!(\UCV) \to f^!(\UCV) \to \wt j_* f_U^!j^*(\UCV) ) .\]
By the definition of fiber squares, for any affine $V$ in $Y$ the subscheme $U\times_X V\subset V$ is affine and hence $\wt j_*$ is affine. Thus $\wt j_*$ is exact. Hence by our assumption on $f_U^!$ the last entry of the the triangle has cohomology only in degrees $[a,b]$. By part (a) and the assumption on $f_Z^!$ the first entry of the triangle has cohomology only in degrees $[a, b-1+1]$. The long exact cohomology sequence shows now that $f^!(\UCV)$ has cohomology only in degrees $[a,b]$.
\end{proof}

\begin{proof}[Proof of \autoref{CohomDimOfShriek}]
Since we may factor $f$ as the composite of a finite type morphism and a localizing immersion and since for a localizing immersion the functor extra-ordinary inverse image is exact, we may assume that $f$ is of finite type. Note also that by \autoref{CorOnReducedSubscheme} we may assume that $X$, and if necessary $Y$, is reduced. We induct on the dimension of $X$ and for a fixed such dimension, on the dimension of~$Y$:

So let $f\colon Y\to X$ be a morphism of finite type with $X$ and $Y$ reduced. Let $j\colon U\into X$ be the embedding of an affine open regular subscheme $U$ of $X$ such that $U$ is equi-dimensional and dense in all components of $X$ of maximal dimension---the existence of $U$ is provided by generic regularity, cf.~\cite[Thm.~24.4]{Matsumura}. Let $i\colon Z\to X$ be the closed reduced complement of $U$. By \autoref{IndStepCohomDimShriek} and the induction hypothesis applied to $Z$, it suffices to prove the result if $X=U$ is regular and equi-dimensional. This is treated by induction on $\dim Y$:

We choose an \emph{affine} open embedding $\wt j\colon V\into Y$ such that $V$ is regular, equi-dimensional and dense in all components of $Y$ of maximal dimension and we denote by $\wt i\colon W\to Y$ the inclusion of the reduced closed complement. As above, we consider the triangle
\[ \wt i_*(f\wt i)^! \to f^! \to \wt j_* (f\wt j)^!  \]
obtained from \autoref{ExactTriangleThm}. Since the underlying morphisms are affine $\wt i_*$ and $\wt j_*$ are exact. Moreover the cohomological dimension of $f\wt i$ is controlled by induction hypothesis. Thus it remains to show that $(f\wt j)^!$ maps Cartier-crystals to complexes with cohomology concentrated in degrees $[-\dim Y,\dim X]$. But with notation as in the diagram in \autoref{IndStepCohomDimShriek} we have that $(f\wt j)^!=f_U^! \circ \wt{j}^*$, and since $f_U$ is a finite type morphism between equi-dimensional regular schemes, this follows from \autoref{ShriekForFTMorBetweenReg}.
\end{proof}

\begin{rem}
By considering the projection $\pi \colon \mathbb{A}^n_k \to \Spec k$ to the prefect ground field, we see that $\omega^\bullet_{\mathbb{A}^n_k} \colonequals \pi^!k$ is a complex concentrated in degree $-\dim \mathbb{A}^n_k = - n$. This shows that the left bound in our theorem is sharp. Similarly, considering the inclusion $i \colon \Spec k \into \mathbb{A}_k^n$ and the Cartier module $\omega_{\mathbb{A}^n_k} \colonequals \pi^!k[-n]$, we get that $i^!\omega_{\mathbb{A}_k^n}= k[-n]$, showing that the right bound is also sharp.
These two cases can be pieced together to get an example where there is cohomology precisely in the range allowed by \autoref{CohomDimOfShriek}.
\end{rem}

\subsection{Stalks of Cartier crystals and the crystalline support}
\label{CrysSupp}\label{sec.Stalks}

In this section, we discuss a notion of stalks for $\kappa$-crystals. Their main use will be a criterion for nilpotence of $\kappa$-sheaves. Related to this, we also define a notion of crystalline support. In parts we are able to clarify the relation between the crystalline support and the non-vanishing of stalks.

For a (not necessarily closed) point $x \in X$ denote by $i_x \colon \Spec k_x = x \into X$ the natural inclusion. For any $\kappa$-crystal $\UCV$ on $X$ over $\Coeff$, we define $i_x^!\UCV$ as the stalk of $\UCV$ at $x$. In general this is a complex not necessarily concentrated in one degree. The reason for this name is that we shall show that a $\kappa$-sheaf $\UCV$ is locally nilpotent if and only if for all $x\in X$ the complex $i_x^!\UCV$ has locally nilpotent cohomology.

For $x \in X$ we observe that the inclusion $i_x \colon \Spec k_x = x \into X$ may be factored as
\[
x \into[i'_{x}] \Spec \CO_{X,x} \to[j_x] X
\]
where $i'_x$ is a closed immersion (and hence finite) and $j_x$ is a localization (hence in particular essentially \'etale). In \autoref{CSh-FShriekNilpotenceFiniteEssEtaleFinite} we observed that if $\UCV$ is (locally) nilpotent, then so is $i_x^!\UCV$ due to its factorization into an essentially \'etale map followed by a finite map. In \autoref{Csh-NilpoenceLocal} we show that $\UCV$ is locally nilpotent if and only if this is the case for all $j_x^!\UCV = j_x^*\UCV= \UCV_x$. To handle the case of the closed immersion $i'_x \colon x \into \Spec \CO_{X,x}$ we first show the following simple, but crucial, lemma.
\begin{lem}\label{CSh-LemmaNilStalksLocal}
Let $(R,\frm)$ be a local Noetherian ring, and denote by $i_x \colon \{\frm\} = x \to X = \Spec R$ the inclusion of the closed point. Assume that $\UCV=(\CV,\kappa)$ is a $\kappa$-sheaf that is supported at $x$.
\begin{enumerate}
\item  If $\UCV$ is coherent then it is nilpotent if and only if $R^0i_x^!\UCV$ is nilpotent.
\item  If $\UCV$ is arbitrary and if $\Fr_R$ is finite, then $\UCV$ is locally nilpotent if and only if $R^0i_x^!\UCV$ is locally nilpotent.
\end{enumerate}
\end{lem}
\begin{proof}
If $\UCV$ is (locally) nilpotent, then, by functoriality, and since $i_x^!$ commutes with filtered direct limits, so is $R^0i_x^!\UCV$. This shows the direction from left to right for both claims.

Assume now that $\UCV$ is finitely generated. Since $x$ is the closed point and $i_x$ is a finite map, then $R^0i_x^!(\UCV)=\CV[\frm] \colonequals  \Hom_{R \tensor \Coeff} (R/\frm \tensor \Coeff,\CV)$ is just the $\frm$-torsion part of $\CV$ viewed as an $R/\frm \tensor \Coeff$-submodule of $\CV$. Note that the Cartier structure on  $\UCV[\frm]$ is just the one given on  $\UCV=(\CV,\kappa)$ restricted to this submodule. Since $\UCV$ is finitely generated and supported in $x$ we may find $e$ sufficiently big such that $\frm^{[p^e]}$ annihilates $\UCV$. Then we have $\kappa^e(\CV)=\kappa^e(\CV[\frm^{[p^e]}]) \subseteq \CV[\frm]$ since $0=\kappa^e(r^{p^e}v)=r\kappa^e(v)$ for all $r \in \frm \tensor \Coeff$ and $v \in \CV$. If now $R^0i_x^!\UCV$ is nilpotent, then $\kappa^s(\CV[\frm])=0$ for some $s \geq 0$. But this implies that $\kappa^{e+s}(\CV)\subseteq \kappa^s(\CV[\frm]) = 0$, and hence $\UCV$ itself is nilpotent. This shows part (a).

For part (b) we use the local nilpotence criterion of \autoref{CSh-CharNilLocNil}. Let $v \in \UCV$ be arbitrary and set $M= (R \tensor \Coeff)v \subseteq \UCV$. Since $M$ is supported at $x$ and finitely generated we argue exactly as above that we may find $e \geq 0$ such that $\kappa^e(\Frid^e_*M) \subseteq \UCV[\frm]=i_x^!\UCV$. The assumption that $\Fr$ is finite on $R$ implies that $\Frid^e_*M$ is finitely generated such that this is also true for $\kappa^e(\Frid^e_*M)$. Assuming now that $R^0i_x^!\UCV$ is locally nilpotent, it is in particular ind-coherent by \autoref{CSh-LNilIsSerre}. Hence by \autoref{Csh-IndCohCharacterization} $\kappa^e(\Frid^e_*M)$ is contained in a coherent, and hence nilpotent, $\kappa$-subsheaf of $R^0i_x^!\UCV$. As above we conclude that $\kappa^{e+e'}\Frid^{e+e'}_*M=0$ for some $e'$. By the local nilpotence criterion of \autoref{CSh-CharNilLocNil} this shows that $\CV$ is locally nilpotent.
\end{proof}
An alternative approach to prove the above lemma would be to use the Kashiwara equivalence of the preceding subsection. However, since one point of our notion of stalks is a point-wise nilpotence criterion, which will lead to a new proof that $f^!$ preserves local nilpotence, we have given a direct proof above.  We now turn to our point-wise criterion for nilpotence of a $\kappa$-sheaf.
\begin{theorem}
Let $\UCV$ be a $\kappa$-sheaf on $X$.
Assuming that either $\UCV$ is ind-coherent or $\Fr$ is finite on $X$, then $\UCV$ is locally nilpotent if and only if $R^0i_x^!\UCV$ is locally nilpotent for all $x\in X$.
\end{theorem}
\begin{proof}
We show the result under either hypothesis at once. The `only if' direction has been shown before, so we concentrate on the converse and assume that $R^0i_x^!\UCV$ is locally nilpotent for all $x \in X$. Let $M=M_0$ be a coherent $\CO_{X \times C}$-subsheaf of $\CV$ and consider $M_i \colonequals  \kappa^i(\Frid^i_*M) = \kappa(\Frid_*M_{i-1})$ for $i \geq 0$. If $\UCV$ is ind-coherent we may and do assume that $M$ is a coherent $\kappa$-subsheaf. By \autoref{CSh-CharLocNilCoh}, we have to show that $M_i=0$ for $i \gg 0$. We first show that all $M_i$ are coherent $\CO_{X \times C}$-subsheaves of $\CV$: In the ind-coherent case this is clear since $M$ is a coherent $\kappa$-sheaf and $M_i \subseteq M$ for all $i$. If $\Fr$ is finite on $X$ then $M_i$ is coherent because it is the image under $\kappa$ of the coherent subsheaf $\Frid^i_*M\subset \UCV$.

Since, if $r M_i=0$ then $r M_{i+1}=r\kappa(\Frid_*M_i)=\kappa(r\Frid_*M_i)=\kappa(\Frid_*r^{q}M_i)=0$ we have $\Ann M_i \subseteq \Ann M_{i+1}$. This implies that $Z_i \defeq \Supp M_i \supseteq \Supp M_{i+1}$. Thus the $Z_i$ form a descending sequence of closed subsets of $X \times \CScheme$ which stabilizes to $Z$. Replacing $M$ by $M_i$ for $i \gg 0$ we may assume $\Supp M_i=Z$ for all $i$. Replacing $\UCV$ by the $\kappa$-sheaf generated by $M$, \ie by $\sum_i M_i$ we may also assume that $\Supp \CV = Z$ as well. Let $\eta$ be the generic point of a maximal dimensional component of $Z$ and let $x=\pr_1 \eta$ be the projection of $\eta$ to $X$. Localizing at $x$ we are now in the situation of \autoref{CSh-LemmaNilStalksLocal}, so that $\UCV_x$ is locally nilpotent. Using the local nilpotence criterion of \autoref{CSh-CharLocNilCoh} it follows that $(M_i)_x=0$ for $i \gg 0$. But by the choice of $x$ this implies that $M_i$ itself is zero for $i \gg 0$, which finishes our proof.
\end{proof}
The above yields the following criterion for local nilpotence:
\begin{cor}\label{CSh-LocalNilForIndonStalks}
For a $\kappa$-sheaf $\UCV \in \IndCohC(X,\Coeff)$ (or $\UCV \in \QCohC(X,\Coeff)$ if $\Fr$ is finite on $X$) the following conditions are equivalent.
\begin{enumerate}
\item $\UCV$ is locally nilpotent.
\item $R^0i^!_x\UCV$ is locally nilpotent for all $x \in X$.
\item $i^!_x\UCV$ is cohomologically nilpotent for all $x \in X$ (\ie $R^a i_x^!\UCV$ is locally nilpotent for all $a$ and $x\in X$).
\end{enumerate}
\end{cor}
\begin{rem}\label{RemOnNonclosednessOfSupport2}
If $X$ is of finite type over a field, we shall see in \autoref{CrysSupp.NonzeroFibersAreDense} that the above conditions are equivalent to

\begin{enumerate}
\item[\itshape (d)] $i^!_x\UCV$ is cohomologically nilpotent for all \emph{closed} points $x \in X$.
\end{enumerate}
\end{rem}
\begin{ex}\label{Ex:NonclosednessOfSupport}
For arbitrary $X$, condition (d) is not equivalent to the others: Let $\Coeff=\BF_q$ and let $X=\Spec V$ be the spectrum of a discrete valuation ring $V$ with uniformizer $\pi$ and (for simplicity) perfect residue field $k$. Let $i \colon s \to X$ be its closed point and $j \colon \eta \to X$ be its generic point.  Let $\Uomega_X\cong V\mathrm{d}\pi$ be the dualizing sheaf of $X$ with $\kappa$-structure obtained from an isomorphism $\omega_X \to[\cong] \Fr^! \omega_X$. Setting $\UCV = j_*j^*\omega_X$ we get from the triangle
\[
    i_*i^! \omega_X \to \omega_X \to j_*j^*\omega_X = \UCV
\]
by applying $i^!$ (and using that $i^!i_* = \id$) that $i^!M=0$. However, $\UCV$ is cleraly nonzero.
\end{ex}

\begin{cor}\label{CSh-CohomNilonStalks}
For $\UCV^\bullet \in \D^+_{\ind}(\QCohC(X,\Coeff))$ (or $\UCV^\bullet \in \D^+(\QCohC(X,\Coeff))$ if $\Fr$ is finite on $X$) the following conditions are equivalent:
\begin{enumerate}
\item $\UCV^\bullet$ is cohomologically nilpotent.
\item For all $x \in X$ the complexes $i^!_x\UCV^\bullet$ are cohomologically nilpotent.
\end{enumerate}
\end{cor}
\begin{proof}
The implication  (a) $\Rightarrow$ (b) is a particular case of \autoref{CSh-FShriekNilpotenceFiniteEssEtaleFinite}.

The proof  (b) $\Rightarrow$ (a) is a standard inductive devissage argument based on \autoref{CSh-LocalNilForIndonStalks}: Assume that $H^l(\UCV^\bullet)$ is locally nilpotent for all $l < l_0$. Let the truncation $\tau_{\leq l_0}\UCV^\bullet$ be the sub-complex of $\UCV^\bullet$ defined by
\[
    (\tau_{\leq l_0}\UCV^\bullet)^l = \begin{cases} \UCV^l & \text{for } l < l_0 \\ \image(\UCV^{l-1}\to \UCV^{l}) & \text{for } l = l_0, \text{ and} \\ 0 & \text{for } l > l_0. \end{cases}
\]
Note that by our choice of $l_0$ the complex $\tau_{\leq l_0}\UCV^\bullet$ is cohomologically nilpotent. If $\UCH^\bullet$ denotes the degree-wise quotient of $\UCV$ by $\tau_{\leq l_0}\UCV^\bullet$ one has the exact triangle
\[
    \tau_{\leq l_0}\UCV^\bullet \to \UCV^\bullet \to \UCH^\bullet \to[+1]\, .
\]
Applying $i_x^!$ to this triangle and the direction (a) $\Rightarrow$ (b) for $\tau_{\leq l_0}\UCV^\bullet$ we see that $i_x^!\UCH^\bullet$ is cohomologically nilpotent for all $x \in X$. We will be done by induction once we show that $H^{l_0}(\UCH^\bullet) \cong H^{l_0}(\UCV^\bullet)$ is locally nilpotent. By construction,  $\UCH^l=0$ for all $l < l_0$. The lowest term of the spectral sequence $R^ii_*^!H^j(\usc) \Rightarrow R^{i+j}i^!_x(\usc)$  derived from \autoref{RightDerivedOfFlat} applied to $\UCH^\bullet$ asserts that $R^0i_x^!H^{l_0}(\UCH^\bullet) \cong H^{l_0}(i^!_x\UCH^\bullet)$. Since for all $x \in X$ the right term is locally nilpotent by assumption, \autoref{CSh-LocalNilForIndonStalks} implies that $H^{l_0}(\UCH^\bullet)$ is locally nilpotent as desired.
\end{proof}
\autoref{CSh-CohomNilonStalks}, which was shown without using any of the results in \autoref{CSh-sec-InverseImage}, yields a second proof of \autoref{CSh-FShriekNilpotenceForEssFinite} that $f^!$ preserves cohomological nilpotence because it allows one to reduce the statement to the case where $f$ is an essentially finite type map between the spectra of fields. Namely if $f\colon Y \to X$ is essentially of finite type and $\UCV^\bullet \in D^+_\lncoh(\QCohC(X,\Coeff)$ is cohomologically nilpotent, let $y \in Y$ be arbitrary and  $x = f(y)$ such that one has the commutative diagram
\[\xymatrix{
    y \ar[r]^{i_y}\ar[d]_{f_y} & Y \ar[d]^f \\ x \ar[r]^{i_x} &X
    }
\]
By \autoref{CSh-CohomNilonStalks} it suffices to show that $i^!_y f^! \UCV^\bullet = f_y^!( i^!_x \UCV^\bullet)$ is cohomologically nilpotent. Since $i^!_x\UCV^\bullet$ is cohomologically nilpotent we are reduced to the case of a map of a field $K$ into a finitely generated extension field $L$. In this case $f^!$ can be described explicitly, is exact and hence clearly preserves cohomological nilpotence:
\begin{ex}
If $f \colon \Spec L \to \Spec K$ is induced by the inclusion $K \subseteq L$ of finitely generated field extensions, then $f^!$ can be described explicitly as follows:
\begin{enumerate}
\item If $L/K$ is purely transcendental of transcendence degree $n$, then $f^!\underline{\phantom{n}} \!= \!L[n] \tensor_K \underline{\phantom{n}}$.
\item If $L/K$ is finite algebraic, then $f^!\underline{\phantom{n}} = \Hom_K (L,\underline{\phantom{n}})$.
\item If  $L/K$ is finite separable, the trace map $\tr_{L/K} \colon L \to K$ induces an isomorphism $L \tensor_K \underline{\phantom{n}}  \to[l \tensor v \mapsto \tr_{L/K}(l)v] \Hom_K(L,\underline{\phantom{n}})$ which identifies $f^!\usc \cong f^*\underline{\phantom{n}} = L \tensor_K \underline{\phantom{n}}$
\end{enumerate}
In each of these cases $f^!$ is an exact functor and hence preserves nilpotence.
\end{ex}

\begin{defn}\label{CrysSupp.Def1}
Let $\UCV$ be a coherent $\kappa$-sheaf on $X$ over $\Coeff$. The \emph{crystalline support} of $\UCV$ is defined as
\[ \SuppCrys \UCV := \pr_1 \Big(  \bigcap_{e\ge 0} \Supp \kappa^e(\CV) \Big) .\]
\end{defn}

\begin{rem}\label{CrysSupp.BecomesStationary}
For $\UCV\in\CohC(X,\Coeff )$, the sets $ \Supp \kappa^e(\CV) $ form a descending sequence of closed subsets of the noetherian scheme $X\times\CScheme$. Therefore the sequence becomes stationary, and in particular $\SuppCrys \UCV := \pr_1 \big(  \Supp \kappa^e(\CV) \big)$ for all $e\gg0$. 
\end{rem}

\begin{lem}\label{CrysSupp.LemOnSuppF_*}
For an affine morphism $f\colon Y\to X$ which is a bijection on points and a coherent sheaf $\CF$ on $Y$ one has \[\Supp(f_*\CF)=f(\Supp\CF).\]
\end{lem}
\begin{proof}
Since $f$ is affine and the assertion is local, we may assume that $Y=\Spec S$ and $X=\Spec R$ are affine and that $f$ is given by a ring homomorphism $f^\#:R\to S$. Because $f\colon \Spec S\to\Spec R$ is a bijection on points, the kernel of $f^\#$ is nilpotent. Therefore it suffices to consider the case where $S$ is an extension ring of $R$. Again by the bijectivity of $f$, for any prime ideal $\Fp$ of $R$ the ring $S\otimes_RR_\Fp$ is local and thus equal to $S_\FP$ for $\FP=f^{-1}(\Fp)$. It follows that for any finite type $S$-module $M$ we have $M_\Fp=M_\FP$, which proves the lemma.
\end{proof}

\begin{lem}\label{CrysSupp.FridInvariance}
For $\UCV\in\CohC(X,\Coeff)$ and $T:=\Supp(\kappa^e(\CV))$ with $e\gg0$, one has $T=\Frid T$.
\end{lem}
\begin{proof}
Suppose that $\Supp(\kappa^e(\CV))$ becomes stationary for $e\ge e_0$. Then for any $e\ge e_0$ we have
\begin{eqnarray*}
 \Supp\kappa^e\CV&=&\Supp\kappa^{e+1}\CV \ \ = \ \ \Supp\big(\kappa(\Frid_*\kappa^e\CV)\big)\\&\subset &\Supp \Frid_*(\kappa^e\CV)\ \ =\ \ \Frid(\Supp\kappa^e\CV),
\end{eqnarray*}
where in the equality is by \autoref{CrysSupp.LemOnSuppF_*} using that $\Frid$ is a topological isomorphism. This shows that $T\subset \Frid T$. But then $\Frid^{-n}(T)$ is a descending sequence of closed subsets of $X\times\CScheme$, and so it becomes stationary. Using again that $\Frid$ is a topological isomorphism, we must have $\Frid T=T$.
\end{proof}
By \cite[Lem.~4.1.6]{BoPi.CohomCrys}, any closed subset $T$ of $X\times\CScheme$ such that $T=\Frid T$ is a finite union $T=\bigcup_i Y_i\times D_i$ for closed subsets $Y_i\subset X$ and $D_i\subset \CScheme$. We deduce:

\begin{cor}\label{CrysSupp.IsClosed1}
The crystalline support  of any $\UCV\in \CohC(X,\Coeff)$ is closed in~$X$.
\end{cor}

\begin{prop}\label{CrysSupp.IndepOfNilIso}
If $\UCV,\UCW\in \CohC(X,\Coeff)$ are nil-isomorphic, then $\SuppCrys \UCV=\SuppCrys\UCW$.
\end{prop}
\begin{proof}
Let $\phi\colon \UCV\to\UCW$ be a nil-isomorphism. After replacing both $\UCV$ and $\UCW$ by their respective image under $\kappa^e$ for $e\gg0$, see \autoref{CrysSupp.BecomesStationary}, we may assume that $\Supp(\kappa^e\CV)=\Supp\CV$ for all $e\ge0$ and the same for $\UCW$. Now choose $e\ge0$ such that there is a diagram
    \[
        \xymatrix{ \Frid^e_*\CV \ar[r]^-{\kappa^e} \ar[d]_{\Frid^e_*\phi}& \CV \ar[d]^\phi\\ \Frid^e_*\CW \ar[ur]^\alpha \ar[r]^-{\kappa^e} & \CW \rlap{.} \\
        }
    \]
as in \autoref{CSh-NilIso}, where $\alpha$ is a morphism of $\kappa$-sheaves. Note that $\Frid^e$ is regarded as a morphism $X\times\CScheme\to X\times\CScheme$.

Next note that if $\psi\colon\CF\to\CG$ is a homomorphism of sheaves then $\Supp(\image(\psi)) \subset\Supp\CF$. Hence $\kappa_\CW^e=\phi\circ\alpha$ implies that
\[ \Supp \CW =\Supp \kappa_\CW^e((\Fr^e\times\id)_*\CW)  \subset \Supp(\image(\alpha)) 
\subset \Supp\CV.\]
The same argument for $\kappa^e_\CV=\alpha\circ \Frid^e_*\phi$ yields
\begin{eqnarray*}
 \Supp \CV &\!\!\!\!=\!\!\!\!& \Supp \kappa_\CV^e((\Fr^e\times\id)_*\CV)  \subset\Supp(\image(\Frid_*^e\phi))\subset \Supp ((\Fr^e\times\id)_*\CW)\\
 &=& (\Fr^e\times\id)\Supp\CW\stackrel{\ref{CrysSupp.FridInvariance}}= \Supp\CW.
\end{eqnarray*}

Combining both inclusions, the proposition follows.
\end{proof}
The following is an immediate consequence:
\begin{cor}
Let $\UCV$ be in  $\LNilCohC(X,\Coeff)$. If $\phi_i\colon \UCV\to\UCW_i$, $i=1,2$, are nil-isomorphisms with $\UCW_i\in\CohC(X,\Coeff)$ (and hence $\kernel\phi_i\in\LNilC(X,\Coeff)$ for $i=1,2$), then $\SuppCrys \UCW_1=\SuppCrys\UCW_2$.
\end{cor}
This corollary allows us to extend the crystalline support to $\LNilCohC(X,\Coeff)$ and to complexes in the following way.
\begin{defn}\label{CrysSupp.Def2}
Suppose $\UCV\in\LNilCohC(X,\Coeff)$ and $\phi\colon\UCV\to \UCW$ is a nil-iso\-morph\-ism with $\UCW\in\CohC(X,\Coeff)$. Then one defines the \emph{crystalline support} of $\UCV$ as
\[ \SuppCrys \UCV :=  \SuppCrys \UCW .\]
For $\UCV^\bullet$ in $\D^b_\lnilcoh(\QCohC(X,\Coeff))$ we define the \emph{crystalline support} as
\[ \SuppCrys \UCV^\bullet := \bigcup_{i\in\BZ} \SuppCrys(H^i(\UCV^\bullet)) .\]
\end{defn}

The following result is now obvious:
\begin{cor}\label{CrysSupp.DependsOnCrys1}
\begin{enumerate}
\item The crystalline support of $\UCV\in\LNilCohC(X,\Coeff)$ only depends on the crystal in $\LNilCrysC(X,\Coeff)$ that it represents.
\item The crystalline support of $\UCV^\bullet\in\D^b_\lnilcoh(\QCohC(X,\Coeff))$ only depends on the complex in $\D^b_\crys(\QCrysC(X,\Coeff))$ that it represents.
\end{enumerate}

In both cases, the crystalline support is a closed subset of~$X$.
\end{cor}

We summarize some properties of the crystalline supports with regards to its behavior under open and closed immersions.
\begin{lem}\label{CrysSupp.SuppAndFunctors}
Denote by $j\colon U\to X$ an open immersion with $i\colon Z\to X$ a closed complement. For $W\in \{X,U,Z\}$, let $\UCV_W^\bullet$ be in $\D^b_\crys(\QCrysC(W,\Coeff))$. Then the following hold:
\begin{enumerate}
\item $\SuppCrys j^*\UCV_X^\bullet=U\cap\SuppCrys\UCV_X^\bullet$.
\item $\SuppCrys j_*\UCV_U^\bullet\subset \overline{\SuppCrys\UCV_U^\bullet}$.
\item  If $U\cap\SuppCrys\UCV^\bullet=\emptyset$, then $\SuppCrys i^! \UCV_X^\bullet = \SuppCrys\UCV_X^\bullet$.
\item $\SuppCrys i_*\UCV_Z^\bullet=\SuppCrys\UCV_Z^\bullet$.
\end{enumerate}

\end{lem}
\begin{proof}
Part (a) is obvious, since the formation of $\image\kappa^e$ commutes with restriction to $U\times \CScheme$. Part (d) holds because the exact functor $i_*$ also commutes with the formation of $\image\kappa^e$. Next we prove~(c): Define $W:=\SuppCrys i^!\UCV^\bullet $. By (d), the complex $i_*i^!\UCV^\bullet $ has crystalline support $W$. By our hypothesis $j^*\UCV^\bullet=0$ and thus $j_*j^*\UCV^\bullet=0$ holds as well. It follows from the exact triangle in \autoref{ExactTriangleThm} that $\UCV^\bullet\cong i_*i^!\UCV^\bullet $ has crystalline support equal to $W$, and this proves~(c).

To prove (b), write $Z_U$ for $\SuppCrys\UCV_U^\bullet$ and $Z_X$ for the closure of $Z_U$ in $X$, and consider the diagram
\[ \xymatrix{
U \ar@{^{ (}->}[r]^j & X \\
Z_U \ar@{^{ (}->}[u]^{i_U} \ar@{^{ (}->}[r]^{\wt j} & Z_X \ar@{^{ (}->}[u]^{\wt i}  \\
 }\]
where the horizontal arrows are closed immersions and the vertical ones are open embeddings. From the exact triangle for $Z_U\into U \hookleftarrow U\setminus Z_U$ we obtain $\UCV_U^\bullet\cong i_{U*}i_U^!\UCV_U^\bullet$. Therefore
\[ j_* \UCV_U^\bullet =  j_*i_{U*} (i_U^!\UCV^\bullet) = \wt i_*( \wt j_{*} i_U^!\UCV^\bullet).\]
By part (d), the crystalline support of the right hand side lies in $Z_X$, which is~(b).
\end{proof}

\begin{cor}\label{CrysSupp.FibersAndSuppCrys}
For $\UCV^\bullet\in \D^b_\crys(\QCrysC(X,\Coeff))$ one has
\[ \{x\in X \mid i_x^! \UCV^\bullet\neq 0 \in \D^b_\crys(\QCrysC(x,\Coeff)) \} \subset
\SuppCrys \UCV^\bullet.\]
\end{cor}
\begin{proof}
Denote by $U$ the complement of $\SuppCrys \UCV^\bullet$ and by $j\colon U\into X$ the corresponding open immersion. By \autoref{CrysSupp.SuppAndFunctors}(a) we have $j^*\UCV^\bullet=0$. Since for any $x\in U$ we have $i_x$ factors through $j$, we deduce $i_x^!\UCV^\bullet=0$ for all such~$x$.
\end{proof}

\begin{lem}\label{CrysSupp.OnRegular}
Let $X$ be a regular connected scheme and $\UCV\in\CrysC(X,\Coeff)$ with $\SuppCrys\UCV=X$. Then there is a dense open subset $U\subset X$ such that for all $x\in U$ one has $R^{d_x}i_x^! \UCV\neq 0$ in $\CrysC(x,\Coeff)$ where $d_x$ is the dimension of~$x$.
\end{lem}
\begin{proof}

We follow the argument in \cite[Thm.~4.6.2]{BoPi.CohomCrys}. We may replace $\UCV$ by $\UCV/\UCV_{nil}$ and assume that the $\kappa$-sheaf $\UCV$ has no nilpotent Cartier subsheaves. This means, that the adjoint to the structural map $\wt{\kappa} \colon \CV \to \Frid^! \CV$ is injective. Let $T:=\Supp\kappa^e(\CV)$ for $e\gg0$. As recalled above \autoref{CrysSupp.IsClosed1}, the set $T$ is a finite union $T=\bigcup_i Y_i\times D_i$ for suitable closed subsets $Y_i\subset X$ and $D_i\subset C$. By passing to a dense open subset of $X$, we may assume that $Y_i=X$ for all $i$, so that $T=X\times D$ for some closed $D\subset C$.

The injectivity of $\wt{\kappa} \colon \CV\to \Frid^!\CV$ shows that for any generic point of $X\times D$, $\wt{\kappa}$ is an isomorphism, since $\Frid^!\CV$ and $\CV$ have the same length, \cf \cite[proof of Proposition 4.1]{BliBoe.CartierFiniteness}.

Let $Z\subset X\times D$ to be the support of $\coker\wt{\kappa}$. By construction it is a nowhere dense closed subset. Choose $D'$ a scheme of finite type over $\BF_q$ such that $D$ is a localization of $D'$ and let $Z'$ be the closure of $Z$ in $X\times D'$. Then $W':=(X\times D')\setminus Z'$ is an open dense subset of $X\times D'$. Let $V$ be the set of points $x\in X$ for which $W'\cap x\times D'$ is dense in $x\times D'$. As the projection $\pr_1\colon X\times D' \to X$ is a morphism of finite type, \cite[Prop.~9.5.3]{EGA4} shows that $V$ is constructible. By assumption it contains the generic point of $X$; hence $V$ contains an open dense subset $U\subset X$.

By construction, for every $x \in U$ the induced morphism $\wt{\kappa} \colon (i_x\times \id)^!\CV \to \Frid^!(i_x \times \id)^!\CV$ is also an isomorphism at the generic points of $x \times D$. Hence in particular, $R^{d_x}i^!\UCV$ is nonzero as a Cartier crystal, provided it is non-zero as a sheaf. Using the Koszul complex for a system of parameters at $x$ one computes that $R^{d_x}(i \times \id)^!\CV \cong \CV/\Fm_x\CV$ which is nonzero since $x \times D \subseteq  \Supp \CV$.
\end{proof}
\begin{rem}\label{CrysSupp.MaxDimofI!}
Note that in the situation of \autoref{CrysSupp.OnRegular}, we have $R^di_x^!\UCV=0$ for all $x\in X$ and $d>d_x$. This is so because the Koszul resolution has length exactly $d_x$ and thus already $R^di_x^!\CV$ vanishes for all $d>d_x$.
\end{rem}

\begin{cor}\label{CrysSupp.NonzeroFibersAreDense}
For $\UCV^\bullet\in \D^b_\crys(\QCrysC(X,\Coeff))$  one has
\[ \overline{\{x\in X \mid i_x^! \UCV^\bullet\neq 0 \in \D^b_\crys(\QCrysC(x,\Coeff)) \}} = \SuppCrys \UCV^\bullet.\]
If moreover $X$ is of finite type over a field $k$, then
\[ \overline{\{x\in X\hbox{ closed} \mid i_x^! \UCV^\bullet\neq 0 \in \D^b_\crys(\QCrysC(x,\Coeff)) \}} = \SuppCrys \UCV^\bullet.\]
\end{cor}
\begin{proof}
For any closed immersion $i\colon Z\to X$ and $x\in Z$ we have $i_x^!=i_x^!i^!$. Taking $Z=\SuppCrys\UCV^\bullet$ and applying \autoref{CrysSupp.SuppAndFunctors}(c), we may thus assume that $X=\SuppCrys\UCV^\bullet$. To prove the assertion we may now replace $X$ by any dense open subset. By generic regularity, we can therefore assume that $X$ is regular. By passing to a component, we may in addition assume that $X$ is connected.

To prove the first assertion, we choose $d$ minimal such that $\SuppCrys H^d(\UCV^\bullet)$ contains the generic point $\xi$ of $X$. Since $d_\xi=0$, by \autoref{CrysSupp.MaxDimofI!}  the spectral sequence of \autoref{RightDerivedOfFlat} degenerates to
\[ i_\xi^! (H^j\UCV^\bullet) \cong R^j i_\xi^! \UCV^\bullet. \]
Now by \autoref{CrysSupp.OnRegular}, we have $i_\xi^! (H^d\UCV^\bullet)\neq0$. This shows that $R^d i_\xi^! \UCV^\bullet$ is non-zero and thus establishes the first assertion.

For the second assertion, let $d$ be maximal such that $\SuppCrys H^d(\UCV^\bullet)$ contains $\xi$. By passing from $X$ to a dense open subscheme we may assume that $\SuppCrys H^{d'}(\UCV^\bullet)=0$ for all $d'>d$. Then for any closed point in $x$, the spectral sequence of \autoref{RightDerivedOfFlat} gives
\[ R^ji_x^! (H^k\UCV^\bullet) \To R^{j+k} i_x^! \UCV^\bullet. \]
Since $X$ is of finite type over $k$, we have $d_x=\dim X$ for any closed point. Using \autoref{CrysSupp.MaxDimofI!}, we deduce for $j+k=d_x+d$ the isomorphism
\begin{equation}\label{DegenSSforClosedPt}
 R^{d_x}i_x^! (H^d\UCV^\bullet) \cong R^{d+d_x} i_x^! \UCV^\bullet.
\end{equation}
We choose a dense open subset $U$ for $H^d(\UCV^\bullet)$ as in \autoref{CrysSupp.OnRegular}. Then for any closed point $x\in U$ the left hand side of \autoref{DegenSSforClosedPt} is non-zero and thus for any such $x$ we deduce that $i_x^! \UCV^\bullet$ is non-zero. But since $X$ is of finite type over $k$, the closed points in $U$ are dense in $X$ and this concludes the proof.
\end{proof}

\begin{rem}\label{ExampleOnSecondDefOfSupp}
By \autoref{Ex:NonclosednessOfSupport} the set $\{x\in X \mid i_x^! \UCV^\bullet\neq 0 \in \D^b_\crys(\QCrysC(x,\Coeff)) \}$ is not closed in general. However, we presently do not know if it is closed if on assumes that $X$ is of finite type.
\end{rem}